\documentclass[a4paper,11pt]{amsart}

\usepackage{color}
\usepackage{comment}
\usepackage{amssymb}
\usepackage{amsfonts}
\usepackage{amscd}
\usepackage{slashed}
\usepackage{pdflscape}
\usepackage{tikz}
\usepackage{enumerate}
\usepackage{multirow}
\usepackage{stmaryrd}

\usepackage{bbold}
\usepackage[backref,pagebackref,linkcolor=blue]{hyperref}
\usepackage{mathrsfs}
\newcommand{\bm}[1]{\mbox{\boldmath $ #1 $}}

\newtheorem{theorem}{Theorem}[section]
\newtheorem{lemma}[theorem]{Lemma}
  
\newtheorem{proposition}[theorem]{Proposition}
\newtheorem{corollary}[theorem]{Corollary}

\theoremstyle{definition}
\newtheorem{definition}[theorem]{Definition}

\theoremstyle{remark}
\newtheorem{remark}[theorem]{Remark}

\newtheorem{problem}[theorem]{Problem}

\usepackage{graphicx} 
\usepackage{amsmath} 
\usepackage{amsfonts}
\usepackage{amssymb}

\newcommand{\be}{\begin{equation}}

\newcommand{\ee}{\end{equation}}

\newcommand{\g}{g^o}

\newcommand{\cG}{{\mathcal G}}
\newcommand{\cU}{\mathcal{U}}
\newcommand{\G}{\Gamma}
\renewcommand{\H}{B}
\newcommand{\D}{\Delta}

\newcommand{\al}{\mbox{\boldmath$\Delta$}}

\newcommand{\aNd}{\mbox{\boldmath$ \nabla$}\hspace{-.6mm}}

\newcommand{\si}{\sigma}

\newcommand{\ts}{\textstyle}

\newcommand{\ba}{\begin{array}}

\newcommand{\ea}{\end{array}}

\newcommand{\beq}{\begin{eqnarray}}

\newcommand{\eeq}{\end{eqnarray}}

\newtheorem{lm}{lemma}

\newtheorem{thee}{theorem}

\newtheorem{proo}{proposition}

\newtheorem{co}{corollary}

\newtheorem{rem}{remark}

\newtheorem{deff}{definition}

\newcommand{\bd}{\begin{deff}}

\newcommand{\ed}{\end{deff}}

\newcommand{\bl}{\begin{lm}}

\newcommand{\el}{\end{lm}}

\newcommand{\bp}{\begin{proo}}

\newcommand{\ep}{\end{proo}}

\newcommand{\bt}{\begin{thee}}

\newcommand{\et}{\end{thee}}

\newcommand{\bc}{\begin{co}}

\newcommand{\ec}{\end{co}}

\newcommand{\brm}{\begin{rem}}

\newcommand{\erm}{\end{rem}}

\newcommand{\f}{\overline{f}}
\newcommand{\F}{\overline{F}}

\usepackage{t1enc}
\def\frak{\mathfrak}

\def\Bbb{\mathbb}
\def\Cal{\mathcal}

\newcommand{\newc}{\newcommand}

\let\ccdot\cdot
\def\cdot{\hbox to 2.5pt{\hss$\ccdot$\hss}}

\newc{\aR}{\mbox{\boldmath{$ R$}}}
\newc{\aS}{\mbox{\boldmath{$ S$}}}
\newc{\aT}{\mbox{\boldmath{$ T$}}}
\newc{\aW}{\mbox{\boldmath{$ W$}}}
\newcommand{\aX}{\mbox{\boldmath{$ X$}}\hspace{-.2mm}}
\newcommand{\aI}{\mbox{\boldmath{$ I$}}}
\newcommand{\eaI}{\bm{\mathcal I}}
\newcommand{\iaI}{\bm{\mathcal I}^*}
\newcommand{\eaX}{\bm{\mathcal X}}

\newc{\aD}{\mbox{\boldmath{$ D$}}\hspace{-.2mm}}

\renewcommand{\colon}{\scalebox{1.2}{:}}

\newcommand\oh{{\overline h}}

\newc{\aK}{\mbox{\boldmath{$ K$}}}
\newc{\aL}{\mbox{\boldmath{$ L$}}}
\newcommand{\ca}{{\Cal A}}

\newcommand{\ce}{{\Cal E}}

\newcommand{\cO}{{\Cal O}}

\newcommand{\cf}{{\Cal F}}

\newcommand{\ct}{{\Cal T}}

\usepackage{amssymb}
\usepackage{amscd}

\newcommand{\nd}{\nabla}

\newcommand{\Rho}{P}
\newcommand{\Up}{\Upsilon}

\newcommand{\End}{\operatorname{End}}
\newcommand{\im}{\operatorname{im}}
\newcommand{\Ric}{\operatorname{Ric}}

\newcommand{\I}{\bs{I}}

\newcommand{\cT}{{\mathcal T}}

\let\hash=\sharp  %%added 14June2004

\let\i=\iota

\newcommand{\aM}{\, \widetilde{\!M}\, }
\newcommand{\asi}{\tilde{\sigma}}

\newcommand{\cW}{{\Cal W}}

\newcommand{\nn}[1]{(\ref{#1})}

\newcommand{\X}{\mbox{\boldmath{$ X$}}}
\newcommand{\sX}{{\mbox{\scriptsize\boldmath{$X$}}}} 
\newcommand{\sI}{{\mbox{\scriptsize\boldmath{$I$}}}}        % scriptsize ambient 
\newcommand{\aF}{\boldsymbol{F}}

\newcommand{\h}{\mbox{\boldmath{$ h$}}}
\newcommand{\bg}{\mbox{\boldmath{$ g$}}}

\renewcommand{\S}{\Sigma}

\let\G=\Gamma

\newcommand{\J}{{\mbox{\sf J}}}

\newcommand{\trH}{{\mbox{\sf H}}}
\newc{\obstrn}[2]{B^{#1}_{#2}}

\newcommand{\rpl}                         % +) or <+
{\mbox{$
\begin{picture}(12.7,8)(-.5,-1)
\put(0,0.2){$+$}
\put(4.2,2.8){\oval(8,8)[r]}
\end{picture}$}}

\newcommand{\lpl}                         % (+ or +>
{\mbox{$
\begin{picture}(12.7,8)(-.5,-1)
\put(2,0.2){$+$}
\put(6.2,2.8){\oval(8,8)[l]}
\end{picture}$}}

\usepackage{ifthen}

\newc{\tensor}[1]{#1}
\newc{\Mvariable}[1]{\mbox{#1}}
\newc{\down}[1]{{}_{#1}}
\newc{\up}[1]{{}^{#1}}

\newc{\JulyStrut}{\rule{0mm}{6mm}}
\newc{\midtenPan}{\mbox{\sf S}}
\newc{\midten}{\mbox{\sf T}}
\newc{\midtenEi}{\mbox{\sf U}}
\newc{\ATen}{\mbox{\sf E}}
\newc{\BTen}{\mbox{\sf F}}
\newc{\CTen}{\mbox{\sf G}}

\renewcommand{\P}{\mbox{$\Bbb P$}}
\newcommand{\w}{\mbox{\bf w}} 

\def\sideremark#1{\ifvmode\leavevmode\fi\vadjust{\vbox to0pt{\vss% the remark
 \hbox to 0pt{\hskip\hsize\hskip1em%                          will appear only
 \vbox{\hsize3cm\tiny\raggedright\pretolerance10000%          on the side
 \noindent #1\hfill}\hss}\vbox to8pt{\vfil}\vss}}}%
\newcommand{\edz}[1]{\sideremark{#1}}

\numberwithin{equation}{section}

\newcommand\degree{{\bm N}}
\newcommand\extd{{\bm d}}
\newcommand\extdS{{\bm d}_{_\Sigma}}
\newcommand\cod{{\bm \delta }}
\newcommand\codS{{\bm \delta }_{_\Sigma}}
\newcommand\FL{{\bm \Delta }}

\newcommand\N{{\mathscr N}}
\renewcommand\I{{\mathscr I}}
\newcommand\Is{{\mathscr I}^\star}
\renewcommand\D{{\mathscr D}}
\newcommand\Ds{{\mathscr D}^\star}
\renewcommand\X{{\mathscr X}}
\newcommand\Xs{{\mathscr X}^\star}
\newcommand{\Y}{{\mathscr Y}}
\newcommand{\Ys}{{\mathscr Y}^\star}
\renewcommand\F{{\mathcal F}}

\newcommand\XsS{\X^{\phantom{\star}}_{_\Sigma}}

\renewcommand\colon{\scalebox{1.3}{$:$}}
\newcommand{\wDs}{\widehat{\D}^\star}
\newcommand{\wD}{\widehat{\D}}

\newcommand{\Dt}{\widetilde \D}
\newcommand{\Dts}{{\widetilde \D}^\star}
\newcommand{\coker}{\operatorname{coker}}

\newcommand{\oD}{\overline{\D}}
\newcommand{\oDs}{\overline{\D}^{\, \star}}

\newcommand{\wiota}{\tilde\iota}
\newcommand{\wepsilon}{\tilde\varepsilon}

\newcommand{\Integer}{{\mathbb Z}}

\newcommand{\A}{\mathcal A}
\newcommand{\B}{\mathcal B}

\begin{document}

\renewcommand{\today}{}
\title{Poincar\'e--Einstein Holography for Forms via  Conformal Geometry in the Bulk }
\author{A. Rod Gover${}^\diamondsuit$, Emanuele Latini${}^\clubsuit$ \& Andrew Waldron${}^\spadesuit$}

\address{${}^\diamondsuit$Department of Mathematics\\
  The University of Auckland\\
  Private Bag 92019\\
  Auckland 1\\
  New Zealand,  and 
  Mathematical Sciences Institute, Australian National University, ACT
  0200, Australia} \email{gover@math.auckland.ac.nz}
  
  \address{${}^{\clubsuit}$Institut f{\"u}r Mathematik, Universit{\"a}t Z{\"u}rich-Irchel, Winterthurerstrasse 190, CH-8057 Z{\"u}rich, 
  Switzerland, and INFN, Laboratori Nazionali di Frascati, CP 13,
  I-00044 Frascati, Italy} \email{emanuele.latini@math.uzh.ch}
  
  \address{${}^{\spadesuit}$Department of Mathematics\\
  University of California\\
  Davis, CA95616, USA} \email{wally@math.ucdavis.edu}

\vspace{10pt}

\vspace{10pt}

\renewcommand{\arraystretch}{1}

\begin{abstract} 

We study higher form Proca equations on Einstein manifolds with boundary data along conformal infinity. 
We solve these Laplace-type boundary problems formally, and to all orders, by constructing  an operator which projects arbitrary forms to
solutions. We also develop a product formula  for solving these asymptotic problems in general.
The central tools of our approach are (i)  the conformal geometry of differential forms and the associated exterior tractor calculus, and (ii) a generalised notion
of scale which encodes the connection between the underlying geometry  and its boundary. The latter also controls the breaking of conformal invariance in a very strict way by coupling conformally invariant equations to the scale tractor associated with the generalised scale.
From this, we obtain a map from existing solutions to new ones that exchanges
Dirichlet  and Neumann boundary conditions.
Together, the  scale tractor and exterior structure extend the solution generating algebra of~\cite{GWasym} to a conformally invariant, Poincar\'e--Einstein calculus on (tractor) differential forms.
This calculus leads to explicit holographic formul\ae\  for  all the higher order conformal operators on weighted differential forms, differential complexes, and Q-operators of~\cite{BrGodeRham}. This complements the results of Aubry and Guillarmou~\cite{AG-BGops} where associated conformal harmonic spaces parametrise smooth solutions.

\vspace{2cm}
\noindent
%\begin{keywords}

\end{abstract}

\maketitle

\pagestyle{myheadings} \markboth{Gover, Latini \& Waldron}{Holography for Forms}

\newpage

\tableofcontents

\section{Introduction}\label{intro}

The Poincar\'e model realises hyperbolic $(n+1)$-space
$\mathbb{H}^{n+1}$ as the interior of a unit Euclidean ball, but
equipped with a metric conformally related to the Euclidean metric in
a way that places the boundary $n$-sphere $S^n$ at infinity.  This
provides a concrete setting for identifying the isometry group of
$\mathbb{H}^{n+1}$ with the conformal group of $S^n$ and so
a geometric foundation for Poisson transforms linking 
representations of $G=SO(n+1,1)$, as induced by its maximal parabolic,
to those induced by its maximal compact subgroup~\cite{Helgason}.
We develop here new tools for this programme, however our main focus is 
its  curved generalisation and our constructions are based in this setting.
Such curved analogues of this ``flat model'' underlie
striking new developments in mathematics and physics. 

Let $M$ be a
$d:=n+1$-dimensional, compact manifold with boundary $\Sigma=\partial M$
(all geometric structures assumed smooth).  A pseudo-Riemannian metric $\g$ on
the interior~$M^+$ of~$M$ is said to be {\it conformally compact} if it
extends to $\Sigma$ by $$g=r^2\g\, ,$$ where~$g$ is non-degenerate up to the
boundary, and $r$ is a defining function for the boundary ({\it
  i.e.}\ $\Sigma$ is the zero locus of $r$, and $d r$ is non-vanishing
along~$\Sigma$). Assuming that $\Sigma$ does not contain null
directions, the restriction of $g$ to $T\Sigma$ in $TM|_\Sigma$
determines a conformal structure, and this is independent of the
choice of defining function $r$; $\S$~with this conformal structure is
termed the {\em conformal infinity} of $M^+$. In the case of Riemannian
signature, if in addition $$|dr|^2_g=1$$ along $\Sigma$, then sectional
curvatures approach $-1$ asymptotically and so the structure is said
to be {\em asymptotically hyperbolic (AH)}~\cite{Ma-hodge}. On the other hand, if the
interior is negative Einstein (without loss of generality the Ricci
curvature satisfying $\Ric^{\g}=-ng^o$), then the structure is said to
be {\em Poincar\'e-Einstein (PE)}; PE~restricts the asymptotic
sectional curvatures and in the Riemannian case implies~AH.

These structures provide a framework for
relating conformal geometry and associated field
theories of the boundary to the far field phenomena of the interior
(pseudo-)Riem\-annian geometry of one higher dimension; the latter
often termed the {\em bulk}. Such problems may be viewed as 
analogues of the Poisson transform on the hyperbolic model
structure. For the construction of local conformal invariants, a
highly influential approach to PE structures was developed by
Fefferman-Graham in~\cite{FGast}; this was partly inspired by related
ideas from physics and general relativity~\cite{LeBrunH}. On the
global side critical aspects of the spectral theory for (Riemannian
signature) conformally compact manifolds was developed by Mazzeo and
Mazzeo-Melrose~\cite{Ma-hodge,Ma-unique,MaMe}. The local and global
directions were brought together in~\cite{GrZ} which develops a
scattering matrix approach to PE structures.

In physics, there is a general notion of  holography which strives  to capture field/string theories and their geometries
in terms of corresponding structures on a space of lower dimension~\cite{tHooft,Susskind}; a hologram provides a visual analogue
of this principle.  Much of the mathematics of conformally compact
structures may be viewed as a concrete realisation of this idea, and
indeed these geometries are used as prototypes for holography and
related renormalization ideas. In this Article, we shall use the term holography in reference to 
 this programme in the setting of Poincar\'e-Einstein
manifolds. Aside from the mathematics revealed, we view our work as
contributing to the investigation of Maldacena's conjectural AdS/CFT
correspondence within quantum theory. This
proposes to relate string theory on
the bulk to a boundary conformal field theory, and has been the
stimulus for much of the recent intense interest in the described
directions~\cite{Mal,AdSCFTreview,Henningson,Juhl}.

In a nutshell the current work is concerned with developing a
comprehensive holographic treatment of differential forms, and
conformally weighted differential forms; 
we study both bulk field equations  and  conformally invariant differential operators
intrinsic to the boundary.
Our approach uses heavily
some new ideas surrounding a generalised notion of scale introduced in
~\cite{GoPrague,GoIP,Goal}, that is closely linked to conformal tractor
calculus~\cite{BEG,CapGoamb,GoPetCMP}. Building on~\cite{GWasym} we
develop the new theory  and  calculus required to 
apply this to differential form problems.

The results we obtain draw together three recent
developments.  In~\cite{BrGodeRham} it is shown that differential
forms host a particularly rich conformal theory. New conformal (detour)
complexes and related gauge fixing operators found
there lead to corresponding conformal cohomologies, a notion of
conformal harmonics, as well as  invariants and invariant operators that generalise Branson's
Q-curvature~\cite{tomsharp}. In a study of harmonic~$k$-forms on PE manifolds,  Aubry and Guillarmou found a 
``holographic'' meaning of these objects~\cite{AG-BGops}; for example, each
space of conformal harmonics as defined in~\cite{BrGodeRham}, for the conformal infinity, is seen to
fit into an exact sequence involving the spaces of harmonics and~$L_2$-harmonics of the bulk.
On the other hand, in what might at first appear to
be an unrelated development, it was found in~\cite{Gover:2008pt,Grigoriev:2011gp} that massive,
massless and partially massive particle theories~\cite{Deser:1983tm,Deser:2001us} are described
simultaneously via a (tractor calculus mediated)-interaction of scale
with conformal differential operators (this also underlies the recent ``shadow field'' approach of~\cite{Metsaev}). In fact, by design, the tools
developed in~\cite{Gover:2008pt} are entirely compatible with the boundary
calculus of~\cite{GWasym}. An extension of that machinery is used here to formulate
a class of Proca systems for higher forms. We show  that these are
compatible with the forms problem treated in~\cite{AG-BGops}. 
Here we solve formally  Dirichlet and Neumann Proca problems, see Problem~\ref{laylanguage} and~\ref{solprob2}. We obtain explicit formul\ae\ for complete solutions:
in the massless case, we recover the conformal harmonic space from~\cite{BrGodeRham} as a condition
for compatibility with smoothness of solutions, consistent with results of~\cite{AG-BGops}. An important feature of this approach is that it provides explicit holographic formul\ae \ for the detour and gauge operators defining these spaces. Here and throughout the Article, a {\it holographic formula} for an operator on the conformal infinity  is a canonically obtained bulk operator which recovers the given boundary operator  by restriction.

The scattering programme of~\cite{GrZ} and others~\cite{GuilN,Joshi,HPT,MSV,Vasy} surrounds
natural boundary problems on conformally compact manifolds, see also the related Dirichlet problems~\cite{Anderson,Sullivan}; this programme is particularly
rich for  Poincar\'e-Einstein structures.
In~\cite{GWasym} it is shown quite generally that the asymptotics of such boundary
problems can be solved to all orders using an algebra of geometric
operators; there the problems treated  include not only scalars but general
twistings of such by tractor bundles; this laid  the universal route to handling general tensor problems. A key point of that approach is
that it is not only algebraically efficient, but it is also geometrically
conceptual. A conformally compact manifold is the same as a (compact)
conformal manifold with boundary equipped with a certain density
(which should be interpreted as a generalised scale, termed later a defining scale) the zero locus of
which defines the boundary. The point in~\cite{GWasym} is that this
structure {\em canonically determines} an ${\frak sl}(2)$-generating
triple of differential operators which not only yield the natural
boundary problems (including those previously studied in the
literature), but also a solution generating algebra which solves these.  Here we develop a calculus of scale for differential forms which extends this idea in a way that solves higher form Proca systems.

Consider broadly the ``holographic problem'' for differential forms. Since
the bulk has higher dimension than the bounding conformal infinity, we
expect that freely specified boundary forms 
correspond to differential forms on the interior satisfying
constraining equations. We investigate this idea on PE manifolds, and
where the interior equation is the higher form Proca system
\begin{equation}\label{proca1}
\cod\extd A - m^2 A = 0, \qquad \cod A=0,
\end{equation}
for a (differential) $k$-form $A$. Here $\extd$ is the exterior
derivative, $\cod$ the negative of its formal adjoint, and the
parameter $m^2$ is a constant over the manifold. Even though it will
be seen that $m^2$ is quadratic in a natural spectral parameter, it
can be negative; the historical notation is used because  $m^2$
recovers the square of the rest mass in certain settings.  Note also that 
$\cod\extd A - m^2 A = 0 $ implies the divergence/transversality condition $\cod A=0$, unless $m^2=0$ in which case 
$\cod A=0$ is known as the Feynman gauge fixing relation.  
It implies that $\FL A= m^2 A$, which for
$m^2=\big(\frac n2 -k+\ell\big)\big(\frac n2 -k-\ell\big)$ is the
 interior equation considered in~\cite{AG-BGops}  (for suitable integers $\ell$).

Let us summarise our main results: 
\begin{itemize}

\item We show how to write the Proca system \nn{proca1} using the
  tractor calculus machinery. In fact we obtain much more: We show
  that the Proca system \nn{proca1} arises canonically from the PE
  structure, and the mass term is determined (quadratically) by the conformal 
weight. It is given by an obvious coupling of conformally
  invariant (or in physics terminology---Weyl invariant) tractor
  equations to the defining scale tractor, see Proposition~\ref{westproca}.

\vspace{2mm}
\item In contrast to \nn{proca1}, the tractor system of Proposition
 ~\ref{westproca} is well-defined up to the boundary and so it is seen, 
 via this use of conformal geometry and the tractor
  interpretation of a PE structure, that the interior system \nn{proca1}
  canonically determines two compatible boundary problems and these
  are given in Problem~\ref{laylanguage} and the expression~\nn{laidagain}. Solutions of these problems are called,
  respectively, (Proca) solutions of the {\em first type}, and (Proca)
  solutions of the {\em second type}.  

\vspace{2mm}
\item Following~\cite[Proposition 5.10]{GWasym}, we show that any solution of the
  Proca system~\ref{proca1} may be coupled with the defining scale of
  the PE structure to yield another solution, see {\it e.g.}\ Theorem
 ~\ref{swapg}; at the level of formal solutions this maps solutions of
  the first type to solutions of the second type, see Theorem~\ref{swap}
  and Corollary~\ref{swapc}.  For solutions to the global boundary
  problem this {\em scale duality map} gives a new solution of the
  same interior equation with the {\it r\^oles} of  the Dirichlet and Neumann data
  exchanged, see Remark~\ref{drem}.
 
\vspace{2mm}
\item The Proca Boundary  Problem~\ref{laylanguage} is solved formally to all orders in Theorem
 ~\ref{BIGTHEOREM}. See also Proposition~\ref{fsol} in Section
 ~\ref{prodf} which shows that even when written directly in terms of
  (weighted) differential form boundary data, remarkably simple
  explicit formul\ae \  are available for the solution.  The solution to
  Problem~\ref{solprob2} then follows by scale duality, and this is
  the content of Corollary~\ref{swapc}. Other cases require log terms
  in the solutions. Within this context Problem~\ref{logfire} is
  generic and is solved in Theorem~\ref{mainl}. The remaining
  exceptional weights are the important cases of true forms and their
  weight duals. These are the subject of Problem~\ref{trueproblem} and
  Problem~\ref{solprobdualtrue} and are solved in, respectively,
  Theorem~\ref{tpsol} and Theorem~\ref{tpdsol}.

\vspace{2mm}
 \item In Section~\ref{products} we give a product form for these solutions expressed as a certain product of
 second order differential operators that projects arbitrary bulk forms to solutions. This provides an
 alternative solution to the tractor extension problem of~\cite{GWasym} 
 and was inspired by the Fefferman--Hirachi product solution~\cite{Fefferman-Hirachi} for an ambient, scalar, Goursat-type problem.

\vspace{2mm}
 \item In Section~\ref{funSec} we obtain holographic constructions
   and formul\ae\  for the higher order, differential, {\it Branson--Gover (BG)  operators} on forms
   from~\cite{BrGodeRham}. These yield
   holographic formul\ae\ for the natural conformally invariant boundary operators~$L^{\ell}_k$.
   In particular we
    obtain the  
   detour operators
   $L_k$,  and we also find holographic formul\ae\ for  gauge fixing
   operators $G_k$, the Q-operators $Q_k$, the factorisations
$$
L_k= \cod Q_{k+1} \extd \, ,
$$
and hence differential detour complexes (see Theorem~\ref{detids}).
  These are seen to arise both from tangential operators 
  along the boundary~$\Sigma$ as well as the obstruction to smooth solutions to the Proca systems
  for true forms. The former uses surprising holographic identities for powers of the Laplace--Robin operator (see Theorem~\ref{fun}).
\end{itemize}

The technique for solving all these problems is to construct a universal operator that projects arbitrary bulk forms to (formal) solutions.
This is a composition of a projector which solves the problem of inserting forms in tractors (see Section~\ref{holographicproj}) and a modification of the solution generating  operator on 
tractors constructed in~\cite{GWasym} which acts as a projector onto formal power series solutions. 
This extends the idea of {\em curved translation}, initiated in~\cite{Eastwood-Rice}, which constructs new invariant differential operators from existing ones. Here we use a similar
idea to obtain solutions of the Proca problem from solutions of a related universal tractor problem. 
Concretely this yields  the following remarkably simple expression for solutions to the Proca system~\nn{proca1}
$$
A=q^* \colon K \colon \, q_W A_0\, ;
$$
see Theorem~\ref{fruit}. 
(This formula holds 
 avoiding one family of distinguished weights indexed by form degree and requiring a separate treatment.)
 In the above,  the differential form  $A_0$ is an arbitrary smooth extension to the interior of Dirichlet boundary data along~$\Sigma$
 and the operator $q_W$ is a differential splitting operator mapping forms to tractors along the lines of~\cite{BrGodeRham} and developed in detail in Section~\ref{compinsert}; the bundle map $q^*$ is a left inverse for this.
The operator $\colon K \colon$ is a specially constructed variant of the solution generating operator of~\cite{GWasym} adapted to forms and tuned to  the Proca system\footnote{It has recently come to our attention  that, in the context of completions of the enveloping algebra $\cU(\frak{sl}(2))$, the operator $\colon K \colon$ is an extremal projector~\cite{Lowdin,Tolstoy}. These are projections $P$ such that $P e_\alpha=0=e_{-\alpha}P$ for all positive roots $\alpha$ labeling Lie algebra generators $e_\bullet$.}.

It is central to our approach to construct and solve the tractor version of the Proca system.
Therefore we first develop a direct approach to  exterior tractor calculus on general conformal structures. 
 This avoids using the Fefferman--Graham ambient metric, yields explicit formul\ae, and recovers and extends the identities found in~\cite{BrGodeRham}. 
 Also, a  complete algebra of differential splitting operators is obtained
using this exterior tractor calculus (see Sections~\ref{algebra} and~\ref{compinsert}).
In Section~\ref{coho} classical conformally invariant equations on forms are surveyed and their origins from the cohomology of the nilpotent 
exterior Thomas D-operator are explained.

To describe the Proca system, it essential to couple tractor forms to the PE defining scale.
This draws additional canonical operators into the tractor exterior calculus
and determines natural boundary conditions for a canonical and universal class of extension problems that we call the {\it tractor Proca equations}:
\begin{itemize}
\item The solution generating algebra of~\cite{GWasym} for  the forms analog of the Laplace--Robin operator $I\cdot D$ of~\cite{powerslap,GoIP}
is extended to include exterior and interior multiplication by the scale tractor~$I$. 
This is the basis for an exterior calculus of scale described in Section~\ref{extcalcsc}, and along the boundary  gives an extrinsic and conformally invariant  Robin-type operator~$\delta_R$ on forms.
\item On PE structures, the Laplace--Robin operator has  natural ``square roots''
$$
\Is \D + \D \Is = I\cdot \slashed D = \I \, \Ds + \Ds \I\, .
$$ 
This supersymmetry  is described in Section~\ref{solnalgform} and plays a critical {\it r\^ole}  in the  tractor description of the Proca system of Section~\ref{Procasect}.
\item 
We develop tangential versions of bulk tractor operators that allow us to follow the tractor Proca equations to the boundary 
in order to capture the required boundary conditions, see Sections~\ref{tractorsontheedge} and~\ref{bctractors}.
\item Solutions to the tractor Proca system are developed in Section~\ref{formssol}. This relies on a holographic formula $\Pi$ for an operator that commutes with the solution generating operator $\colon K\colon$ and extends to the interior  a projector onto boundary
tractor sections with image isomorphic to boundary forms. Moreover, the operator $\Pi$ ensures the tractor analog of the Proca transversality condition $\delta A=0$ and is  compatible  with  
 the boundary conditions. The composition of $\Pi$ and $\colon K\colon$  projects arbitrary tractor forms $\A_0$ onto  formal  solutions to the generic tractor Proca problem:
$$
\A=\Pi \, \colon K\colon\,  
\A_0\, ,
$$
see Theorem~\ref{BIGTHEOREM}.
\end{itemize}

\vspace{1cm}

\thanks{A.W. thanks Eugene Skvortsov for a very useful discussion of the extremal projector method. A.G.\ gratefully acknowledges support from the Royal Society of
  New Zealand via Marsden Grant 10-UOA-113; E.L. and A.W. wish to thank
  the University of Auckland for hospitality during the preparation
  of this work.}

\newpage

\section{Bulk conformal geometry and extension problems}\label{EXTENSION}
 
 \newcommand{\Real}{{\mathbb R}}
 \newcommand{\Complex}{{\mathbb C}}

 Here we lay out our conventions and some basic facts about conformal geometry (see Sections~\ref{landofRiemann} 
 and~\ref{CGeom}) before stating in Section~\ref{extprob} the general formulation of the extension problems solved in this Article.
 In Section~\ref{coulomb}, we solve a generalised divergence  extension problem which is required later to handle the transversality condition of the Proca system.
To solve the full Proca problem, tractor calculus plays an central {\it r\^ole}; the basic ingredients are stated in Section~\ref{TRACTORS}.
In Section~\ref{LRs} we review the calculus of defining scales and 
construct the Laplace--Robin operator which controls the dynamics 
of the systems we solve. In Section~\ref{LRalg} we  give the solution generating algebra of~\cite{GWasym}.  
This part of the Article is completed in Section~\ref{products} by giving a novel product form for the solution generating operator of~\cite{GWasym}.
For details and background on conformal geometry relied on here see~\cite{CapGoamb,GoPetCMP}.

\subsection{Riemannian conventions}\label{landofRiemann}
 
The main background and notations in this Article are mostly those of~\cite{GWasym}. Here we briefly highlight some main points.
Throughout we focus on manifolds $M$ of dimension~$d:=n+1$ at least three equipped with a metric, or a conformal equivalence class of Riemannian metrics.
 All structures will be assumed smooth.
For a given metric with Levi-Civita connection~$\nabla$, the Riemann
curvature tensor $R$ is given by
$$
R(X,Y)Z=\nd_{X}\nd_{Y}Z-\nd_{Y}\nd_{X}Z-\nd_{[X,Y]}Z,
$$
where $X$, $Y$, and $Z$ are arbitrary vector fields. In an 
index notation  $R$ is denoted by $R_{ab}{}^{c}{}_d$,
and $R(X,Y)Z$ is $X^aY^bZ^d R_{ab}{}^{c}{}_d$.
This can be decomposed into the totally trace-free {\em Weyl curvature}
$C_{abcd}$ and the symmetric {\em
Schouten tensor} $\Rho_{ab}$ according to
\begin{equation}\label{Rsplit}
R_{abcd}=C_{abcd}+2g_{c[a}\Rho_{b]d}+2g_{d[b}\Rho_{a]c},
\end{equation}
where $[\cdots]$ indicates  antisymmetrisation over the enclosed
indices. 
Thus $\Rho_{ab}$ is a trace modification of the Ricci tensor 
${\rm Ric}_{ab}=R_{ca}{}^c{}_b$:
$$
\Ric_{ab}=(n-2)\Rho_{ab}+ \J g_{ab}, \quad \quad \J:=\Rho^a_{~a}.
$$

\subsection{Conformal and almost Riemannian geometry}\label{CGeom}
 
Recall that  a conformal geometry $(M,c)$ is a $d$-manifold~$M$
equipped with an equivalence class of metrics~$c$ such that $g,\widehat g\in c$ obey $$\widehat g=\Omega^2\, g\, ,\qquad \Omega:=e^{\omega}\, ,$$ for some $\omega\in
C^{\infty}(M)$. Observe that a conformal structure can be viewed as a smooth ray subbundle
$\cG\subset S^2T^*M$ whose fibre at a point $x\in M$ consists of the values of $g_x$ for all metrics $g\in c$. The principal bundle~$\cG$ has structure group $\Real_+$ whose irreducible representations $\Real_+\ni t\mapsto t^{-w/2}\in \mbox{End}(\Complex)$, are labeled by weights~$w\in\Complex$. 
These induce ``conformal density'' line bundles~$\ce M[w]$. A section~$\tau\in \Gamma \ce M[w]$ is equivalent to a smooth function $\underline\tau$ of~$\cG$ that satisfies the  homogeneity property
$$
\underline\tau(x,\Omega^2 g)=\Omega^w \underline\tau(x,g)\, .
$$
%For brevity of notation, for weighted bundles $\ce^\bullet M[\, .\, ]$, we often write $\ce^\bullet [\, . \, ]$ unless confusion over the underlying base manifold~$M$ may arise.

Each metric $g\in c$ determines a canonical, positive, section $\tau\in
\Gamma\ce_+ M[1]$, {\it viz.}\ the section with the property that
$\underline{\tau}(p,g)=1$ for all $p\in M$. (The conformal density bundles are oriented, and a subscript plus indicates the $\Real_+$-ray subbundle.)
It follows that there is a
tautological section of $S^2T^*\!M\otimes \ce M[2]$ that is termed the
\hypertarget{conformal metric}{{\it conformal metric}}, denoted~$\bg$ with the property that any
nowhere zero section $\tau\in \Gamma \ce M[1]$ determines a  metric $g\in c$ 
via 
$g:=\tau^{-2}\bg$.
Henceforth the conformal metric~$\bg$ is the
default object that will be used to identify $TM$ with $T^*\!M[2]\cong T^*M\otimes \ce M[2]$ (rather
than a metric from the conformal class) and to form metric traces.

A Riemannian manifold can be treated as a conformal manifold $(M,c)$ equipped with a nowhere zero density $\sigma\in\Gamma\ce M[1]$, 
since the metric is recovered by $\sigma^{-2}\bg$. Following~\cite{Goal}, we formulate an obvious generalised notion of this.
\begin{definition}
A  {\em defining scale} is a  section $\si\in\Gamma\ce M[1]$ that is nowhere vanishing on an open dense subset of~$M$. 
A conformal manifold $(M,c)$ equipped with a defining scale  $\sigma\in\Gamma \ce M[1]$ is called an {\em almost Riemannian
structure} $(M,c,\sigma)$; where $\sigma$ is non-zero, $\sigma^{-2}\bg$ defines a Riemannian metric.

In the case that
$\si$ is {\em nowhere} zero this is a {\em true scale}, so that $\si^{-2}\bg$
is a metric  everywhere. 
\end{definition}
This notion of almost Riemannian structures arises naturally in the context of defining densities:
Consider a smooth, oriented hypersurface~$\Sigma$ given as the zero locus of some smooth (at least in a neighbourhood of~$\Sigma$) {\it defining function} $s$, with $\extd s\neq 0$ along $\Sigma$. Generally, to optimally employ the \hypertarget{conformal structure} conformal structure~$c$,
we will replace the defining function $s$ by a {\it defining density}~$\sigma$ which is
%The defining density~$\sigma$ is 
the unique conformal density $\sigma\in \Gamma\ce M[1]$ that yields the defining function $s$ in the trivialisation of~$\ce M[1]$
determined by some~$g\in c$. A~preferred defining density is a special example of a defining scale.

This notion is especially natural and useful for conformally compact manifolds.  A~conformally compact manifold is an almost Riemmanian manifold with boundary such that the  zero locus ${\mathcal Z}(\sigma)=\partial M$. Thus, in the notation of the Introduction  $g^o=r^{-2}g=\sigma^{-2}\bg$ and~$r$ is the component function representing $\sigma$ in the scale $g$. In particular
this applies to special case of PE structures where $g^o$ is negative Einstein.

Since each $g\in c$ determines a trivialisation of $\ce M[w]$, it  also defines a corresponding {\it Levi-Civita connection} $\nabla$ (see {\it e.g.}~\cite{CapGoamb}).
Moreover, a metric $g\in c$ canonically determines a true scale~$\tau\in\Gamma\ce M[1]$  by the requirement $\underline\tau(x,g)=1$. We will write
$\nabla^\tau$ for the connection corresponding to this scale. 
Almost Riemannian structures $(M,c,\sigma)$ come equipped with the canonical Levi--Civita connection $\nabla^\sigma:=\nabla^o$ {\it away} from the zero locus ${\mathcal Z}(\sigma)$ of $\sigma$.
In the scale $\tau$, acting on a density $\mu\in \Gamma\ce M[w]$, on $M\setminus {\mathcal Z}(\sigma)$ 
$$
\nabla^o \mu = \big(\nabla^\tau-w\sigma^{-1}n\big)\mu\, ,\qquad n := \nabla^\tau \sigma\, .  
$$
The \hypertarget{doublerole}{operator}
$$
\overline \nabla :=\sigma \nabla^\tau-w n     
$$
extends $\sigma \nabla^o$ to ${\mathcal Z}(\sigma)$. 
In the case where ${\mathcal Z}(\sigma)=\Sigma$ for some hypersurface $\Sigma$, this reduces to $w n_{|_\Sigma} \mu_{|_\Sigma}$
along $\Sigma$. Thus, the Levi--Civita connection off $\Sigma$ and (when $w\neq0$) the normal to the hypersurface are smoothly and canonically 
 incorporated   in a single operator. Moreover, the Levi--Civita connection has the following important yet obvious property.
 \begin{proposition}\label{overnabla}
 Let $\mu\in\Gamma \ce^\bullet M[w]$, where $\ce^\bullet M[w]$  indicates any weight $w$ tensor bundle. Then
 $$
 \overline \nabla \sigma \mu = \sigma \overline \nabla \mu\, .
 $$
 \end{proposition}

\subsection{Extension problems}\label{extprob}

Our core extension problems are formulated as below.
\begin{problem}\label{canextprob}
Let $y:\Gamma \cf\to\Gamma\cf'$ be a given operator acting on sections of some vector bundle~$\cf$ over $M$
with codomain the section space of another vector bundle over $M$. Then, for fixed ``boundary data'' $f|_\Sigma\in\Gamma\cf|_\Sigma$, where $\Sigma\subset M$
is a hypersurface in $M$ (or $\Sigma=\partial M$),
find $f\in \Gamma\cf$ such that the {\it extension}~$f$ of~$f|_\Sigma$ obeys
$$y f = 0\, .$$
\end{problem}

For the class of \hypertarget{extension operator}{{\it extension operators}} $y$ studied here
we are interested in explicit
%exact global solutions are very difficult to obtain, so we generally consider 
asymptotic solutions. Assuming that $\sigma$ is a defining density for $\Sigma={\mathcal Z}(\sigma)$ the zero locus of~$\sigma$, these can be treated by an expansion in the density $\sigma$~\cite{GWasym}.
In the case that we view~$\sigma$ as a defining scale, so that it gives a {\it preferred} defining density, we obtain  canonical  coordinate independent expansions
$f^{(\ell)}$ 
satisfying
$$
y f^{(\ell)}=\sigma^{\ell} f_{\ell}\, ,
$$ 
for some smooth $f_{\ell}\in \Gamma(\cf'\otimes\ce M[-\ell])$.
We will often abbreviate the right hand side of the above display by $O(\sigma^\ell)$.
In many cases we are able to find asymptotic solutions for arbitrarily high integers~$\ell$.

To begin with, we consider an asymptotic solution to a simple model problem
coming from  the transversality condition for massive
form fields; this is an integrability condition  of Proca's equations.
In the massless limit, this condition reduces to the Coulomb/Feynmann gauge condition for
 Maxwell's equations.

\subsection{The generalised divergence extension problem}\label{coulomb}

Throughout, we will denote the tensor product $\Lambda^k M\otimes \ce M[w]=: \ce^kM[w]$ and refer to sections thereof as \hypertarget{exterior density}{{\it weighted forms}}, or simply forms. Sections of $\Lambda^kM$ will be called {\it true forms} with section space denoted~$\Omega^k M=\ce^kM[0]$.
In a true scale $\tau$, we will write $\extd$ for the exterior action of the Levi-Civita connection $\extd^\tau :=\varepsilon(\nabla^\tau):\Gamma\ce^kM[w]\to \Gamma\ce^{k+1}M[w]$.
Similarly, for the codifferential,~$\cod$~will denote $\cod^\tau:=\iota(\nabla^\tau):\Gamma\ce^kM[w]\to \Gamma\ce^{k-1}M[w-2]$. For almost Riemannian structures, the canonical Levi--Civita
connection $\nabla^o$ determines exterior and interior operators acting on $\Gamma\ce^kM[w]$ 
\begin{equation}\label{d0delta0}
\extd^o := \extd - w \sigma^{-1}\varepsilon(n) \, ,\qquad \cod^o:=\cod - (d+w-2k) \sigma^{-1}\iota(n)\, .
\end{equation}
Here $\varepsilon(\,.\,)$ and $\iota(\,.\,)$ denote, respectively, exterior and  interior products.
We note  the following obvious but highly useful identities.
\begin{lemma}\label{taut}
The operators $$\extd^o:\Gamma \ce^kM[w]\longrightarrow \Gamma \ce^{k+1}M[w]\, ,\qquad \cod^o:\Gamma \ce^kM[w]\longrightarrow \Gamma \ce^{k-1}M[w-2]$$
and $$\cod^o\extd^o+\extd^o\cod^o=:\FL^{\! o}\ :\Gamma \ce^kM[w]\longrightarrow \Gamma \ce^{k}M[w-2]\, ,$$
defined away from $\Sigma$, 
obey
\begin{eqnarray*}
\big(\extd^o\big)^2 \:\: = &\!0\!& =\:\:  \big(\cod^o\big)^2\, ,\\[2mm]
{}[\FL^{\! o},\extd^o]\: =&\!0\!&=\: [\cod^o,\FL^{\! o}]\, ,\\[2mm]
{}[\extd^o,\sigma]\:\: =&\!0\!&=\:\: [\sigma,\cod^o]\, .
\end{eqnarray*}
\end{lemma}
\begin{proof}
All these results follow from standard ones for the exterior derivative, codifferential and form Laplacian
calculating in the {\it preferred interior choice of scale} $g^o\in c$ determined by $\sigma$.
They can also be obtained by explicit computation based on
the following elementary
commutator and anticommutator operator identities
\begin{equation}\label{comms}
\begin{array}{ccc}
{} [\extd,\sigma]=\varepsilon(n) \, ,&\hspace{3mm}&{}[\cod,\sigma]=\iota(n)\, , \\[2mm]
 \{\extd,\varepsilon(n)\}=0\, , &&\{\cod,\iota(n)\}=0\, . 
\end{array}
\end{equation}

\end{proof}

More importantly, from $\extd^o$ and $\cod^o$   we can define operators that extend to the boundary~$\Sigma$.
\begin{definition}\label{wdef}
\hypertarget{wiota}{Let} $A\in \Gamma\ce^kM[w]$. The operators 
$$
\wepsilon :\Gamma\ce^kM[w]\longrightarrow\Gamma\ce^{k+1}M[w+1]\:\mbox{ and }\:
\wiota:\Gamma\ce^kM[w]\longrightarrow\Gamma\ce^{k-1}M[w-1]$$
defined, for some $g\in c$ by
$$
 \wepsilon\,  A := \big[w\varepsilon(n)-\sigma \extd\big] A\: \mbox{ and }\:
\wiota \, A := \big[(d+w-2k) \iota(n)-\sigma \cod\big] A\, ,
$$
respectively extend 
$\wepsilon= -\sigma \extd^o$
and $
\wiota=-\sigma \cod^o$
to~$\Sigma$.
\end{definition}

\begin{remark}
 The operators   $\wepsilon$ and $\wiota$ play a double {\it r\^ole} 
(see the discussion of $\nabla^o$ \hyperlink{doublerole}{above}): Away from~$\Sigma$ in the preferred interior scale $g^o\in c$,
they yield minus the codifferential $-\extd$ and exterior derivative $-\cod$. Along $\Sigma$ and avoiding weights $w=0$ and $w=2k-d$, respectively, they give interior and exterior
multiplication by the normal covector $n$. In reference to this, we will often call  them  the {\it holographic exterior} and {\it interior normals}.
\end{remark}

From Lemma~\ref{taut} (and smoothness), we immediately deduce  the algebra of the operators  $\wepsilon$ and $\wiota$.
\begin{corollary}\label{walgebra}
The holographic exterior normal~$\wepsilon$ and holographic interior normal~$\wiota$ acting on weighted forms obey the relations
$$ \{\wiota,\wepsilon\}=\sigma^2\FL^{\! o}\, ,$$
\begin{eqnarray*}
\wepsilon^{\, 2} \:\: = &\!0\!& =\:\:  \wiota^{\, 2}\, ,\\[2mm]
{}[\FL^{\! o},\wepsilon\, ]\: =&\!0\!&=\: [\, \wiota,\FL^{\! o}]\, ,\\[2mm]
{}[\wepsilon,\sigma]\:\: =&\!0\!&=\:\: [\sigma,\wiota\, ]\, .
\end{eqnarray*}
\end{corollary}

To place our approach in a familiar context, we propose a ({\it na\"\i ve}) version of the extension Problem~\ref{canextprob} for the case $f=A\in\ce^k M[w+k]$
and $y=-\wiota$. The weight $w+k$ is chosen for easier comparison with the tractor calculus analog of this problem.

\newcommand{\Natural}{{\mathbb N}}

\begin{problem}\label{deltaprob}
Given $A|_\Sigma\in \Gamma\ce^kM[w+k]|_\Sigma$ and an arbitrary extension~$A_0$ of this, find $A_i\in\Gamma\ce^kM[w+k-i]$, $i=1,2,\ldots$, so that
$$A^{(\ell)}=A_0+\sigma A_1 + \sigma^2 A_2+\cdots+\sigma^{\ell'} A_{\ell'} \!+O(\sigma^{\ell'+1})$$
solves
\begin{equation}\label{extprob}
\wiota \, A^{(\ell)}=O(\sigma^\ell)\, ,
\end{equation}
off $\Sigma$, for some integers $\ell'$ and $\ell\in\Natural\cup\infty$ as high as possible.
\end{problem}

\begin{remark}
Away from~$\Sigma$, $-\wiota=\sigma \cod^o$ (see Definition~\ref{wdef}). So, in the preferred scale $\sigma$ off~$\Sigma$,  the above Problem yields the divergence condition
\begin{equation}\label{deltaA}
\cod A = 0\, ,
\end{equation}
often referred to as the Coulomb or, in a relativistic context, Feynmann gauge choice for massless fields. 
For  the Proca equation, this is a necessary integrability condition.
The point is that equation~\nn{extprob} naturally  extends the condition~\nn{deltaA} to~$\Sigma$.
Therefore we term the equation $\wiota A=0$ a {\it generalised divergence} condition.
\end{remark}

An elementary approach to treating Problem~\ref{deltaprob}, is to write out Equation~\nn{extprob} in  some choice of scale $g\in c$
and solve iteratively: Using the first relation of~\nn{comms}, one immediately finds the recursion relation
\begin{equation}\label{notsocool}
(i-d-w+k)\iota(n)A_i +\cod A_{i-1}=0\, ,\quad i=0,1,\ldots \, , \quad A_{-1}:=0\, .
\end{equation}
When $d+w-k\neq 0$,
the base case $i=0$ gives the condition 
\begin{equation}\label{along}
\iota(n) A_0= O(\sigma)\, .
\end{equation}
Note that we write $O(\sigma)$ on the right hand side because this and  higher order terms in~$\sigma$ can be removed by appropriately shifting $A_1,A_2,\ldots$ in the ansatz.

Off~$\Sigma$, equation~\nn{along} is a harmless algebraic condition on the choice of extension~$A_0$ of~$A|_\Sigma$. However, 
by the obvious isomorphism  (used henceforth without comment) between $T^*\Sigma$ and the annihilator of the normal in $T^*M|_\Sigma$, the condition $\iota(n)A_0|_\Sigma=0$ says  that $A_0|_\Sigma$ is a differential form on~$\Sigma$. 
This shows that the problem cannot be solved in general, and we must make the restriction
$$
A|_\Sigma=A_0|_\Sigma = A_{_\Sigma} \quad \mbox{for some}\quad A_{_\Sigma}\in \Gamma\ce^k\Sigma[w+k]\, .
$$
Upon making this restriction on the data $A|_\Sigma$, the algebraic recursion clearly has solutions to all orders so long as the coefficient
$i-d-w+k\neq0$. When this does vanish the problem is potentially 
obstructed by $\cod A_{i-1}$ (modulo $O(\sigma)$) where $A_{i-1}\in \Gamma \ce^k M[2k-n]$.
Observe, interestingly enough,  the boundary codifferential 
$\cod_{_\Sigma}$ is conformally invariant acting on~$\Gamma \ce^k \Sigma[2k-n]$.
We will see shortly that these obstructions are mostly avoidable, while the remaining exceptional cases fit into an
rich and interesting picture, see Section~\ref{OBSTQD}.

We have by  now  established the following result.
\begin{proposition}\label{series}
For $d+w-k\notin \Integer_{\geq1}$ and $A|_\Sigma\in\Gamma\ce^k\Sigma[w+k]$, Problem~\ref{deltaprob} can be solved to order $\ell=\infty$. 
When $d+w-k=0$, the restriction on $A|_\Sigma$ can be relaxed. For $d+w-k=m\in\Integer_{\geq 1}$, a solution exists to the
order $\ell=m$.
\end{proposition}

When the structure $(M,c,\sigma)$ is AH,   the  defining scale~$\sigma$ obeys (for $g\in c$)
\begin{equation}\label{unit conormal}
|\nabla \sigma|_g^2 := |n|_g^2 = 1+ O(\sigma)\, , 
\end{equation}
so that $n_g$ is a unit conormal for any $g\in c$. This condition is effectively no restriction; for any  almost Riemannian structure $(M,g,\sigma)$ obeying $|n|^2_g>0$ along $\Sigma$, we can find a new scale $\sigma'$ such that $(M,g,\sigma')$ is AH in an obvious way.
For AH structures a simple product-type solution to Problem~\ref{deltaprob} is available.

\begin{proposition}\label{holpro1}
Let $(M,c,\sigma)$ be an AH structure. 
Then, for any $w\neq -k,k-n$, and $A|_\Sigma \in \Gamma\ce^k\Sigma[w+k]$, Problem~\ref{deltaprob} can be solved to order $\ell=\infty$ by
\begin{equation}\label{reallycool}
%A=\big(\iota(n) -\frac{1}{n+w-k}\, \sigma \cod\big)\big(\varepsilon(n)  -\frac1{w+k}\,  \sigma \extd \big)A_0\, .
A=\scalebox{.95}{$\frac{1}{(w+k)(n+w-k)}$}\: \wiota \, \wepsilon\,  A_0\, .
\end{equation}
\end{proposition} 

\begin{proof}
Note $\wiota A = 0$ identically by virtue of $\wiota^{\, 2}=0$ as established in Corollary~\ref{walgebra}.
It remains to show that $\big(\scalebox{.95}{$\frac{1}{(w+k)(n+w-k)}$}\: \wiota \, \wepsilon\,  A_0\big)\big|_\Sigma=A|_\Sigma$.
Computing along~$\Sigma$ we have 
$$
\scalebox{.95}{$\frac{1}{(w+k)(n+w-k)}$}\: \wiota \, \wepsilon\,  A_0=\iota(n)\varepsilon(n) A_0\, .
$$
Then, using that $n$ is a unit conormal to $\Sigma$ we note that $\iota(n)\varepsilon(n)$ is a projector onto the subbundle~$\ce^k\Sigma[w+k]$ of $\ce^kM[w+k]|_\Sigma$.
But since $A_0|_\Sigma=A|_\Sigma\in \Gamma\ce^k \Sigma[w+k]$ we have $\iota(n) A_0=0$ along~$\Sigma$.
Thus, (using that $\{\iota(n),\varepsilon(n)\}=n^2$) along $\Sigma$ the above display equals $A|_\Sigma$.
\end{proof}

\begin{remark}
Observe that the right hand side of~\nn{reallycool} actually provides a global solution to the $\wiota A=0$ problem with the given boundary data.
The Proposition is an example of a more general {\it holographic boundary projector} technique that we shall develop in Section~\ref{holographicproj}.
There, this terminology is used to refer to a bulk operator that acts as a projector along~$\Sigma$ and solves a set of prescribed bulk equations (such as the one in Problem~\ref{deltaprob}).
\end{remark}

\begin{remark}
The series solution determined by~\nn{notsocool} and the holographic boundary projector solution~\nn{reallycool} are easily verified to be compatible, indeed the latter yields a series solution that terminates at $O(\sigma^2)$ because the ``coefficient'' at that order is coclosed.
Together, Propositions~\ref{series} and~\ref{holpro1} give an order $\ell=\infty$ solution to Problem~\ref{deltaprob} for any weight save $w$ subject to  $n+w-k=0$ or $w=-k\in- \{0,1,\ldots,\lfloor\frac{d-1}2\rfloor\}$.
These exceptional weights 
will be discussed in detail when we consider boundary conditions for higher form Proca systems in Section~\ref{bctractors}.
\end{remark}

\subsection{Conformal tractor calculus}\label{TRACTORS}

For a conformal $d$-manifold $(M,c)$, a key tool will be the \hypertarget{standard tractor bundle}standard tractor bundle~$\ct M$ and associated 
{\it tractor calculus}~\cite{BEG} for building conformally invariant differential operators. The tractor bundle is a rank $d+2$ vector bundle
equipped with a canonical tractor connection $\nabla^\ct$ (or simply $\nabla$ when the context is clear). A given metric $g\in c$ determines
the isomorphism
$$
\ct M\stackrel{g}{\cong}\ce M[1]\oplus T^*M[1]\oplus\ce M[-1]\, .
$$
We will often employ this isomorphism to express sections $T\in\Gamma\ct M$ as
$$
T\stackrel{g}{=}\begin{pmatrix}\nu\\ \mu_a\\ \rho\end{pmatrix}=:T^A\, .
$$
Here, and throughout, we frequently employ an abstract index notation ({\it cf.}~\cite{ot}) to denote sections of the various vector bundles encountered.

In the obvious way, the notation $\stackrel{g}=$ indicates calculations in a scale determined by $g\in c$. We will use this notation for emphasis if the scale is not clear by context.
In terms of the above splitting, the tractor connection is given by
\begin{equation}\label{trconn}
\nd^{\ct}_a
\left( \begin{array}{c}
\si\\\mu_b\\ \rho
\end{array} \right) : =
\left( \begin{array}{c}
    \nabla_a \nu-\mu_a \\
    \nabla_a \mu_b+ \bg_{ab} \rho +\Rho_{ab}\, \nu \\
    \nabla_a \rho - \Rho_{ac}\mu^c  \end{array} \right) .
\end{equation}
Changing to a conformally related metric
$\widehat{g}=e^{2\omega}g$  gives a different
isomorphism, which is related to the previous one by the transformation
formula
\begin{equation}\label{trconn}
\widehat{\begin{pmatrix}\nu\\\mu_b\\\rho\end{pmatrix}}=\begin{pmatrix}\nu\\ \mu_b+\nu\Up_b \\ \rho-\bg^{cd}\Up_c\mu_d-
\tfrac{1}{2}\nu\bg^{cd}\Up_c\Up_d\end{pmatrix}\, ,
\end{equation}
where $\Upsilon_a$ is the one-form $\extd\omega$. It is straightforward to verify that the
right-hand-side of~\nn{trconn} also transforms in this way and this
verifies the conformal invariance of~$\nabla^\ct$.

In the above formul\ae, we have denoted by $\bg_{ab}$ the \hyperlink{conformal metric}{conformal metric} introduced in Section~\ref{CGeom}.
From this, we can build the conformally invariant \hypertarget{tractor metric}{{\it tractor metric}} $h$ on~$\ct M$
given (as a quadratic form on $T^A$ as above) by
\begin{equation*}
%\label{trmet}
\begin{pmatrix}\nu\\ \mu\\ \rho\end{pmatrix}\longmapsto\  \bg^{-1}(\mu,\mu)+2\si \rho \ = : h(T,T)=h_{AB}T^AT^B~; 
\end{equation*}
it is preserved by the connection. 
We shall often write $T\cdot T$ or $T^2$ as a shorthand for the right hand side of this display.
Note that this has signature
$(p+1,q+1)$ on a conformal manifold $(M,c)$ of signature $(p,q)$. The tractor metric
$h_{AB}$ and its inverse $h^{AB}$ are used to identify $\ct M$ with its
dual in the obvious way. 

Tensor powers of the standard tractor bundle $\ct M$, and tensor products
thereof, are vector bundles that are also termed tractor bundles. We
shall denote an arbitrary \hypertarget{tractor bundle}tractor bundle by $\ct^\Phi M$ and write
$\ct^\Phi M[w]$ to mean $\ct^\Phi M\otimes \ce M[w]$; $w$ is then said to be
the weight of $\ct^\Phi[w]$. In the obvious way, we may introduce a \hypertarget{weight operator}{{\it weight operator}} $\w$ on
 sections of weighted tractor bundles $\ct^\Phi M[w]\ni f$ by
\begin{equation}\label{weight}
\w f = w f\, .
\end{equation}

Whereas the tractor connection maps sections of a weight 0 tractor
bundle~$\ct^\Phi$ to sections of $T^*\!M \otimes \ct^\Phi$, there is a
conformally invariant operator which maps between sections of weighted
tractor bundles. This is the \hypertarget{Thomas D}{{\it Thomas D- (or tractor D-) operator}}
$$
D^A :  \Gamma \ct^\Phi M [w]\mapsto \Gamma (\ct^A M\otimes \ct^\Phi M[w-1]),
$$
given in a scale $g$ by
\begin{equation}\label{Dform}
D^A V \stackrel{g}{=}\left(\begin{array}{c} (d+2w-2)w V\\[1mm]
(d+2 w-2) \nabla_a V\\[1mm]
-(\Delta+ \J w) V   \end{array} \right) ,
\end{equation}
where $\Delta=\bg^{ab}\nabla_a\nabla_b$, $V\in \Gamma \ct^\Phi M[w]$ and $\nabla$ is the coupled
Levi-Civita-tractor connection~\cite{BEG,T}. 

A key point to emphasise here is that the Thomas D-operator is a
fundamental object in conformal geometry. On a conformal manifold the
tractor bundle is ``as natural'' as the tangent bundle. Moreover the Thomas D-operator acting on densities in $\Gamma\ce M[1]$ basically
defines the tractor bundle, see~\cite{CapGoirred}.

We will also make frequent use of the canonical conformally invariant operator
$$
X^A: \Gamma \ct^\Phi M[w]\mapsto \Gamma (\ct^A M\otimes \ct^\Phi M[w+1]),
$$
defined by multiplication by the canonical tractor~$X^A$. This derives from the  canonical invariant map $\ce M[-1]\to \ct M$ where $\rho\mapsto X^A\rho$. In a choice of splitting $g\in c$
$$
X^A\stackrel{g}=\begin{pmatrix}0\\0\\1\end{pmatrix}\, .
$$
We may also view $X^A$ as a canonical, {\em null} section of~$\ct^AM[1]$.

\subsection{The calculus of scale}\label{LRs}

Almost Riemannian manifolds are equipped with a  splitting of geometry into ``conformal'' and ``scale'' parts.
This melds perfectly with the conformal tractor calculus. Together these yield a powerful calculus of scale~\cite{Goal,GoIP,GWasym} that we will
further develop and exploit in the following Sections.
Central to this is the object we now define.
\begin{definition}
 Let $\sigma$ be a defining scale for an almost Riemmanian structure $(M,g,\sigma)$. Then
\begin{equation}\label{scaletractor}
I^A:=\frac{1}{d}\, D^A \si\, ,
\end{equation} is called the {\em scale tractor}. Note that $\si=X^AI_A$. 
\end{definition}

In the metric $g^o=\sigma^{-2} \bg$, $I^2=-\frac{2\scalebox{.75}{$\J$}}d$.
Hence we use the terminology {\em almost scalar constant}~\cite{Goal} if an almost Riemannian geometry~$(M,c,\sigma)$ obeys
$$
I^2=\mbox{constant}\, .
$$

Putting together the scale tractor and Thomas D-operator gives a canonical degenerate Laplace operator
$$
I\cdot D:=I^A D_A: \Gamma\ct^\Phi[w]\longrightarrow \ct^\Phi[w-1]\, .
$$
Recall we denote the zero locus ${\mathcal Z}(\sigma)$ by $\Sigma$ (which could possibly be empty).
If we calculate in the metric $g^o=\sigma^{-2} \bg$ away from~$\Sigma$, and trivialise density bundles accordingly, we have
$$
I\cdot D \stackrel{g^o}= -\Big(\Delta^{g^o}+\frac{2w(d+w-1)}{d}\, \J^{g^o}\Big)\, ,
$$
where again $\Delta = \bg^{ab} \nabla_a \nabla_b$ for the coupled Levi-Civita-tractor connection~$\nabla$.
This allows a study of Laplacian eigen-equations. 

If~$\sigma$ also satisfies the  almost scalar constant condition with $I^2=1$, then along
its zero locus
$$
I\cdot D=(d+2w-2)\,  \delta_n\, ,
$$
where $\delta_n\stackrel{g}= n^a \nabla_a^g-w H^g$ is the first order (tractor-twisted) conformal Robin operator~\cite{cherrier,BrGoOps}. Here $n^a$ is a unit normal and $H^g$ is the mean curvature measured in the metric~$g$. Thus, on conformally compact manifolds, the operator $I\cdot D$ unifies both the  interior  Laplace   problem with boundary dynamics and hence we generally dub it the
{\em Laplace--Robin operator}.

\subsection{The Laplace--Robin solution generating algebra}\label{LRalg}

Here and until further notice we work on an almost Riemannian geometry.
In that setting
the Laplace--Robin operator provides a  distinguished choice for the extension operator $y$.
Moreover, together with the defining scale, it generates  an algebra that facilitates the solution of extension problems~\cite{GWasym}.
To display this, we first define a triplet of canonical operators.
\begin{definition}
Let $\sigma\in\Gamma\ce M[1]$ be a defining scale
with vanishing locus $\Sigma$ and nowhere vanishing $I^2$. Then we define the triplet of operators $\{x,h,y\}$ \hypertarget{h}{} 
\begin{equation*}
\begin{array}{rcccc}
x &\! \! :\! \!&\Gamma\ct^\Phi M[w]&\longrightarrow&\Gamma\ct^\Phi M[w+1]\\
&&\rotatebox{90}{$\in$}&&\rotatebox{90}{$\in$}\\[1mm]
& &f&\longmapsto&\sigma f\\[4mm]
h & \!\! :\!\! &\Gamma\ct^\Phi M[w]&\longrightarrow&\Gamma\ct^\Phi M[w]\\
&&\rotatebox{90}{$\in$}&&\rotatebox{90}{$\in$}\\[1mm]
& &f&\longmapsto&(d+2w) f\\[4mm]
y & \!\! :\!\! &\Gamma\ct^\Phi M[w]&\longrightarrow&\Gamma\ct^\Phi M[w-1]\\
&&\rotatebox{90}{$\in$}&&\rotatebox{90}{$\in$}\\[1mm]
& &f&\longmapsto&-\frac{1}{I^2}I\cdot D f\\[2mm]
\end{array}
\end{equation*}
\end{definition}

Following~\cite[Proposition 3.4]{GWasym} the following result holds.
\begin{proposition}\label{thesl2}
 The operators $\{x,h,y\}$ obey the ${\frak g}:={\frak sl}(2)$ algebra
\begin{equation}\label{sl2}
[h,x]=2x\, ,\qquad [x,y]=h\, ,\qquad [h,y]=-2y\, .
\end{equation}
\end{proposition}

This operator  algebra was called a {\it solution generating algebra} in~\cite{GWasym} for reasons that will rapidly become apparent.

\begin{remark}
As proven in~\cite{GWasym} the above solution generating algebra holds for any defining scale~$\sigma$ such that its scale tractor~$I$ is nowhere null. Even in the case~$I^2=0$, a contraction of the above algebra holds~\cite{GWasym}. In the latter half of this Article we will draw further operators into the above algebra that act on tractor forms. That system has a proclivity towards almost Einstein structures.
\end{remark}

In what follows we will often employ the quadratic Casimir
$$
c:=xy+\frac14h(h-2)\, ,
$$
which commutes with $x$, $y$ and $h$.

Our results rely heavily on identities in the enveloping algebra $\cU({\frak g})$, in particular, for any $k\in {\mathbb Z}_{\geq 0}$
\begin{eqnarray}\label{core}
[y,x^k]&=&-k \, x^{k-1}\, (h+k-1)\, ,\\[2mm]
[y^k,x]&=&-k \, y^{k-1}\, (h-k+1)\, .\label{rcore}
\end{eqnarray}

It will also be highly advantageous to extend the enveloping algebra by the operator~$x^\alpha$, for any $\alpha\in{\mathbb C}$,
where
\begin{equation*}
\begin{array}{rcccc}
x^\alpha &\! \! :\! \!&\Gamma\ct^\Phi M[w]&\longrightarrow&\Gamma\ct^\Phi M[w+\alpha]\\
&&\rotatebox{90}{$\in$}&&\rotatebox{90}{$\in$}\\[1mm]
& &f&\longmapsto&\sigma^\alpha f\, .
\end{array}
\end{equation*}
A straightforward calculation~\cite{GWasym} shows that the identity~\nn{core} can be extended to arbitrary values of the exponent $\alpha\in {\mathbb C}$ of $x$
\begin{equation}\label{corarb}
[y,x^\alpha]=-\alpha \, x^{\alpha-1}\, (h+\alpha-1)\, .
\end{equation}

In Section~\ref{logsect}, we will also  need to draw $\log
x$ into our operator algebra. In~\cite{GWasym} it was shown that
%Formally, the identity
%$$
%\lim_{\alpha\rightarrow 0}\frac{x^\alpha-1}{\alpha}=\log x
%$$
%coupled with the relations $[h,x]=2x$ and $[y,x^k]=-x^{k-1} k(h+k-1)$ above suggest  that
\begin{equation}\label{ylogx}
[h,\log x]=2\, ,\qquad [y,\log x]=-x^{-1}(h-1)\, .
\end{equation}
Moreover, if $\mu\in \Gamma \ce_+[w]$ is any positive weighted conformal density,
then $\log\mu$ is a weighted log density, {\it viz.} a section of ${\mathcal F}M[w_0]$, the log density bundle induced from the log representations of ${\mathbb R}_+$~\cite{GWasym}. It
obeys
$$
[h,\log\mu]=2w_0\, , 
$$
read as an operator relation acting on any section of a weighted tractor
bundle.  The relations  in~\nn{ylogx} 
hold on arbitrary sections of weighted tractor bundles as well as on log densities or tensor products of these with conformal
densities, or more generally on a log density bundle in  tensor product with any
weighted tractor bundle.

\subsection{Product solutions}

\label{products}

In this Section we develop various operator identities valid in the universal enveloping algebra $\cU(\frak{ g})$
and completions thereof corresponding to power series in the generator~$x$. These are quite general in their validity and are
intimately related to the theory of extremal projectors (see~\cite{Tolstoy}). It  is   convenient to use $x$ 
as a series variable in this algebraic context.
Later, when we apply these results in a geometric setting,
the Lie algebra generator~$x$ will be represented by the scale~$\sigma$.

\subsubsection{The first solution}

In~\cite{GWasym}, the following problem was posed and solved in terms of a certain solution generating operator based
on a normal ordered representation for elements of~$\cU({\frak g})$.
\begin{problem}\label{first soln}
  Given $f|_{\Sigma}$, and an arbitrary extension $f_0$ of this to
  $\ct^\Phi M[w_0]$ over~$M$, find $f_i\in \ce^{\Phi}[w_0-i]$ (over
  $M$), $i=1,2,\cdots,\ell$, so that
$$
f^{(\ell)}: = f_0 + x f_1  + x^2f_2 +\cdots + x^\ell f_\ell + O(x^{\ell+1})
$$ solves $$y f = O(x^\ell)\, ,$$ off $\Sigma$, 
for $\ell\in \mathbb{N}\cup \infty$ as high as possible.
\end{problem}

\begin{lemma}\label{next}
If $f^{(\ell)}$ solves Problem~\ref{first soln} and $\ell\neq d+2w_0-2$ to $O(x^\ell)$, then 
$$
f^{(\ell+1)} = \frac{1}{(\ell+1)(d+2w_0-2-\ell)}\, \big[xy+(\ell+1)(h-\ell-2)\big]\, f^{(\ell)}
$$
is a solution to $O(x^{\ell+1})$.
\end{lemma}   
\begin{proof}
The proof relies on a direct computation utilizing the algebra~\nn{sl2} and~\nn{core}. For brevity, we call $h_0:=d+2w_0$
and in addition use that $yf^{(\ell)}=O(x^\ell)$ means that $yf^{(\ell)}=x^\ell g_{\ell}$ for some $g_\ell$:
\begin{eqnarray*}
y \big[xy+(\ell+1)(h-\ell-2)\big]\, f^{(\ell)}&=&\big[xy+\ell(h-\ell-1)\big] y f^{(\ell)}\\[1mm]
&=&\big[xy+\ell(h-\ell-1)\big] x^{\ell} g_{\ell}\\[1mm]
&=&x^{\ell+1} y g_\ell\ = \ O(x^{\ell+1})\, .
\end{eqnarray*}
The normalization factor $\big[(\ell+1)(h_0-\ell-2)\big]^{-1}$ ensures that $f^{(\ell+1)}$ is again a series of the form $f_0+x f_1 + \cdots$. 
\end{proof}

The above Lemma, suggests a  solution to Problem~\ref{first soln} in terms of products of operators acting on the initial data~$f_0$ on~$\Sigma$. A key insight is to express 
elements in $\cU({\frak g})$ in terms of the Casimir $c$ and Cartan generator~$h$.
In particular, the operator appearing in Lemma~\ref{next} is one of a sequence of operators defined  for any $j\in  {\mathbb Z}$
$$
c_j:= c-\frac14(h-2j)(h-2j-2)\, .
$$
Simple algebra shows that
\begin{equation}\label{ci}
c_j=xy+j(h-j-1)\, ;
\end{equation}
Lemma~\ref{next} utilises $c_{\ell+1}$.
Moreover (and this fact will feature prominently in the following)
$$
c_0=xy\, .
$$
Since $yc=cy$ and $yh=(h+2)y$, it follows immediately that
$$
y\, c_j=c_{j-1}\, y\, .
$$
In turn we have
\begin{eqnarray}
y\, c_1c_2\cdots c_{\ell}&=&c_0 c_1\cdots c_{\ell-1}\, y\nonumber\\
&=&xy\,  c_1c_2\cdots c_{\ell-1}\, y\nonumber\\
&=&x\,  c_0 c_1\cdots c_{\ell-2}\, y^2\nonumber\\
&\vdots&\nonumber\\
&=&x^\ell y^{\ell+1}\, .\label{yleft}
\end{eqnarray}
Hence we already see that 
\begin{equation}\label{prodform}
{\mathscr F}:=c_1c_2\cdots c_{\ell} f_0
\end{equation}
obeys $ y \mathscr F=O(x^\ell) $.  To solve Problem~\ref{first soln},
we still need to relate this product to a series expansion in~$x$.  To
that end we recall the following identity from~\cite{GWasym}
(which also can be derived directly from equation~\nn{core})
\begin{equation}\label{bessel}
y \, x^j y^j = j(j-1) x^{j-1}y^{j} - j x^{j-1} y^{j} (h-2) + x^j y^{j+1}\, ,\qquad j\in{\mathbb Z}_{\geq0}\, .  
\end{equation}
It is very convenient to elevate this relation to one for formal power series of elements in~$\cU({\frak g})$ of the form
$$
\colon K(z)\colon=\sum_j \colon z^j \colon\,  \alpha_j(h)\, ,\qquad K(z)\in {\mathbb C}[[z]]\, .
$$
When we allow the coefficients $\alpha_j$ to be polynomials in the Cartan generator $h$, these are always ordered to the right. Moreover, for $z:=xy$  the normal ordering $\colon \, \bullet  \, \colon$ defines the operators
$$
\colon z^j\colon := x^j y^j\, ,\qquad 
\colon x^ mz^j\colon := x^{m+j} y^j\, ,\qquad 
\colon z^j y^m\colon := x^j y^{j+m}\, .\qquad 
$$ 
Then equation~\nn{bessel} can be restated as
\begin{equation}\label{besselop}
y \, \colon K(z)\colon = \colon\big(z K''(z) + K(z)\big) y\colon  -\colon K'(z)y\colon\,   (h-2)\, .
\end{equation}

The following technical Lemma establishes the relationship between solutions in the product form~\nn{prodform} and the formal series solutions of~\cite{GWasym}.

\begin{lemma}
The product $c_1c_2\cdots c_\ell$ equals the following polynomial in~$x$
\begin{equation}\label{claim}
c_1c_2\ldots c_\ell = \sum_{j=0}^\ell \frac{\ell!}{j!}\,  x^j y^j (h-j-2)_{\ell-j}=:\, \colon K^{(\ell)}(z)\colon
\ ,
\end{equation}
where the Pochhammer symbol denotes the product $$\big(h-j-2\big)_{\ell-j}:=(h-j-2)(h-j-3)\cdots(h-\ell-1)\, ,$$
(subject to $(m)_0:=1$). 
\end{lemma}

\begin{proof}
We proceed by induction. The base case $\ell=1$ holds by definition
of $c_1$
$$
\colon K^{(1)}(z)\colon=c_1=(h-2) + xy\, .
$$
The induction step requires  us to  demonstrate that
$$c_{\ell+1}\sum_{j=0}^\ell \frac{\ell!}{j!}\,  x^j y^j (h-j-2)_{\ell-j}=\sum_{j=0}^{\ell+1} \frac{(\ell+1)!}{j!}\,  x^j y^j (h-j-2)_{\ell-j+1}\, .$$
Computing the left hand side using~\nn{besselop} and the fact that $x\colon K(z)y\colon =\colon zK(z)\colon$ yields
$$
\colon \Big[z^2 \frac{d^2K^{(l)}(z)}{dz^2}+zK^{(l)}(z)\Big] \colon-\colon z\,  \frac{dK^{(l)}(z)}{dz}\colon\ (h-2)
+(\ell+1)\,  \colon K^{(l)}(z)\colon\,  (h-\ell-2)\, .
$$
By inspection, the last term reproduces all terms in the expression for $\, \colon K^{(\ell+1)}(z)\colon$ save for the last, $j=\ell+1$, term in the sum
on the right hand side of the induction step requirement above, namely $ x^{\ell+1} y^{\ell+1}  $.
The remaining terms produce exactly this missing term essentially because $K^{(\ell)}(z)$ solves equation~\nn{besselop}
to $O(z^{\ell+1})$. This is also easily explicitly verified by direct computation: Calling $E:=z\frac{d}{dz}$, we have
$$
\colon \Big[z^2 \frac{d^2K^{(l)}(z)}{dz^2}+zK^{(l)}(z)\Big] \colon-\colon z\,  \frac{dK^{(l)}(z)}{dz}\colon\ (h-2)
\hspace{5.3cm}
$$
\begin{eqnarray*}
\hspace{1.5cm}&=&\colon \big[E(E-1)+z\big]K^{(l)}(z)\colon-\colon E K^{(l)}(z)\colon (h-2)\\[2mm]
&=&-\sum_{j=0}^\ell \frac{\ell!}{j!}\,  x^j y^j (h-j-2)_{\ell-j}\, j(h-j-1) + \sum_{j=0}^\ell \frac{\ell!}{j!}\,  x^{j+1} y^{j+1} (h-j-2)_{\ell-j}\\
&=&-\sum_{j=1}^\ell \frac{\ell!}{(j-1)!}\,  x^j y^j (h-j-1)_{\ell-j+1} + \sum_{j=1}^{\ell+1} \frac{\ell!}{(j-1)!}\,  x^{j} y^{j} (h-j-1)_{\ell-j+1}\\[2mm]
&=&x^{\ell+1}y^{\ell+1}\, .
\end{eqnarray*}
\end{proof}

Now, returning to the Problem~\ref{first soln}, taking $f_0$ to have definite weight $w_0$, so that $h f_0= h_0 f_0$ with
$h_0:=d+2w_0$, then the above Lemma  implies
\begin{eqnarray}\label{nonnormalized}
{\mathscr F}&=&c_1c_2\cdots c_{\ell} f_0= \sum_{j=0}^\ell \frac{\ell!}{j!}\,  x^j y^j (h_0-j-2)_{\ell-j}f_0
\nonumber\\
&=&\ell! (h_0-2)_{\ell}\ \colon\Big[\, 1+\frac{z}{h_0-2}+\frac{z^2}{2!(h_0-2)(h_0-3)}
+\cdots + \frac{z^{\ell}}{\ell! (h_0-2)_{\ell}}\Big]\colon\,  f_0
\, .\label{nonnormalized}
\end{eqnarray}
Therefore we see that, away from values $h_0\in{\mathbb Z}_{\geq 2}$,  ${\mathscr F}\big/(\ell! (h_0-2)_{\ell})$ solves Problem~\ref{first soln} and thus  have established the following result. 
\begin{proposition}\label{firstsolprop}
The solution to Problem~\ref{first soln} for any $\ell\in {\mathbb Z}_{\geq 1}$, when $h_0\notin {\mathbb Z}_{\geq 2}$,~is
\begin{equation}\label{product solution}
f^{(\ell)}=\Big[\prod_{j=1}^{\ell}\frac{c_j}{j(h_0-j-1)} \Big]f_0\, .
\end{equation}
\end{proposition}
\begin{remark}
In the limit $\ell \to \infty$, the polynomial in $z$ in the squared brackets on the right hand side of display~\nn{nonnormalized}
for the non-normalised solution, gives the series expansion of the Bessel function of the first kind
\begin{equation}\label{besselfn}
K^{h_0}(z):=z^{\frac{h_0-1}{2}} \Gamma(2-h_0)\, J_{1-h_0}\big(2\sqrt{z}\, \big) \, .
\end{equation}
Once the Cartan generators $h$ in the $c_j$ of the product solution\nn{product solution} are replaced by their eigenvalues $h_0$
when they act on $f_0$, then
the infinite product from the $\ell\to\infty$ limit can also be formally evaluated; use of the computer package Maple yields
$$
\prod_{j=1}^{\infty}\frac{c-\frac14(h_0-2j)(h_0-2j-2)}{j(h_0-j-1)} =\frac{-\big(c-\frac{1}{4}h_0(h_0-2)\big)^{-1}\ \Gamma(2-h_0)}{
\Gamma\Big(\frac{4-h_0}{2}-\sqrt{c+\frac14}\,\Big)\, \Gamma\Big(\frac{4-h_0}{2}+\sqrt{c+\frac14}\, \Big)}\, .
$$
Note that this ratio of gamma functions can be encoded by a single beta function and the left hand side can also be expressed in terms of $xy$ so
$$\prod_{j=1}^\infty\frac{xy+j(h_0-j-1)}{j(h_0-j-1)}=
-\left[xy\, \beta\left(\frac{4-h_0}{2}-\sqrt{c+\frac14}\, ,\, \frac{4-h_0}{2}+\sqrt{c+\frac14}\right)\right]^{-1}
\, .$$
\end{remark}
\begin{remark}
The formal square root of the Casimir $\sqrt{c+\frac14}$ appeared in~\cite{Hallowell:2007zb,Hallowell:2005np}
where it was employed to extend the algebra ${\mathcal U}(\frak{g})$ by an operator whose eigenvalues gave the
so-called ``depth'' of states. In the language used here, this amounts to supposing that~$f$ is a simultaneous eigenfunction
of the Casimir $c$ and Cartan generator $h$ and obeys $y^j f\neq 0 = y^{j+1} f$. The integer $j$ is the depth. 
\end{remark}
\begin{lemma}\label{xright}
$$
c_1c_2\ldots c_{\ell}\,  x = x^{\ell+1}y^{\ell}\, .
$$
\end{lemma}
\begin{proof}
The proof is identical to the derivation of the identity~\nn{yleft} for $y$ acting from the left on a product of $c_i$'s
save that one now uses the readily verified identity
$$
c_j\,  x = x \, c_{j-1}\, .
$$
\end{proof}
\begin{remark}
Lemma~\ref{xright} shows that the solution~\nn{product solution} does not depend on how $f_{_\Sigma}$ is extended off~$\Sigma$ to $f_0\in\Gamma\ct^\Phi M[w_0]$ since shifting $f_0\to f_0+x f_1$ for any $f_1\in\Gamma\ct^\Phi M[w_0-1]$ we have $c_1c_2\ldots c_{\ell}\, x f_1=O(x^{\ell+1})$. 
\end{remark}
\begin{remark}
The poles at $h_0=2,3,\ldots$ in the solution given in Proposition~\ref{firstsolprop} suggest that the simple product 
formula must be adjusted at these values. Indeed, observe that if $h_0\in {\mathbb Z_{\geq 2}}$ and $f^{(h_0-2)}$
is an $O(x^{h_0-2})$ solution to~Problem~\ref{first soln}, then Lemma~\ref{next} fails to provide a $O(x^{h_0-1})$ solution 
because
$$
c_{h_0-1} f^{(h_0-2)}=c_0 f^{(h_0-2)}=xy f^{(h_0-2)}
$$
and therefore vanishes along~$\Sigma$. 
\end{remark}

Even when $h_0\in{\mathbb Z}_{\geq 2}$, we may still evaluate products of the form~\nn{prodform} for $\ell> h_0-1$
\begin{eqnarray*}
\overline{\mathscr F}=c_1c_2\ldots c_{\ell} f_0 &=& c_1\ldots c_{h_0-2} xy c_{h_0}\ldots c_\ell f_0\\
&=& x^{h_0-1} y^{h_0-1}c_{h_0}\ldots c_{\ell}f_0\\
&=& x^{h_0-1} c_{1}\ldots c_{\ell-h_0+1}\overline f_0\, , \mbox{ where } \overline f_0=y^{h_0-1} f_0\, .
\end{eqnarray*}
Clearly $y\overline{\mathscr F}=O(x^{h_0-1+\ell})$ but the behaviour of $\overline{\mathscr F}$ near $\Sigma$ is no longer of the Dirichlet
type stipulated in Problem~\ref{first soln}. In fact, this is an example of a solution of the second kind.

\subsubsection{The second solution}\label{sol2}

Since the equation $yf=0$ corresponds to a second order differential equation for a normal ordered solution generating operator
as in~\nn{besselop}, we expect to find a second independent solution. In~\cite{GWasym}, by extending the algebra ${\mathcal U}({\frak g})$ by $x^\alpha$ for $\alpha\in {\mathbb C}$, this was shown to be the case with precisely determined boundary behaviour encapsulated by the following.
\begin{problem}\label{second soln}
  Given $\f_0|_{\Sigma}\in \Gamma \ct^\Phi M[-d-w_0+1]|_\Sigma$ and an
  arbitrary extension $\f_0$ of this to $\Gamma\ct^\Phi M[-d-w_0+1]$ over
  $M$, find $\f_i\in \Gamma\ct^{\Phi}M[-d-w_0+1 -i]$,
  $i=1,2,\cdots$, so that
\begin{equation}\label{Jf}
\f^{(\ell)}: = x^{h_0-1} \big( \f_0 + x\,  \f_1  + x^2\, \f_2 +\cdots + O(x^{\ell+1}) \big)\, , \quad h_0:=d+2w_0\, ,
\end{equation}
solves $y \f  = O(x^{h_0-1+\ell})$, off $\Sigma$, 
for $\ell\in \mathbb{N}\cup \infty$
as high as possible.
\end{problem}
\begin{remark}
If the leading behaviour of $\f$ is relaxed to $\f=x^\alpha\big( \f_0 + x\,  \f_1  +\cdots\big)$, one quickly learns (see~\cite{GWasym}, Section~5.3) that
the value $\alpha=h_0-1$ is forced. An easy way to understand this is to note that for any $g\in\Gamma \ct^\Phi M[-d-w_0+1]$, by virtue of~\nn{corarb},
the following holds
$$
y\, x^{h_0-1}g=x^{h_0-1} y\,  g\, .
$$
This underlies a scale duality map on solutions which is implicit in the solution below and discussed in more detail in Section~\ref{sol2}. 
\end{remark}

We now seek a product type solution of Problem~\ref{second soln}.
The above remark immediately shows us how to proceed. Since the operator $y$ effectively commutes through the leading~$x^{h_0-1}$
behaviour, we are left with a version of the original Problem~\ref{first soln}, but now at the weight~$-d-w_0+1$. This establishes the following
result.
\begin{proposition}\label{secondsolprop}
The solution to Problem~\ref{second soln} for $\ell$ arbitrarily high and any $h_0\notin{\mathbb Z_{\leq0}}$ is
\begin{equation}\label{product solution 2}
\f^{(\ell)}=x^{h_0-1}\, \Big[\prod_{j=1}^{\ell}\frac{c_j}{j(1-j-h_0)} \Big]\f_0\, .
\end{equation}
\end{proposition}

\begin{remark}
Again, we have a product form for the solution generating operator, defined independently of how $\f|_\Sigma$ is extended off $\Sigma$.
Performing the infinite product, this is formally encoded by the operator
$$
\frac{-\ x^{h_0-1}\ \big(c-\frac{1}{4}h_0(h_0-2)\big)^{-1}\ \Gamma(h_0)}{
\Gamma\Big(\frac{h_0+2}{2}-\sqrt{c+\frac14}\,\Big)\, \Gamma\Big(\frac{h_0+2}{2}+\sqrt{c+\frac14}\, \Big)}\ 
: \ \Gamma\ct^\Phi M[-d-w_0+1]|_\Sigma\longrightarrow \Gamma\ct^\Phi M[w_0]\, .
$$
In addition, composing the above with the operator $$y^{h_0-1}:\Gamma\ct^\Phi M[w_0]\to \Gamma\ct^\Phi M[-d-w_0+1]\, ,$$ yields 
$$
\frac{-\ x^{h_0-1}\ \big(c-\frac{1}{4}h_0(h_0-2)\big)^{-1}\ \Gamma(h_0)\ y^{h_0-1}}{
\Gamma\Big(\frac{h_0+2}{2}-\sqrt{c+\frac14}\,\Big)\, \Gamma\Big(\frac{h_0+2}{2}+\sqrt{c+\frac14}\, \Big)}\ 
: \ \Gamma\ct^\Phi M[w_0]|_\Sigma\longrightarrow \Gamma\ct^\Phi M[w_0]\, .
$$
The image of both the above formal maps is in the kernel of the operator $y:\Gamma\ct^\Phi M[w_0]\to \Gamma\ct^\Phi M[w_0]$.
\end{remark}

Notice that for any $h_0-1\in {\mathbb Z}_{>0}$ (respectively $ {\mathbb Z}_{<0}$), the $\ell=\infty$ power series solutions of the first (respectively second) kind are obstructed. Moreover, when $h_0=1$, the solutions of the first and second kind coincide. As explained in Section~5.4 of~\cite{GWasym} this presages the appearance of solutions with log terms. We will return to these in Section~\ref{logsect}.

\section{Tractor exterior calculus}

\label{tractorforms}

Form bundles, by virtue of their exterior calculus, play a distinguished {\it r\^ole}  among tensor bundles; the same holds
for so-called {\it tractor form} bundles. This is perhaps not surprising on the basis of the close relationship between
these and form bundles in the Fefferman--Graham ambient space. In fact, the latter connection is exploited heavily in the 
exposition on tractor forms given in~\cite{BrGodeRham}. Here we mainly eschew an ambient approach, and construct the main elements
of a {\it tractor exterior calculus}  directly from the viewpoint of the underlying conformal manifold~$(M,c)$.

The natural tractor analog of a one-form belongs to the weight~$-1$, rank $d+2$ tractor bundle $\ct^AM[-1]$ which for a given $g\in c$ 
enjoys the isomorphism
$$
\ct^AM[-1]\stackrel{g}{\cong}\ce M[0]\oplus\Lambda^1M\oplus\ce M[-2]\, .
$$
The $k$-fold exterior product of this bundle yields a tractor bundle of weight~$-k$ which we denote $\ct^k M[-k]$.
It is isomorphic for a choice $g\in c$ to the direct sum
$$
\ct^k M[-k]\stackrel{g}{\cong}\Lambda^{k-1} M\oplus\Lambda^k M\oplus\ce^{k-2}M[-2]\oplus\ce^{k-1}M[-2]\, .
$$ 
(Recall that $\ce^kM[w]$ denotes the bundle $\Lambda^k M\otimes \ce M[w]$.)
Tensoring with the weighted conformal density bundle $\ce M[w+k]$ we obtain a weight~$w$ bundle which we term the
\hypertarget{exterior tractor bundle}{{\it exterior (or form) tractor bundle}} $\ct^k M[w]$. For $g\in c$ this obeys
$$
\cT^kM[w]\stackrel{g}{=}\ce^{k-1}M[w+k]\oplus \ce^kM[w+k]\oplus \ce^{k-2}M[w+k-2]\oplus \ce^{k-1}M[w+k-2]\, .
$$
We may employ this isomorphism to express sections $\F\in\Gamma\ct^kM[w]$ as
$$
\F\stackrel{g}=\begin{pmatrix}F^+\\F\\F^{+-}\\F^-\end{pmatrix}\, .
$$
For ease of discourse, we will often refer to the various components of this splitting as the ``northern'', ``western'', ``eastern'' and ``southern''
slots, respectively, for the weighted forms $F^+$, $F$, $F^{+-}$ and $F^-$. This terminology was devised in the presentation of~\cite{BrGodeRham}
to reflect the composition series structure of the tractor form bundle. On the other hand, here we use a representation that makes
operator compositions compatible with matrix multiplication.

Invariant sections of $\ct^k M[w]$ corresponding to conformally related choices of metric $\widehat g=e^{2\omega} g$,
are related by the transformation formula
\begin{equation}\label{trafsF}
\widehat{
\begin{pmatrix}
F^+ \\ F \\ F^{+-} \\ F^-
\end{pmatrix}}
=
\begin{pmatrix}
1 & 0 & 0 & 0\\
\varepsilon(\Upsilon)&1&0&0\\
-\iota(\Upsilon)&0&1&0\\
\frac12\big(\varepsilon(\Upsilon)\iota(\Upsilon)-\iota(\Upsilon) \varepsilon(\Upsilon)\big) & -\iota(\Upsilon) & -\varepsilon(\Upsilon) & 1
\end{pmatrix}
\begin{pmatrix}
F^+ \\ F \\ F^{+-} \\ F^-
\end{pmatrix}\, .
\end{equation}
(As usual we denote $\Upsilon := {\bf d} \omega$ and $\iota,\varepsilon$ are the standard exterior and interior products on forms.)
Moreover, in terms of the above splitting, the tractor connection
is given by
$$
\nabla^{\cT}_v\begin{pmatrix}
F^+ \\ F \\ F^{+-} \\ F^-
\end{pmatrix}
=
\begin{pmatrix}
\nabla_v & -\iota(v) & \varepsilon(v) & 0\\
\varepsilon({\mathbb P}v)&\nabla_v&0&\varepsilon(v)\\
-\iota({\mathbb P}v)&0&\nabla_v&\iota(v)\\
0 & -\iota({\mathbb P}v) & -\varepsilon({\mathbb P}v) & \nabla_v
\end{pmatrix}
\begin{pmatrix}
F^+ \\ F \\ F^{+-} \\ F^-
\end{pmatrix}\, ,
$$
where $v$ is an arbitrary section of $TM$ and ${\mathbb P}$ denotes the \hypertarget{Schouten} canonical endomorphism $TM\to TM$
obtained from the Schouten tensor (in an index notation, ${\mathbb P} v := (\Rho^a_b v^b)$).
It is easy to check that this formula enjoys the transformation property in~\nn{trafsF}.
In what follows, we will often assume some choice of splitting and represent various operators on tractor forms
by a $4\times 4$ matrix of operators as in the example above. 

\renewcommand{\G}{{\Cal G}}

For completeness we record the exterior algebra and \hypertarget{tractor star}{{\it tractor Hodge star operator}} of~\cite{BrGodeRham}
in the splitting notations above. The wedge product maps sections $\F$ and $\G$ of ${\cT}^kM[w]$ and $\cT^{k'}M[w']$, respectively, 
to a section $\F\wedge \G$ of $\cT^{k+k'}[w+w']$ given in the above splitting as
$$
\F\wedge\,  \G\stackrel{g}{=}\scalebox{.9}{$
\begin{pmatrix}F^+\\[1mm] F\\[1mm] F^{+-}\\[1mm] F^-\end{pmatrix}\bigwedge
\begin{pmatrix}G^+\\[1mm] G\\[1mm] G^{+-}\\[1mm] G^-\end{pmatrix}
=
\begin{pmatrix}F^+\wedge G+(-1)^k F\wedge G^+\\
F\, \wedge G\\ 
F^{+-}\!\wedge G + (-1)^{k}\Big[F^-\wedge G^+-F^+\wedge G^-\Big]+ F\wedge G^{+-} \\
F^-\wedge G + (-1)^k F\wedge G^-
\end{pmatrix} \, .$}
$$
The conformal Hodge star operator $*:\ce^kM[w]\stackrel\cong{\ \rightarrow} \ce^{d-k}M[d+w-2k]$. We shall denote the degree operator on
forms by~$\degree$, so that for $A\in\Gamma\ce^kM[w]$, $$\degree A=k\, A\, .$$ In these terms we have $**=(-1)^{\scalebox{.70}{\degree}(d-\scalebox{.70}{\degree})+q}$
(where $(-1)^q$ is the sign of the metric determinant in the metric signature $(p,q)$  and is unity for  Riemannian metrics). 
From the conformal Hodge star we can build the tractor Hodge star 
$$\star:\ct^{k}M[w]\stackrel{\cong}{\ \rightarrow}\ct^{d+2-k}M[w]\, ,$$
defined in a given splitting by 
$$
\star\,  \F \ \stackrel{g}=\ \begin{pmatrix}*(-1)^{\scalebox{.75}{\degree}}&0&0&0\\
0&0& * &0\\
0&*\, (-1)  &0&0\\
0&0&0&*(-1)^{\scalebox{.75}{\degree}-1}
\end{pmatrix}\begin{pmatrix}F^+\\F\\F^{+-}\\F^-\end{pmatrix}\, .
$$
We call the \hypertarget{tractor degree}{degree operator on tractor forms}~$\N$ so that for $\F\in\Gamma\ct^kM[w]$, 
\begin{equation}\label{tdegree}\N \F=k\, \F\, .\end{equation} 
It follows directly that
$$
\star\, \star=(-1)^{\N(d+2-\N)+q+1}\, ,
$$
in concordance with the relationship between tractor forms and forms in the Fefferman--Graham ambient space~\cite{BrGodeRham}.

\newcommand{\cb}{{\mathcal B}}
\newcommand{\sgn}{{\rm sgn}}
The standard inner product on $\Omega^k M$ provides an inner product on $\Gamma\ct^k M\big[\frac d2\big]$ as follows.
Firstly, $\Gamma\ct^{d+2} M[-d]$ is canonically isomorphic to $\Omega^d M$ and not only are elements of the latter conformally invariant, 
but they can be integrated over $M$ (oriented). For $\cb\in\Gamma\ct^{d+2} M[-d]$ we shall write $\int_M \cb$ assuming this above isomorphism.
Since the tractor Hodge star $\star:\Gamma\ct^k M\big[-\frac d2\big]\to\Gamma \ct^{d+2-k}M\big[-\frac d2\big]$, given $\A,\A'\in\Gamma\ct^k M\big[-\frac d2\big]$,
we define the conformally invariant inner product $(\A,\A')=\int_M \A\ \scalebox{1.1}{$\wedge$}\,  \star\A'$.
Moreover,  at arbitrary weights $\A,\A'\in\Gamma\ct^k M\big[w\big]$
$$
\langle \A,\A'\rangle := \sgn\left(\int_M  \A \ \scalebox{1.1}{$\wedge$}\, \star \A'\right) \in \{\pm 1,0\}\, ,
$$
is conformally invariant. This pairing is useful for developing orthogonal decompositions.

\subsection{Algebra of  invariant operators}
\label{algebra}

We now develop a conformally invariant exterior calculus of tractor forms. This efficiently compresses 
a large class of operators and accompanying identities. 

The utility of differential forms relies on the exterior derivative operator
$$
\extd:\Omega^k M\to\Omega^{k+1}M\, ,\qquad \extd^2=0\, ,
$$
its Hodge dual $(-1)^{\scalebox{.7}{$\degree$}}\! *^{-1}\!\extd\,  *$, {\it viz.} the codifferential
$$
\cod:\Omega^k M\to\Omega^{k-1}M\, ,\qquad \cod^2=0\, ,
$$
and the supersymmetry algebra formed by these and the form Laplacian
$$
\FL=\cod \extd + \extd \cod=:\{\cod,\extd\}\, ,\qquad [\extd,\FL]=0=[\FL,\cod]\, .
$$
An examination of the Thomas D-operator in~\nn{Dform}, suggests there ought exist a conformally invariant operator on
form tractors that unifies the exterior derivative, codifferential and form Laplacian. A natural candidate for such an operator
would simply be the exterior, or skew, action of the Thomas D-operator $D^A$. Although partially true, this expectation is not fulfilled,
essentially because what appears in the bottom slot of~\nn{Dform} is the Bochner--Laplacian, rather than its form counterpart.
The solution to this problem was given in~\cite{BrGodeRham};  just as the difference between form and Bochner Laplacians
is given by a natural action of the curvature tensor, a similar modification~$\slashed D{}^A$ of the Thomas D-operator exists 
such that its exterior action on tractor forms is nilpotent.

It is straightforward to compute the exterior derivative-type operator $\varepsilon(\slashed{D})$ of~\cite{BrGodeRham} in the $4\times 4$ 
matrix notation for a given splitting as introduced above. We shall denote this operator by
$$
{\mathscr D} : \Gamma\cT^kM[w] \longrightarrow \Gamma\cT^{k+1}M[w-1]\, ,
$$
and refer to it as the \hypertarget{exterior D}{{\it exterior tractor D-operator}}.
Our result for this computation is
\begin{equation}\label{D}
{\mathscr D}\stackrel{g}=
\scalebox{.95}{$
\begin{pmatrix}
-(d+2w-2){\extd}&(w+{\degree})(d+2w-2) & 0 & 0\\
0&(d+2w-2){\extd}&0&0\\
\!\!{\FL }+(w+{\degree}-1)(\J-2{\mathbb P})\!\!\!\!\!&-2{\cod }&\!\!\!\!\!\!\!(d+2w){\extd}&\!(w+{\degree}-1)(d+2w)\!\! \\
[\J-2{\mathbb P},{\extd}] & \!\!\!\!-{\FL }-(w+{\degree})(\J-2{\mathbb P}) \!\!\!\!& 0 & -(d+2w){\extd}
\end{pmatrix}\, .$}
\end{equation}
Here we have canonically extended ${\mathbb P}$ by linearity to act on forms of degree $k$.
Nilpotency 
$$
{\mathscr D}^2=0\, ,
$$
of ${\mathscr D}$ follows from~\cite{BrGodeRham}, but can also be readily verified by a simple matrix composition
based on the above display. 

The tractor analog of the codifferential is also given in~\cite{BrGodeRham} as the interior action of $\iota(\slashed D)=(-1)^{\N-1}\star^{-1}\D\ \star$. We will denote this operator by
$$
{\mathscr D}^\star:\Gamma\cT^kM[w] \longrightarrow \Gamma\cT^{k-1}M[w-1]\, ,
$$
and dub it the \hypertarget{interior D}{{\it interior tractor D-operator}}.
It acts in a given splitting according to 
\begin{equation}\label{Ds}
{\mathscr D}^\star\stackrel{g}=
\scalebox{.88}{$
\begin{pmatrix}
-(d+2w-2){\cod }&0&\!\!\!\!\!\!-(d+2w-2)(d+w-{\degree})&\!\!\!\!\!\!\!\!\!\!\!\!0\\
\!\!-{\FL } +(d+w-{\degree}-1)(\J-2{\mathbb P}) &\!(d+2w){\cod }\!\!\!\!\!&\!\!\!\!\!-2{\extd} \!& \!\!\!\!\!\!\!\!\!\!\!\!\!\!\!\!\!\!(d+2w)(d+w-{\degree}-1)\!\\
0&0&\!\!\!\!(d+2w-2){\cod }&\!\!\!\!\!\!\!\!\!\!\!\!0\\
[\J-2{\mathbb P},{\cod }]&0&\!\!\!\!\!\!\!\!\!\!\!\!\!-{\FL } +(d+w-{\degree})(\J-2{\mathbb P})\!\!\!\!&\!\!\!\!\!\!\!\!-(d+2w){\cod }
\end{pmatrix}\, .$}
\end{equation}
Per~\cite{BrGodeRham}, it is also nilpotent
$$
{\mathscr D}^\star{}^2=0\, ,
$$
and anticommutes with the exterior derivative
$$
{\mathscr D}^\star\, {\mathscr D}+ {\mathscr D}\ {\mathscr D}^\star=0\, .
$$
Altogether, including the weight operator $h:=d+2\w$ (see~\nn{weight}) and tractor form degree~$\N$ (as in~\nn{tdegree})
we have established  the beginnings of an {\it exterior tractor algebra}
\begin{center}
%\shabox{
\begin{tabular}{cc}
$
\{\D,\Ds\}=0\, $, & $\D^2=0=\Ds{}^2$,\\[1mm]
$[\N,\D]=\D\, $, & $[\N,\Ds]=-\Ds$,\\[1mm]
$[h,\D]=-2\D\, $, & $[h,\Ds]=-2\Ds$.
\end{tabular}
%}
\end{center}

We now
augment these identities with exterior and interior multiplication by the canonical tractor $X^A$:
Let us denote exterior multiplication by this with $\varepsilon(X)$, which we term the \hypertarget{exterior canonical tractor operator}{{\it exterior canonical tractor}}
$$
\X:\ct^k M[w]\rightarrow \ct^{k+1} M[w+1]\, .
$$
In a choice of splitting this is represented by
\begin{equation}\label{X}
\X\ \stackrel{g}= \ \begin{pmatrix}
0&0&0&0\\
0&0&0&0\\
-1&0&0&0\\
0&1&0&0
\end{pmatrix}\, .
\end{equation}
For later use, it is worth noting that this operator moves the northern and western slots, respectively, to the eastern and southern slot.
Its adjoint, the \hypertarget{interior canonical tractor}{{\it interior canonical tractor}} $\iota(X)=(-1)^{\N-1}\star^{-1}\X\star$ will be denoted
\begin{equation}\label{Xs}
\Xs:\ct^k M[w]\rightarrow \ct^{k-1} M[w+1]\, .
\end{equation}
In a choice of splitting this is represented by
\begin{equation}\label{Xs}
\Xs\ \stackrel{g}=\   
\begin{pmatrix}
0&0&0&0\\
1&0&0&0\\
0&0&0&0\\
0&0&1&0
\end{pmatrix}\, ,
\end{equation}
which moves the northern slot to the west and the eastern slot to the south.

Therefore, we may now add to our exterior tractor algebra the relations
\begin{center}
%\shabox{
\begin{tabular}{cc}
$
\{\X,\Xs\}=0\, $, & $\X^2=0=\Xs{}^2$,\\[1mm]
$[\N,\X]=\X\, $, & $[\N,\Xs]=-\Xs$,\\[1mm]
$[h,\X]=2\X\, $, & $[h,\Xs]=2\Xs$.
\end{tabular}
%}
\end{center}

To compute an algebra for products of the interior and exterior Thomas D- and canonical tractor operators amongst one another
we rely on the following: 
\begin{lemma}\label{DXid}
Let $h^{AB}$ denote the tractor metric and $h:=d+2\w$. Then the Thomas~D- and canonical tractor operators obey the identity
\begin{equation}\label{ureq}
h\,  X^A D^B - (h-2)\,  D^B X^A - 2 \, X^B D^A + h(h-2) \, h^{AB}
=0\, .\end{equation}
\end{lemma}

The simplest proof of the above employs the ambient techniques of~\cite{GoPetCMP,CapGoamb}, although this relation also follows from known tractor identities, see for example~\cite{GoSrni99,Goinvariant}.
Proofs of this and other results relying on an ambient formulation are collected in Appendix~\ref{AMBIENT}.

\begin{remark}\label{hashing}
Suppose $\cW$ is any weight~$-2$ tractor tensor in $\Gamma(\otimes^k\End(\cT M)
)$ and let $\hash$ denote the natural \hypertarget{tensor endomorphism}{tensorial action of endomorphisms}
on tractor sections (to the right) so, for example, for $k=1$ and acting on a rank one
tractor $T^A$
$$
\cW^\hash T^A:=T^B\, \cW_B{}^{A}\,  .
$$
Moreover, let us suppose that $\cW$ is orthogonal to the canonical tractor $X^A$, {\it i.e.},
the contraction of $X^A$ with any index of $\cW$ vanishes. Hence 
$$
\big[\cW^{\overbrace{\hash\cdots\hash}^{k{\rm \  times}}},X^A\big]=0\, .
$$
It follows immediately that Lemma~\nn{DXid} still holds upon the replacement of the Thomas D-operator by
$$
D^A\longmapsto  D^A - X^A \cW^{\hash\cdots \hash}\, .
$$
\end{remark}

Now recall the definition of $\slashed D{}^A:=D^A - X^A \Omega^{\hash\hash}$ of~\cite{BrGodeRham}. Here $\Omega\in\Gamma\big(\otimes^2\End(\cT M)\big)$
 equals $1/(d-4)$ times the $W$-tractor of~\cite{GoSrni99} in dimensions other than~$4$, and  is perpendicular to the canonical tractor $X^A$.
It follows immediately from Lemma~\nn{DXid} and Remark~\nn{hashing} that
\begin{equation}\label{genrel}
h \, X^A \slashed{D}{}^B - (h-2) \, \slashed{D}{}^B X^A - 2 \, X^B \slashed{D}{}^A + h(h-2) \, h^{AB}=0\, ,
\end{equation}
in dimensions $d\neq 4$.

\begin{remark}
\label{Wtilde} As it stands, $\slashed{D}$ is not defined in four dimensions. This restriction is
inconsequential for our purposes. The first point is that the exterior version of the operator~$\D$ given in~\nn{D} is well-defined in all dimensions $d\geq 3$, see~\cite{BrGodeRham}. 
(A quick way to see this is to observe $\X \Omega^{\hash\hash}$  is well defined in any dimension.)
On the other hand, beyond the exterior setting the above formula is useful for efficiently proving several results. In practice for the problems solved here, we may use the above formula in all dimensions because in the presence of an Einstein scale 
a version of it holds in four dimensions:
It is possible to construct an invariant tractor $\widetilde  W/(d-4)$ that equals $W/(d-4)$ in dimensions other than four and is well defined when $d=4$, although in that dimension
(and only then) $\widetilde  W/(d-4)$ depends on the scale~\cite[Section 4]{powerslap}; see also the proof of Proposition~\ref{Yform} below.
%Thus equation~\nn{genrel} is not valid for four dimensional structures that are not conformally Einstein. 
\end{remark}

\begin{proposition}\label{XDalgebra}
The exterior and interior Thomas D- and canonical tractor operators  $\D$, $\Ds$, $\X$ and $\Xs$, satisfy 
\begin{center}
%\shabox{
\begin{tabular}{c}
\scalebox{1.1}{$
(h-2)\, \D\X+ (h+2)\, \X\D \, = \, 0\,  = \, (h-2)\, \Ds\Xs+ (h+2)\, \Xs\Ds\, ,$}
\\[4mm]
\scalebox{1.1}{$
h\,  \X \Ds+ (h-2)\, \Ds \X + 2\Xs \D  -\Big(\frac{d+h}{2}-\N+2\Big) h(h-2)=0\, ,
$}
\\[4mm]
\scalebox{1.1}{$
h\,  \Xs \D + (h-2)\, \D \Xs + 2\X \Ds + \Big(\frac{d-h}{2}-\N\Big) h(h-2)=0\, .
$}
\end{tabular}
%}
\end{center}
\end{proposition}

\begin{proof}
The proof amounts to acting with the left hand side of the identity~\nn{genrel} on an arbitrary tractor form and then
taking appropriate irreducible parts.
%, respectively, applying
%exterior multiplication to both indices, interior multiplication to both indices, {\it sequentially} interior multiplication to index~$B$ followed by exterior %multiplication to index~$A$ and  finally exterior multiplication to index~$B$ and in turn
%interior multiplication to index~$A$, to generate the four identities listed. In doing so, one must remember that
%applying interior multiplication on index~$B$ and then exterior multiplication on index $A$ to $M^B N^A$ (say), 
%produces $-\iota(M)\varepsilon(N)+M^A N_A$. Moreover, $X^A \slashed D{}_A=h-2=\slashed D{}^A X_A-2$.
\end{proof}

At the weight $w=-d/2$, the operators $\D$, $\Ds$ are the zero maps on $\ker\X$, $\ker\Xs$, respectively.
 We can, however, define a nontrivial analog of $\D$ and $\Ds$ at this weight on these spaces
by considering the residues of $(d+2w)^{-1}\D$ and $(d+2w)^{-1}\Ds$  at $w=-\frac d2$.
This motivates the following definitions.
\begin{definition}\label{HAT}
Acting on $\ct^k M[w]$ with $w\neq -\frac d2$, define the composition of operators
$$
\wD:=\, \D\ \frac{1}{h}\, ,\qquad \wDs:=\, \Ds\, \frac 1h\, .
$$
For $w=-\frac d2$ 
%Let $\A\in \ker\Xs\subset \Gamma \ct^k M[-{\ts \frac d2}]$. 
define conformally invariant operators
$$\wD: \ker\X\subset \Gamma \ct^k M[-{\ts \frac d2}]\longrightarrow \ker\X\subset \Gamma\ct^{k+1} M[-1-{\ts \frac d2}]$$
and
$$\wDs: \ker\Xs\subset \Gamma \ct^k M[-{\ts \frac d2}]\longrightarrow \ker\Xs\subset\Gamma\ct^{k-1} M[-1-{\ts \frac d2}]$$
via their expressions for some $g\in c$ acting on $\A\in\ker\X\subset \Gamma \ct^k M[-{\ts \frac d2}]$
$$
\wD \A \ \stackrel{g}=\ \wD\, \begin{pmatrix}0\\0\\B\\\phi\end{pmatrix}\ :=\ \begin{pmatrix}0\\ 0\\\extd B +(k-\frac d2-2)\phi \\-\extd\phi\end{pmatrix}
$$
and on $\A \in\ker\Xs\subset \Gamma \ct^k M[-{\ts \frac d2}]$
$$
\wDs \A \ \stackrel{g}=\ \wDs\, \begin{pmatrix}0\\A\\0\\\phi\end{pmatrix}\ :=\ \begin{pmatrix}0\\\cod A-(k-\frac d2)\phi\\0\\-\cod\phi\end{pmatrix}\, .
$$
\end{definition}

\begin{remark}
The notation~$\frac{1}h$ above makes sense because we can always work with weight eigenspaces. 
%The definitions at weights $w=-\frac d2$ for which $\frac1h$ would blow up are obvious continuations of the operators $\D$ and $\Ds$ restricted
%to the kernels of $\X$ and~$\Xs$, respectively. 
\end{remark}

The operators of the above Definition are intimately related to the exterior and interior actions of the first order differential operator $D^{AB}:=X^{B} \widetilde D^A-X^A \widetilde D^B$  termed the double D-operator~\cite{GoSrni99,Goinvariant}. In particular, the operators $\wD\,  \X$ and $\wDs\Xs$ are given, for some $g\in c$, by

\begin{equation}\label{doubleD}
\D_{[2]}\!:=\!\wD\X \stackrel{g}=\begin{pmatrix}0&0&0&\ 0\\0&0&0&\ 0\\
-\extd & w+\degree  &0&\ 0\\ 0&-\extd &0&\ 0\end{pmatrix}\:\:\:\: \mbox{and}\quad\!\!
\Ds_{[2]}\!:=\wDs\Xs \stackrel{g}=
\begin{pmatrix}
0& \ 0\ & 0& 0\\
\cod & \ 0\  & \! d+w-\degree \!& 0\\
0& \ 0\ & 0& 0\\
0& \ 0\ & -\cod & 0
\end{pmatrix} .
\end{equation}
An important fact is that both the above, nilpotent, Grassmann even,   first order operators  obey the Leibnitz rule acting on products of tractor forms. 
 We shall term~$\D_{[2]}$ and~$\Ds_{[2]}$  the {\it exterior}, and {\it interior double D-operators}, respectively.

It is also useful to introduce a version of Definition~\ref{HAT} tailored to the cokernels of the operators $\X$ and $\Xs$.
For a choice of $g\in c$, the cokernel of $\X$ is defined by the equivalence relation
$$
\begin{pmatrix}
\: \psi\: \\A\\B\\\phi
\end{pmatrix}\ \sim \ 
\begin{pmatrix}
\psi\\A\\B + b\\ \   \phi \,  + f\ 
\end{pmatrix}\in \coker \X\, ,
$$ 
(an analogous formula holds for $\coker \Xs$) 
so we will employ the notations
$$
\coker \X \ni
\begin{pmatrix}
\:\psi\:\\A\\ * \\ *
\end{pmatrix}\, \quad
\mbox{ and }\quad 
\begin{pmatrix}
\:\psi\:\\ *\\ B \\ *
\end{pmatrix}\in \coker \Xs\, ,
$$ 
for elements of $\coker \X$ and $\coker \Xs$, respectively.

\begin{definition}\label{WHAT}
Acting on $\ct^k M[w]$ with $w\neq 1-\frac d2$, define the composition of operators
$$
\Dt:=\, \frac{1}{h}\, \D \, ,\qquad {\Dt}^\star:=\, \frac 1h \, \Ds\, .
$$
For $w=1-\frac d2$ 
%Let $\A\in \ker\Xs\subset \Gamma \ct^k M[-{\ts \frac d2}]$. 
define 
$$\Dt: \coker\big(\X, \Gamma \ct^k M[1-{\ts \frac d2}]\big)\longrightarrow \coker\big(\X,\Gamma\ct^{k+1} M[-{\ts \frac d2}]\big)$$
and
$$\Dts: \coker\big(\Xs,\Gamma \ct^k M[1-{\ts \frac d2}]\big)\longrightarrow \coker\big(\Xs,\Gamma\ct^{k-1} M[-{\ts \frac d2}]\big)$$
via their expressions acting on $\A\in\coker\big(\X, \Gamma \ct^k M[1-{\ts \frac d2}]\big)$ for some $g\in c$ 
$$
\Dt \A \ \stackrel{g}=\ \Dt\, \begin{pmatrix}\ \psi\ \\A\\ *\\ *\end{pmatrix}\ :=\ \begin{pmatrix}-\extd \psi +(k-\frac d2+1)A\\ \extd A\\ *\\ *\end{pmatrix}
$$
and on $\A \in\coker\big(\Xs, \Gamma \ct^k M[1-{\ts \frac d2}]\big)$
$$
\Dts \A \ \stackrel{g}=\ \Dts\, \begin{pmatrix}\ \psi\ \\ *\\ B\\ *\end{pmatrix}\ :=\ \begin{pmatrix}-\cod \psi + (k-\frac d2 -3) B\\ *\\ \cod B\\ *\end{pmatrix}\, .
$$

\end{definition}
Observe that the operators $\X\Dt$ and $\Xs\Dts$ are well-defined acting on~$\ct^kM[w]$ at any weight. It is easy to relate them to the exterior
and interior double D-operators using the first line of Proposition~\ref{XDalgebra}. The result of that computation is the following.
\begin{proposition}\label{hatstotilde}
On weighted tractor forms
$$
\wD \X + \X \Dt = 0=\wDs \Xs + \Xs \Dts\, . 
$$
\end{proposition}

To conclude this Section, we draw the exterior double $D$-operators into our algebra.
\begin{proposition}
On weighted tractor forms
\begin{center}
%\shabox{
\begin{tabular}{lr}
$
{}[\D_{[2]},\Xs]=- \big(\frac{h-d}{2}+\N-2\big)\X\, $, & $[\Ds_{[2]},\D]=\big(\frac{h-d}{2}+\N\big)\, \Ds$,\\[3mm]
$[\D_{[2]},\Ds]=\big(\frac{h+d}{2}-\N+2\big)\, \D $, & $[\Ds_{[2]},\X]=-\big(\frac{h+d}{2}-\N\big)\, \Xs$.
\end{tabular}
%}
\end{center}
\end{proposition}
\begin{proof}
The most direct proof is a straightforward application of the matrix expressions for the exterior tractor operators as given in Equations~\nn{D},~\nn{Ds},~\nn{X},~\nn{Xs}
and~\nn{doubleD}.
\end{proof}

\subsection{Algebra of differential splitting operators}\label{compinsert}

\hypertarget{insertions}{}

For many applications one begins with a differential form $A\in \Omega^\bullet M$, or perhaps more generally a form density $A\in \Gamma\ce^\bullet[\, . \, ]$
which one wishes to handle using the tractor machinery. Given a choice of $g\in c$ and the form~$A$, there in fact exists a quartet of differential {\it insertion operators} $(q_{\rm N}, q_{\rm E}, q_{\rm S}, q_{\rm W})$ which invariantly insert the form into a uniquely determined tractor form, respectively, whose northern, eastern, southern or western slot is given by $A$. We describe these in this Section.

Firstly, we have an isomorphism and its formal adjoint~\cite{BrGodeRham}
$$
q:\Gamma\ce^{k}M[w+k]\stackrel{\cong}{\to}\coker \big(\X, \Gamma \ct^k M[w]\big)\mbox{ and }
q^*\!:\ker \Xs \subset\Gamma \ct^k M[w] \stackrel{\cong}{\to} \Gamma\ce^{k}M[w+k]
$$
via, for some $g\in c$
$$
A \ {\longmapsto} \begin{pmatrix}0\\A\\ *\\ \  *\ \end{pmatrix} \quad \mbox{ and }\quad
\begin{pmatrix}0\\A\\ 0\\ \  \phi\ \end{pmatrix}
\longmapsto A\, .
$$

\begin{figure}
\scalebox{.8}{
\begin{tabular}{rcl}
&$q_N$&\\
&$\ker(\D,\Ds)$&\\
\raisebox{1.7cm}{$q_W\quad\ker(\Ds,\Xs)$}&\includegraphics[scale=.5]{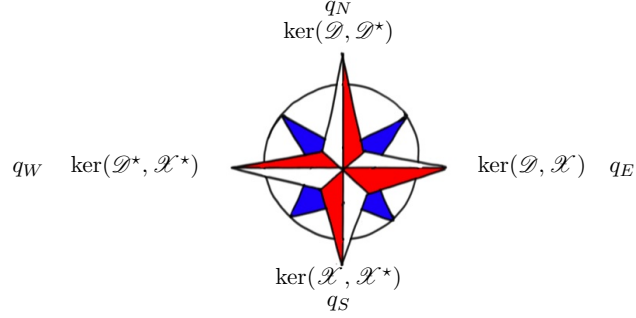}&\raisebox{1.7cm}{$\ker(\D,\X)\quad q_E$}\\[-3mm]
&$\ker(\X,\Xs)$&\\[-1mm]
&$q_S$&\\
\end{tabular}}
\caption{The insertion operators summarised by the points of a compass.}
\end{figure}

The following differential splitting operators have proved to be important~\cite{BrGodeRham,Esrni}.
The first of these is in fact algebraic.
\begin{lemma}[South]\label{southlemma} We have the isomorphism
$$
q_{\rm S}:\Gamma\ce^{k-1}M[w+k-2]\stackrel{\cong}{\longrightarrow}\ker(\X,\Xs)\subset \Gamma \ct^k M[w]
$$
via
$$
A\stackrel{g}{\longmapsto} \begin{pmatrix}0\\0\\0\\ \ A\ \end{pmatrix} \, .
$$
\end{lemma}

\begin{remark}
We will often refer to a tractor form  in $\ker(\X,\Xs)$ as a {\it southern tractor}, and employ the corresponding
language for each of the insertion Lemmas below.
\end{remark}

\begin{proof}
The proof follows immediately by computing the kernels of~$\X$ and~$\Xs$ by  inspecting their matrix expressions in a given $g\in c$ as displayed in equations~\nn{X} and~\nn{Xs}.
\end{proof}

\begin{corollary}\label{soucor}
Let $\B\in \ker(\X,\Xs)\subset \Gamma \ct^k M[w]$ and $w\neq  -k+2,k-d$. Then 
$$
\B=-\frac{1}{(w+k-2)(d+w-k)}\,  \Xs\X\, \Dt\, \wDs\, \B\, ,
$$
and when $w= -k+2,k-d$, $\Xs\X\Dt\, \wDs\B=0$ and the singularities there are removable.
\end{corollary}

In this Article we make heavy use of the western insertion operator that invariantly inserts a (weighted) degree $k$-form in a degree $k$-tractor form.
It is described by the following result.
\begin{lemma}[West]\label{West}
When $w\neq k-d$ there is an isomorphism
$$
q_{\rm W}:\Gamma\ce^kM[w+k]\stackrel{\cong}{\longrightarrow}\ker(\wDs,\Xs)\subset\Gamma \ct^k M[w]
$$
via
$$
A\stackrel{g}{\longmapsto} \begin{pmatrix}0\\ A\\0\\-\frac{1}{d+w-k}\ \cod A\end{pmatrix} \, .
$$
At the critical weight $ w=k-d$, at any fixed scale $g\in c$, there is a related isomorphism that we also denote $q_W$
\begin{equation}\label{qWspecial}
\ker(\cod,\cod)\subset\big(\Gamma\ce^kM[2k-d]\oplus \Gamma\ce^{k-1}M[2k-d-2]\big)\stackrel{\cong}{\longrightarrow}\ker(\wDs,\Xs)\subset \Gamma \ct^k M[k-d]
\end{equation}
via
$$
(A,\phi)\stackrel{g}{\longmapsto} \begin{pmatrix}0\\ A\\0\\\ \phi\ \end{pmatrix} \, .
$$
\end{lemma}

\begin{proof}
The proof proceeds as in the previous Lemma, but in addition requires a computation of the kernel of the interior tractor D-operator 
based on the display~\nn{Ds}.
\end{proof}

\renewcommand{\ts}{\textstyle}

\begin{remark}\label{specialwest}
In Equation~\nn{qWspecial} it  is natural to identify the domain objects $(A,\phi)$ with the  components of the invariant tractor on the right in a scale $g\in c$.
Therefore, for some other $\hat g\in c$, this pair is determined by Equation~\nn{trafsF}  to be
$(\hat A,\hat\phi)$ with $\hat A=A$, $\hat\phi=\phi-\iota(\Upsilon) A$ and $\Up$ defined as in Section~\ref{TRACTORS}.
These are coclosed for any $g\in c$.
%$\ker(\cod,\cod)\subset\big(\Gamma\ce^kM[2k-d]\oplus \Gamma\ce^{k-1}M[2k-d-2]\big)\stackrel{\cong}{\longrightarrow}\ker(\wDs,\Xs)\subset \Gamma \ct^k M[k-d]$.

%When $w=-\frac d2$, the restriction to $\ker \cod$ is relaxed and we have
%$$
%q_W :\Gamma\ce^kM[k-{\ts \frac d2}]\oplus \Gamma\ce^{k-1}[k-{\ts \frac d2} -2]\stackrel{\cong}{\longrightarrow}\ker(\Ds,\Xs)\subset \Gamma \ct^k M[-{\ts \frac d2}] \cong
%\ker(\Xs)\subset \Gamma \ct^k M[-{\ts \frac d2}] \, .
%$$

%via
%$$
%(A,\phi)\stackrel{g}{\longmapsto} \begin{pmatrix}0\\ A\\0\\\phi\end{pmatrix} \, .
%$$
\end{remark}

%\begin{remark}
%For the true form case with $w=-k\neq-\frac d2$, $$q_W : \Omega^k M \stackrel{\cong}{\longrightarrow}\ker(\Ds,\Xs)\subset \Gamma \ct^k M[-k]\, .$$
%\end{remark}

%\edz{A: At the weight $w=-d/2$, the operator $\Ds$ is the zero map on $\ker\Xs$. }

\begin{comment}
We can however define a nontrivial analog of $\Ds$ at this weight
by considering the residue of $\Ds/(d+2w)$ at $w=-\frac d2$.
This motivates the following definition.
\begin{definition}\label{HAT}
Let $\A\in \ker\Xs\subset \Gamma \ct^k M[-{\ts \frac d2}]$. 
Then 
$$\wDs: \ker\Xs\subset \Gamma \ct^k M[-{\ts \frac d2}]\longrightarrow \Gamma\ct^{k-1} M[-1-{\ts \frac d2}]$$
via
$$
\wDs \A \ \stackrel{g}=\ \wDs\begin{pmatrix}0\\A\\0\\\phi\end{pmatrix}\ :=\ \begin{pmatrix}0\\\cod A-(k-\frac d2)\phi\\0\\-\cod\phi\end{pmatrix}\, .
$$
\end{definition}
\noindent It is simple to verify that the above map is well defined by examining Displays~\nn{Ds} and~\nn{Xs}.

This also establishes an isomorphism
$$
\widehat q_W:\Gamma\ce^kM[k-{\ts \frac d2}]\stackrel{\cong}{\longrightarrow}\ker(\wDs,\Xs)\subset\Gamma \ct^k M[-{\ts \frac d2}]
$$
via
$$
A\stackrel{g}{\longmapsto} \begin{pmatrix}0\\ A\\0\\\frac{1}{k-\frac d2}\ \cod A\end{pmatrix} \, .
$$
\end{comment}

\begin{remark}\label{obvious}
The maps $q$ and $q^*$ are intimately related to $q^{\phantom{-1}}_W$ and $q_W^{-1}$. Indeed, one has the operator identities
$$
\X q_W = \X q\, \qquad \mbox{ and }\qquad q_W^{-1} \Xs = q^* \Xs\, .
$$
Obvious analogs of these relations also hold for the north and east insertion operators introduced below.
The algebraic operators corresponding to $q$ and $q^*$ for those cases---defined on appropriate cokernel and kernels---will be denoted
$q_{(E)},q_{(N)}$ and $q_{(E)}^*,q_{(N)}^*$. For consistency one should write $q_{(W)}$ for $q$, but we shall use this operator so often that
the latter notation is preferred.
\end{remark}

The eastern insertion operator is related to its western counterpart by (tractor) Hodge duality. It is given by the following Lemma.

\begin{lemma}[East]\label{East}
For $w\neq -k+2$
$$
q_{\rm E}:\Gamma\ce^{k-2}M[w+k-2]\stackrel{\cong}{\longrightarrow}\ker(\wD,\X)\subset \Gamma \ct^k M[w]
$$
by
$$
A\stackrel{g}{\longmapsto} \begin{pmatrix}0\\0\\ A\\-\frac{1}{w+k-2}\, \extd A\end{pmatrix} \, .
$$

\end{lemma}

%\begin{proof}
%The proof is completely analogous to the one of the western lemma above.
%\end{proof}

\begin{remark}
When $ w=-k+2$ and at any fixed scale $g\in c$ we have the isomorphism
$$
\ker (\extd,\extd)\subset\big(\Omega^{k-2}M\oplus\Omega^{k-1}M\big)\stackrel{\cong}{\longrightarrow}\ker(\D,\X)\subset \Gamma \ct^k M[-k+2]
$$
with
$$
(A,F)\stackrel{g}{\longmapsto} \begin{pmatrix}0\\0\\ A\\ F\end{pmatrix} \, .
$$
%while for $w=-\frac d2$
%$$
%q_{\rm E}:\Omega^{k-2}M\oplus\Omega^{k-1}M\stackrel{\cong}{\longrightarrow}\ker(\D,\X)\subset \Gamma \ct^k M[-{\ts\frac d2}]\, .
%$$ 
\end{remark}

\begin{comment}
At the weight $w=-d/2$, the operator $\D$ is the zero map on $\ker\X$. Just as for $\Ds$, we can however define a nontrivial analog of $\D$ at this weight
by considering the residue of $\D/(d+2w)$ at $w=-\frac d2$.
We make therefore the following definition.
\begin{definition}
Let $\A\in \ker\X\subset \Gamma \ct^k M[-{\ts \frac d2}]$. 
Then 
$$\wD: \ker\X\subset \Gamma \ct^k M[-{\ts \frac d2}]\longrightarrow \Gamma\ct^{k+1} M[-1-{\ts \frac d2}]$$
via
$$
\wD \A \ \stackrel{g}=\ \wD\begin{pmatrix}0\\0\\A\\F\end{pmatrix}\ :=\ \begin{pmatrix}0\\0\\\extd A+(k-{\ts \frac d2}-2)F\\-\extd F\end{pmatrix}\, .
$$
\end{definition}

\noindent It is simple to verify that the above map is well defined by examining Displays~\nn{Ds} and~\nn{Xs}.
This also establishes an isomorphism
$$
\widehat q_E:\Gamma\ce^kM[k-{\ts \frac d2}]\stackrel{\cong}{\longrightarrow}\ker(\D,\X)\subset\Gamma \ct^k M[-{\ts \frac d2}]
$$
via
$$
A\stackrel{g}{\longmapsto} \begin{pmatrix}0\\ 0\\A\\-\frac{1}{k-\frac d2-2}\ \extd A\end{pmatrix} \, .
$$
\end{comment}

For the northern  insertion operator there are 
four classes of special weights to account for:
$$
w=\left\{
\begin{array}{lc}
1-\frac d2 \, ,&\\[1mm]
-\frac d2\, , &\\[1mm]
-k\, , & \multirow{2}{*}{ \scalebox{1.3}[2.9]{\}}  \raisebox{.25cm}{$k\in \{0,1,\ldots ,d+2\}\, .$}}\\[1mm]
%-k+2 \, ,& \\
-d-2+k\, ,& \\[1mm]
%-d+k\, ,&
\end{array}
\right.
$$
%Moreover, in even dimensions, there are various intersections of these conditions; this has been depicted in Figure~\ref{dualitygraph}.

\begin{lemma}[North]\label{northlemma}
 For $w\neq -\frac d 2,-k, -d-2+k$
$$
q_{\rm N}:\Gamma\ce^{k-1}M[w+k]\stackrel{}{\longrightarrow}\ker(\D,\Ds)\subset \Gamma \ct^k M[w]
$$
by
$$
A\stackrel{g}{\longmapsto} \begin{pmatrix}A\\[2mm] \frac{1}{w+k}\, \extd A\\[2mm]-\frac{1}{d+w-k+2}\, \cod A\\[2mm]  
 -\frac{1}{d+2w}\Big(\frac{1}{w+k}\, \cod\extd 
 -\frac{1}{d+w-k+2}\, \extd \cod +\J-2\P\Big) A\end{pmatrix} \, .
$$
Moreover, the map $q_N$ is an isomorphism whenever $w\neq1-\frac d2$.
\end{lemma}
\vspace{-1mm}
\noindent Although technically more involved, the proof uses exactly the same techniques as in the other insertion Lemmas.

\begin{corollary}\label{northcor}
Let $\F\in \ker(\D,\Ds)\subset\Gamma\ct^k M[w]$ and $w\neq 1-\frac d2,-\frac d2,-k, k-d-2$, then
$$
\F=-\frac{1}{(d+2w)(w+k)(d+w-k+2)}\, \Ds\wD\, \X\Xs \F\, .
$$
\end{corollary}
\begin{proof}
An elementary proof is to use the above explicit expression for $\ker(\D,\Ds)$ and then calculate
along the same lines as previously. Alternatively, Proposition~\ref{XDalgebra} can be used to move the operators $\D$ and $\Ds$
to the right where they annihilate $\F$. 
\end{proof} 

\begin{remark}
Note that at $w=1-\frac d2$
the display of the above Lemma gives {\it a} solution to $\D\F=0=\Ds\F$, but not the most general one.
\end{remark}

\begin{comment}
\begin{remark}\label{Fermatsremark} It is useful for later purposes to characterise the kernel of $\D$ and $\Ds$ for all weights~$w$ and degrees~$k\in\{0,1,\ldots,d+2\}$.
These are classified and listed in Appendix~\ref{remarks}.
\end{remark}
\end{comment}

\begin{remark}\label{spinningcompass}
It is clear how to project onto each slot and move from slot to slot. Starting from a weight $w$, degree $k$ northern tractor, and ignoring for this discussion distinguished weights, we have the projector
$$
\F=-\frac1{(d+2w)(w+k)(d+w-k+2)}\, \Ds\wD\X\Xs \F\, .
$$
Calling $\widetilde\A=-\frac1{(d+2w)(w+k)(d+w-k+2)}\, \wD\X\Xs \F$ gives 
$
\F=\Ds\widetilde\A$ with $\widetilde\A\in \ker(\wD,\X)$ so $\widetilde \A$ is a weight $w+1$, degree $k+1$, eastern  tractor with corresponding projector 
$$
\widetilde\A=-\frac1{(w+k)(d+w-k+2)}\, \wD\X\Xs \Dts\widetilde \A\, .
$$
To reach the west, write $\F=\D\A$ with
$\A=\frac1{(d+2w)(w+k)(d+w-k+2)} \, \wDs\X\Xs \F$ so that 
$$
\A=\frac1{(w+k)(d+w-k+2)}\, \wDs\X\Xs \Dt\A\, ,$$ 
which is the projector for weight $w+1$, degree $k-1$, western tractors. The same procedure holds for the south where $\F=\D\wDs\B$
and $\B=-\frac{1}{(d+2w)(w+k)(d+w-k+2)}\,  \Xs\X\F$ and the projector is given in Corollary~\ref{soucor}.
\end{remark}

The projector onto  western tractors described in the above Remark  is the one most often required in later developments. Hence  we record it in the following.
\begin{proposition}\label{WESTPROJECTOR}
The operator $\Pi_W:\Gamma\ct^kM[w]\longrightarrow\Gamma\ct^kM[w]$, $w\neq -k,k-d$, defined by
$$
\Pi_W:=\scalebox{.95}{$\frac{1}{(w+k)(d+w-k)}$}\, \Ds_{[2]}\, \D_{[2]}\, ,
$$
obeys $\Pi_W^2=\Pi_W$. Moreover, if $\A\in \ker(\wDs,\Xs)\subset\Gamma\ct^kM[w]$, then $\Pi_W \A=\A$.
\end{proposition}

We close this Section with a useful technical result which follows immediately from the machinery given in this Section.
\begin{proposition}\label{diffsplit}
Let $\A\in \ker(\wDs,\Xs)\subset\Gamma\ct^kM[w]$. Then
$$
\Xs \Dt \A = (w+k) \A\, ,\qquad \wDs \X \A = (d+w-k)\A\, .
$$
\end{proposition}
\begin{proof}
For the second identity,  observe that since $\Xs \A = 0$, we may write $\A=\Xs \B$ for some $\B$. Thus the left hand side 
is well defined for any weight $w$. The result then follows by simple application of the algebra of Proposition~\ref{XDalgebra}.
The first identity can be established the same way when $w\neq 1-\frac d2,k-d$. The first weight is a simple removable singularity,
as can be seen from Lemma~\ref{West}. At the weight $w=k-d$ there is no singularity but one needs to use Equation~\nn{qWspecial}
of Lemma~\ref{West} to establish the result. 
\end{proof}

\newcommand{\FN}{F_N}
\newcommand{\FW}{F_W}
\newcommand{\FE}{F_E}
\newcommand{\FS}{F_S}

\subsection{Conformally invariant equations and the cohomology of the Tho\-mas $D$ operator}\label{coho}

For later developments, we need to understand the conformal differential operators on forms and their 
origins in the tractor calculus.
At low orders, these  follow from the basic equation
$$
\D \F =0\, ,
$$
and various refinements thereof. 

The simplest conformally invariant differential equation for a differential form is the closure condition
$$
\extd  A =0\, ,
$$
which is conformally invariant for any $A\in \Omega^k M$ with $0\leq k\leq d$. For physical models, the form $A$ can either be interpreted
as a field strength or potential. In the former case, it is traditional to use the symbol $F$. Imposing as well a divergence condition, we have the curvature version of the (higher form) {\it Maxwell's equations}
$$
\extd F = 0 = \cod F\, .
$$
The divergence equation is conformally invariant in even dimensions when $F\in \Omega^{\frac{d}{2}}M$.

The potential version of the higher form Maxwell's equations are obtained by taking a divergence of the closure condition
$$
\cod \extd A =0\, .
$$
This equation enjoys a {\it gauge invariance}
$$
A\sim A + \extd \alpha\, ,
$$
for $A\in \Omega^kM$ and $\alpha \in \Omega^{k-1} M$ which is already evident from the conformally invariant de Rham complex
$$
\stackrel{\extd}\longrightarrow \Omega^\bullet M\stackrel{\extd}\longrightarrow\, .
$$
The potential form of Maxwell's equations is conformally invariant in even dimensions for $A\in \Omega^{\frac d2-1}M$.

\newcommand{\boxdn}{{\scalebox{1.1}{$\boxast$}}_{_{\rm BDN}}}
\newcommand{\boxmp}{{\scalebox{1.1}{$\boxbox$}}_{_{\rm MP}}}

Once weighted forms are considered, there are further conformally invariant equations: The higher form {\it Branson--Deser--Nepomechie equation}
$$
\boxdn A:=\Big(\frac{1}{k-\frac d2+1}\cod \extd+\frac{1}{k-\frac d2-1}\extd \cod + \J-2\P\Big)A=0\, ,
$$
is conformally invariant for $A\in \Gamma\ce^kM[k-\frac d2 +1]$ where the conformally invariant {\it Branson--Deser--Nepomechie operator} 
is a map $\Gamma\ce^k\big[k-\frac d2 +1\big]\to\Gamma\ce^k\big[k-\frac d2 -1\big]$.
 (When it is unclear which degree forms $\boxdn$ acts on, we will write $\boxdn^{(k)}$.)
This  conformally invariant generalization of both Maxwell's equations and the Yamabe equation seems to have been first uncovered in~\cite{Tomforms}
and was then independently uncovered in a physical context in~\cite{Deser:1983tm}: The Yamabe equation appears at degree~$k=0$.
The residues of the poles at $k=\frac{d}{2}\pm 1$ ($d\in2{\mathbb N}$) give the higher form Maxwell's equations in their standard potential form, or expressed in terms of a Hodge dual potential, respectively.

Finally, in even dimensions, Maxwell's equations in their curvature
form can be coupled conformally to a Proca-type equation to yield the
conformally invariant {\it coupled Proca--Maxwell} system of equations
$$
*\, \cod F= \big(\FL+ 2(\J-2\P)\big) A \, , \quad \extd F = 0 = \cod A\, ,
$$
where $A\in \Gamma\ce^{1+\frac d2}M[2]$ and $*F\in \Omega^{\frac d2}M$.
%For later use, we streamline the above set of equations in terms of a {\it Maxwell--Proca operator}
%$$
%\boxmp:
%$$
Of course, all of the above systems of equations enjoy a dual formulation 
obtained by applying the isomorphism $*:\ce^kM[w]\stackrel\cong{\ \rightarrow} \ce^{d-k}M[d+w-2k]$.

It is not difficult to write the curvature Maxwell, Branson--Deser--Nepomechie and coupled Proca--Maxwell systems as simple equations for tractor forms. We record that result in the Proposition below.
Before doing so,  for completeness  we record the basic  $\D\F=0$ equation for a definite choice of splitting $g\in c$ of $\F\in\ct^kM[w]$ 
\begin{equation}\label{DFsystem}
\left\{\scalebox{.92}{$
\begin{array}{ccl}
0&=&(d+2w-2)\big(\extd F^+ - (w+k) F\big)\, ,\\[1mm]
0&=& (d+2w-2) \, \extd F\, ,\\[1mm]
0&=& \big(\FL+(w+k-2)(\J-2\P)\big)F^+ -2\cod F +(d+2w)\big(\extd F^{+-}\!+(w+k-2)F^-\big)\,  ,\\[1mm]
0&=& [\J-2\P,\extd ]\, F^+ -\big(\FL+(w+k)(\J-2\P)\big)F-(d+2w) \, \extd F^{-}\, .
\end{array}$}
\right.
\end{equation}

\begin{proposition}
Let $\F\in \ct^k M$. Then the system of equations
$$
\D \F = 0 = \Ds \F = \Xs \F
$$
describes 
\begin{enumerate}[(i)]
\item the Branson--Deser--Nepomechie equation when $w=1-\frac d2$, $k\neq 1+\frac{d}{2}$,
\item the coupled Proca--Maxwell system when $w=1-\frac d2$, $k=1+\frac d2$ and $d\in2{\mathbb N}$.
\end{enumerate}
The system of equations 
$$
\D \F = \Ds \F = \X \F = \Xs \F = 0
$$
describes the curvature version of the higher form Maxwell's equations when $w=1-\frac d2$, $k=1+\frac d2$~($d\in2{\mathbb N}$).
\end{proposition}

\begin{proof}
Given the expressions for the exterior and interior tractor D- and canonical tractor operators for a choice of $g\in c$
explicated in Section~\ref{algebra}, the proof is elementary. For expedience, note that the kernels of $\Xs$ and $\Ds$ are 
spelled out in Lemma~\ref{West}. To complete the isomorphism between the tractor equations and their differential form
counterparts, note that we have  called the western slot $A$. Also, for the Proca--Maxwell system, the Maxwell curvature~$F$ is
the Hodge dual of the southern slot.
\end{proof}

\begin{remark}
For the case $k=1$, the above characterization of the Branson--Deser--Nepomechie equation is exactly that given in~\cite{Gover:2008pt}.
\end{remark}

\renewcommand{\ca}{{\mathcal A}}

It now remains only to explain the origin of the conformally invariant potential version of Maxwell's equations.
The key point here is the gauge invariance which generates a new solution $A+\extd \alpha$ from any given solution~$A$.
Since $\extd$ is nilpotent, this really amounts to a certain cohomology problem. Indeed, the system of equations~\nn{DFsystem} are also gauge invariant under 
$$
\F\sim \F + \D \ca\, .
$$
Hence the relevant problem is the cohomology of $\D$ with $\Gamma\ct^\bullet[\, . \, ]$ as the space of chains. 
This problem will also play an important {\it r\^ole}  in our study of Proca equations in Section~\ref{Procasect}.
Our first result establishes that for generic weights this cohomology is in fact trivial.

\newcommand{\ch}{{\mathcal H}}

\begin{proposition}\label{nocoho}
The cohomology of the differential complex
$$
\cdots 
\stackrel{\D}\longrightarrow \ \Gamma \ct^{k-1}M[w+1] \
\stackrel{\D}\longrightarrow \ \Gamma \ct^{k}M[w] \
\stackrel{\D}\longrightarrow \ \Gamma \ct^{k+1}M[w-1] \
\stackrel{\D}\longrightarrow\cdots\, ,
$$
denoted $\ch^kM[w]$ at $\Gamma \ct^{k}M[w]$, is trivial
whenever $w\neq -\frac d2, 1-\frac d2,-1-\frac d2, -k, -k+2$.
\end{proposition}

\begin{proof}
Consider $$
\frac{1}{(d+2w)(w+k)}\, \D\Xs\Big(1-\frac{2}{(d+2w+2)(w+k-2)}\, \X\Dts\Big) \, \F\, .
$$
We claim that if $\D\F=0$, the above quantity identically equals $\F$. The claim is easily verified using Proposition~\ref{XDalgebra}
to push $\D$ to the right (assuming $w\neq1-\frac d2$) where it annihilates $\F$. 
Therefore $\D\F=0$ $\Rightarrow \F=\D \A$ for some $\A$.
\end{proof}

To analyze further  the cohomology of the exterior Tho\-mas D-operator,
observe that acting with $\D$ we move on the $(w,k)$ plane along lines $w+k=\mbox{constant}$ as shown in Figure~\ref{Dcomplex}.
Similarly, the interior tractor D-operator $\Ds$ moves along lines $w-k=\mbox{constant}$.
In what follows we will not analyze the $\Ds$ cohomology because it can easily be obtained from that of $\D$ by Hodge duality. 

\begin{figure}
%\begin{center}
\scalebox{1}{
\begin{picture}(200,200)(0,-120)
 \thicklines
 \put(10,0){\vector(1,0){225}}
\put(10,0){\vector(0,1){70}}
\put(8.5,-10){$\vdots$}
\put(10,-12){\line(0,-1){180}}
\thinlines
\color{red}{\put(10,-34.5){\line(1,0){180}}}
\color{red}{\put(6,-61.5){\line(1,0){180}}}
\color{red}{\put(3,-89.5){\line(1,0){180}}}

\thinlines

\color{black}{\put(179,49){\line(0,-1){240}}}

\put(-5.5,73){$w$}
\put(226,-3){$k$}
\put(-9,-1){$\scriptscriptstyle0$}
\put(-9,55){$\scriptscriptstyle2$}
\put(-24,-181){$\scriptscriptstyle -d-2$}
\put(181,4){$\scriptscriptstyle d+2$}
\put(-19,-36){$\scriptscriptstyle 1-\frac d2$}
\put(-16,-62.5){$\scriptscriptstyle -\frac d2$}
\put(-24,-92){$\scriptscriptstyle -1-\frac d2$}

%dots
\put(115,-126){\rotatebox{-9}{$\ddots$}}
\put(54.5,-9){\rotatebox{-9}{$\ddots$}}
\put(42,3.5){\rotatebox{-9}{$\ddots$}}

\put(-2,-9){\rotatebox{-9}{$\ddots$}}
\put(167,-121.5){\rotatebox{-9}{$\ddots$}}

%D
\put(76,-72){\scalebox{.9}{\scalebox{.8}{$\D$}}}
\put(50,-46){\scalebox{.9}{\scalebox{.8}{$\D$}}}
\put(22.5,-18.5){\scalebox{.9}{\scalebox{.8}{$\D$}}}
\put(168,-164){\scalebox{.9}{\scalebox{.8}{$\D$}}}
\put(142,-138){\scalebox{.9}{\scalebox{.8}{$\D$}}}
\put(103,-99){\scalebox{.9}{\scalebox{.8}{$\D$}}}
%\put(96,-72){\scalebox{.9}{\scalebox{.8}{$\Ds$}}}
%\put(122,-46){\scalebox{.9}{\scalebox{.8}{$\Ds$}}}
%\put(147,-21){\scalebox{.9}{\scalebox{.8}{$\Ds$}}}
%\put(4,-164){\scalebox{.9}{\scalebox{.8}{$\Ds$}}}
%\put(30,-138){\scalebox{.9}{\scalebox{.8}{$\Ds$}}}
%\put(69,-99){\scalebox{.9}{\scalebox{.8}{$\Ds$}}}

\thicklines
%arrows1-1
\put(11.5,-11.5){\vector(1,-1){20}}
\put(38,-38){\vector(1,-1){20}}
\put(65,-65){\vector(1,-1){20}}
\put(94,-94){\vector(1,-1){20}}
\put(132,-132){\vector(1,-1){20}}
\put(159,-159){\vector(1,-1){20}}

\put(0,56){\vector(1,-1){20}}
\put(27,29){\vector(1,-1){20}}
\put(67.5,-11.5){\vector(1,-1){20}}
\put(94,-38){\vector(1,-1){20}}
\put(121,-65){\vector(1,-1){20}}
\put(148,-92){\vector(1,-1){20}}

\put(79.5,-37.5){\vector(1,-1){20}}
\put(106.5,-64.5){\vector(1,-1){20}}
\color{gray}{\put(0,42){\vector(1,-1){20}}
\put(53,-11){\vector(1,-1){20}}
\put(135,-93){\vector(1,-1){20}}
%\put(169,-131){\vector(1,-1){14}}
}
\put(155,-123){\rotatebox{-9}{$\ddots$}}
\put(36.5,-6){\rotatebox{-9}{$\ddots$}}
\put(26.5,5.5){\rotatebox{-9}{$\ddots$}}

%otherarrows
%\put(0,56){\vector(1,-1){20}}

%arrows-1-1
%\color{gray}{\put(167,-12){\vector(-1,-1){20}}
%\put(141,-38){\vector(-1,-1){20}}
%\put(114,-65){\vector(-1,-1){20}}
%\put(85,-94){\vector(-1,-1){20}}
%\put(47,-132){\vector(-1,-1){20}}
%\put(20,-159){\vector(-1,-1){20}}}

\end{picture}
}
%\end{center}

\vspace{3cm}

\caption{\label{Dcomplex}
In this picture we represent the complex for the exterior and interior tractor D-operators; 
the lower and upper diagonal line of arrows describe the complexes along $w+k=0$ and $w+k=2$ of Proposition~\ref*{w+k=0}
 which have non-empty cohomology. 
The central diagonal line corresponds to the case of generic values of $w+k$ of Proposition~\protect\ref*{w+kgeneric}, which only has non-trivial cohomology 
if the 
 three red horizontal lines  at the values $w=1-\frac d2, -\frac d2, -1-\frac d2$ are traversed.
}
\end{figure}
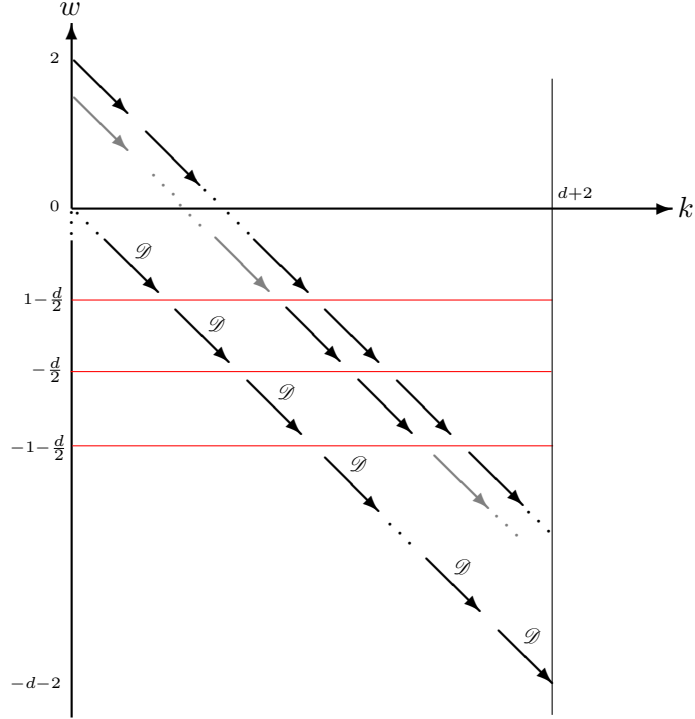

Proposition~\ref{nocoho} reveals that we may expect  non-trivial cohomology only along the lines $w+k=0,2$ or when a
 $w+k$ line meets the  weights $w=1-\frac d2, -\frac d2, -1-\frac d2$.  
This can also be seen by examining the system of equations~\nn{DFsystem} which one has to solve for $\D$ to be closed, 
as well as considering the same system after replacing $w\to w+1$, $k\to k-1$ in order to analyze when a $\D$-closed tractor form
is $\D$-exact. First we give a result for general weights.

\begin{proposition}\label{w+kgeneric} When $w+k\neq0,2$ the cohomology $\ch^kM[w]$ of the operator $\D$ can only be non-empty for $w= 1-\frac d2, -\frac d2,-1-\frac d2$.
Moreover:
 \begin{enumerate}[(i)]
  \item$\ch^{k}M[1-\frac d2]\cong\operatorname{ker} \boxdn^{(k)}$,\\[-2mm]
      \item$\ch^{k+1}M[-\frac d2]\cong\operatorname{ker} \boxdn^{(k)}\oplus \operatorname{coker} \boxdn^{(k)}$,\\[-2mm]
  \item$\ch^{k+2}M[-1-\frac d2]\cong\operatorname{coker} \boxdn^{(k)}$.
\end{enumerate} 
 \end{proposition}
 \begin{proof}
 
 The method of proof for this and the other Proposition~\ref{w+k=0}  for  the cohomology of~$\D$ is the same:
One examines the system of equations~\ref{DFsystem} first at the weight and degree $(w,k)$ of interest in order to study the closure
condition, and then again at $(w+1,k-1)$ to control exactness. The general pattern for the closure conditions is that the first and third of
these equations allow the western and southern slots $F$ and $F^-$ to be determined in terms of their northern and eastern counterparts, $F^+,F^{+-}$, {\it unless} the weight and degree conspire to give them a zero coefficient. For exactness, at $(w+1,k-1)$ the right hand sides of the system of equations~\ref{DFsystem} show that the northern and eastern slots of $\F\in \ker\D$ are algebraically cohomologous to zero, again up to conspiracies between weights and degrees.

For the case at hand there are three conspiracies: (i) For $\ch^k\big[1-\frac d2\big]$, the northern and eastern slots are still cohomologous to zero, but 
only the southern slot can be eliminated by the closure condition. (ii) At $\ch^{k+1}\big[-\frac d2\big]$ only the  eastern slot is  cohomologous to zero
and only the western slot can be eliminated by the closure condition. (iii) For $\ch^{k+2}\big[-1-\frac d2\big]$ the northern slot is again cohomologous to zero and both the western and southern slots are removed by the closure requirement. Schematically then, a complex $0\to {\rm W}\to ({\rm N,S})\to {\rm E}\to 0$ results. It is not difficult to compute the differentials; the result is depicted below 

\scalebox{.88}{
$
%\begin{align*}
%0\longrightarrow
\Gamma\ce^{k}M\big[k-\frac d2 +1\big] 
\xrightarrow{\raisebox{.2cm}{\scalebox{.85}{$\begin{pmatrix}0\\\boxdn\end{pmatrix}$}}}\Gamma\ce^{k}M\big[k-\frac d2 +1\big] \oplus \Gamma\ce^{k}M[k-\frac d2-1] 
\xrightarrow{\raisebox{.2cm}{\scalebox{.85}{$\Big(\boxdn\!\quad\!0\Big)$}}}
\Gamma\ce^{k}M\big[k-\frac d2-1\big]
% \longrightarrow 0
%\end{align*}
$}

 \end{proof}

This leaves the cases $w+k=0,2$ which are closely related to both de Rham cohomology $H^k M$ and the detour complexes that will appear
 later in the Article in a holographic context.

\begin{proposition}\label{w+k=0} When $w+k=0,2$ but $w\neq 1-\frac d2,-\frac d2,-1-\frac d2$
$$
\left\{
\begin{array}{cclc}
\ch^kM[-k]&\!\cong\!& H^{k-1}M\oplus H^k M \, ,& w=-k\, ,\\[1mm]
 \ch^{k}M[2-k]&\!\cong\!& H^{k-2}M\oplus H^{k-1} M\, ,& w=2-k\, .
\end{array}
\right.
$$
Moreover, the differential complexes 
$$\xrightarrow{\D}\,\Gamma\ct^{\frac d2-1}M\big[1-\frac d2\big]\,\xrightarrow{\D}\,
\Gamma\ct^{\frac d2}M\big[-\frac d2\big]\,\xrightarrow{\D}\,\Gamma\ct^{\frac d2+1}M\big[-1-\frac d2\big]\,\xrightarrow{\D}\ ,$$
and
$$\xrightarrow{\D}\,\Gamma\ct^{\frac d2+1}M\big[1-\frac d2\big]\xrightarrow{\D}
\Gamma\ct^{\frac d2+2}M\big[2-\frac d2\big]\,\xrightarrow{\D}\,\Gamma\ct^{\frac d2+3}M\big[-1-\frac d2\big]\,\xrightarrow{\D}\ ,$$
corresponding to $w= 1-\frac d2,-\frac d2,-1-\frac d2$ and $w+k=0,2$, respectively, are equivalent to the following differential complexes
\begin{align*}
%\cdots
\xrightarrow{\scalebox{.5}{$\begin{pmatrix}-2\extd & 0 \\ 0 & 2 \extd\end{pmatrix}$}}\scalebox{.8}{$\Omega^{\frac{d}{2}-2}M\oplus\Omega^{\frac{d}{2}-1}M$}&
\xrightarrow{\scalebox{.5}{$\begin{pmatrix}0&0\\0&0\\ \FL-2(\J-2\P)& \cod \end{pmatrix} \extd$}}
\scalebox{.8}{$\Omega^{\frac{d}{2}-1}M\oplus\Omega^{\frac{d}{2}}M\oplus\Gamma\ce^{\frac d 2-1}M[-2] $}\\[2mm]
&\xrightarrow{\scalebox{.5}{$\begin{pmatrix}2\extd&0&0\\0&-2\extd&0\\ \FL-2(\J-2\P)&\cod&0\end{pmatrix}$}}
%\stackrel{\raisebox{-.1cm}{\scalebox{.5}{$\begin{pmatrix}2\extd&0&0\\0&-2\extd&0\\ \Delta-2(\J-2\P)&-2\Delta&0\\ [ \J-2\P,\extd ]&-\Delta&0\end{pmatrix}$}}}{----\longrightarrow }
%\scalebox{.6}{$
%\Omega^{\frac{d}{2}}M\oplus\Omega^{\frac{d}{2}+1}M\oplus\ce^{\frac d 2-1}M[-2]\oplus \ce^{\frac d 2}M[-2]$}$$
\scalebox{.8}{$
\Omega^{\frac{d}{2}}M\oplus\Omega^{\frac{d}{2}+1}M\oplus\Gamma\ce^{\frac d 2-1}M[-2]$}
\xrightarrow{\scalebox{.5}{$\begin{pmatrix}4\extd&0&0\\0&-4\extd&0\end{pmatrix}$}}\ , 
%\scalebox{.8}{$
%\Omega^{\frac{d}{2}+1}M\oplus\Omega^{\frac{d}{2}+2}M$ }\xrightarrow{\scalebox{.5}{$\begin{pmatrix}6\extd & 0 \\ 0 & -6 \extd\end{pmatrix}$}}\cdots
%
\end{align*}
and 
\begin{align*}
%\xrightarrow{\scalebox{.5}{$\begin{pmatrix}6\extd & 0 \\ 0 & -6 \extd\end{pmatrix}$}}\scalebox{.8}{$\Omega^{\frac{d}{2}-2}M\oplus\Omega^{\frac{d}{2}-1}M$}
\xrightarrow{\scalebox{.5}{$\begin{pmatrix}0&0\\4\extd&0\\0&-4\extd \end{pmatrix}$}}\scalebox{.8}{$\Gamma\ce^{\frac d 2+1}M[2] \oplus\Omega^{\frac{d}{2}-1}M\oplus\Omega^{\frac{d}{2}}M$}&
\xrightarrow{\scalebox{.5}{$\begin{pmatrix}0&0&0\\\cod&2\extd&0\\ \FL+2(\J-2\P)&0&-2\extd\end{pmatrix}$}}
%\stackrel{\raisebox{-.1cm}{\scalebox{.5}{$\begin{pmatrix}2\extd&0&0\\0&-2\extd&0\\ \Delta-2(\J-2\P)&-2\Delta&0\\ [ \J-2\P,\extd ]&-\Delta&0\end{pmatrix}$}}}{----\longrightarrow }
%\scalebox{.6}{$
%\Omega^{\frac{d}{2}}M\oplus\Omega^{\frac{d}{2}+1}M\oplus\ce^{\frac d 2-1}M[-2]\oplus \ce^{\frac d 2}M[-2]$}$$
\scalebox{.8}{$\Gamma\ce^{\frac d 2+1}M[2]\oplus
\Omega^{\frac{d}{2}}M\oplus\Omega^{\frac{d}{2}+1}M$}\\[2mm]
&\xrightarrow{\scalebox{.5}{$\extd\begin{pmatrix}\cod&0&0\\\FL+2(\J-2\P)&0&0\end{pmatrix}$}}
\scalebox{.8}{$
\Omega^{\frac{d}{2}+1}M\oplus\Omega^{\frac{d}{2}+2}M$ }\xrightarrow{\scalebox{.5}{$\begin{pmatrix}-2\extd & 0 \\ 0 & 2 \extd\end{pmatrix}$}}\, .
\end{align*}
% the cohomology of the operator $\D$ reduces to a pair of de Rham cohomology for $k\neq \frac d2-1, \frac d2, \frac d2+1$. Moreover:
% \begin{enumerate}M[(i)]
%  \item$\ch^{\frac d2-1}\big[1-\frac d2\big]$ at $\Gamma \ct^{\frac d2-1}M\big[1-\frac d 2\big]$ is non empty on the space of solution of the dual- Maxwell equation coupled to a co-exact current
%  \item $\ch^{\frac d2}\big[-\frac d2\big]$ at $\Gamma \ct^{\frac d2}M\big[-\frac d 2\big]$ yields the Kernel of the dual Proca system and the cokernel of the Mawell system
%  \item $\ch^{\frac d2+1}\big[-1-\frac d2\big]$ at $\Gamma \ct^{1+\frac d2}M\big[-1-\frac d 2\big]$ yields the cokernel of the dual Proca system.
% \end{enumerate} 
 \end{proposition}
\begin{proof}
Here the proof follows exactly the same methodology as for Proposition~\ref{w+kgeneric}.
\end{proof}

\begin{remark} Although the above differential complexes appear novel, they are closely related to the systems studied already.
For example, the second differential in the first diagram is a prolongation of  the Maxwell operator $\cod \extd$. The third differential is equivalent to the coupled Proca--Maxwell system (upon Hodge dualizing $A$).
\end{remark}

To relate the cohomology of the exterior tractor D-operator~$\D$ to the conformally invariant Maxwell operator and its associated detour complex, we need to refine the space of chains. This is achieved via the following simple observation.
\begin{lemma}\label{maplem}
There are  well-defined  maps 
$$
\left\{\!
\begin{array}{llll}
\scalebox{.9}{$\ker(\Ds,\Xs)\subset\Gamma \ct^k M[-k]$}&\!\longrightarrow\!&
\scalebox{.9}{$ \ker(\Ds,\Xs)\subset\Gamma \ct^{k+1} M[-k-1]\, ,$}& k\neq 1-\frac d2,-\frac d2\, ,\\[2mm]
\scalebox{.9}{$\ker(\Ds,\Xs)\subset\Gamma \ct^{\frac d2-1} M\big[1-\frac d2\big]$}\!\!\!&\!\longrightarrow\!& 
\scalebox{.9}{$\ker(\X,\Xs)\subset\Gamma \ct^{\frac d2} M\big[-\frac d2\big]\, ,$} & d\in 2{\mathbb N}\, ,\\[2mm]
\scalebox{.9}{$\ker(\X,\Xs)\subset\Gamma \ct^{\frac d2} M[-\frac d2]$}&\!\longrightarrow\!& 
\scalebox{.9}{$\ker(\Ds,\Xs)\subset\Gamma \ct^{\frac d2+1} M\big[-1-\frac d2\big]\, ,$} & d\in 2{\mathbb N}\, ,
\end{array}
\right.
$$
given by
$$ \F\longmapsto \D \F\, .$$
\end{lemma}
\begin{proof}
To verify that at the weight $w=-k\neq -\frac d2$ elements of the kernel of $(\Ds,\Xs)$ are mapped by $\D$  again to $\ker(\Ds,\Xs)$,
it suffices to examine the third identity of Proposition~\ref{XDalgebra}. At $w=-k=1-\frac d2$ the $\ker \Ds$ condition on the image of the map $\D$ turns out to be an empty requirement but is augmented by $\ker \X$ which follows from the  first line of Proposition~\ref{XDalgebra}.
To see that at $w=-k=-\frac d2$, the map $\D$ has image contained in $\ker\Xs$ one uses the second line of Proposition~\ref{XDalgebra}
along with the fact that the domain is then taken to be  $\ker(\X,\Xs)$.
\end{proof}
\begin{remark}\label{remd}
By the western Lemma~\ref{West},  $\ker(\Ds,\Xs)\subset\Gamma \ct^k M[-k]\cong \Omega^k M$
so long as $k\neq \frac d2$. Hence, $\D$ induces a map $\Omega^k M\to \Omega^{k+1}M$ when $k\neq \frac d2-1,\frac d2$. 
From the display~\nn{D}, it follows that this map is the exterior derivative~$\extd$.
\end{remark}

\begin{proposition}\label{half de Rham}
Let $d\in 2{\mathbb N}$. Then the differential complex
\begin{align*}
\cdots & \xrightarrow{\D}\  \scalebox{.9}{$\ker(\Ds,\Xs)\subset\Gamma \ct^{\frac d2-2} M\big[2-\frac d2\big]$}
\\
&  \xrightarrow{\D}\   \scalebox{.9}{$\ker(\Ds,\Xs)\subset\Gamma \ct^{\frac d2-1} M\big[1-\frac d2\big]$}
\ \xrightarrow{\D}\ 
\scalebox{.9}{$\ker(\X,\Xs)\subset\Gamma \ct^{\frac d2} M\big[-\frac d2\big]$}
\\&\xrightarrow{\D}\ 
\scalebox{.9}{$\ker(\Ds,\Xs)\subset\Gamma \ct^{\frac d2+1} M\big[-1-\frac d2\big]$}
\ \xrightarrow{\D}\ \scalebox{.9}{$\ker(\Ds,\Xs)\subset\Gamma \ct^{\frac d2+2} M\big[-2-\frac d2\big]$}
\ \xrightarrow{\D}\ \cdots
\end{align*}
is equivalent to
$$
\stackrel{\extd}\longrightarrow \Omega^{\frac d2-2} M \stackrel{\extd}\longrightarrow
\Omega^{\frac{d}{2}-1}M \xrightarrow{\ \cod\extd\ }
\Gamma\ce^{\frac{d}{2}-1}M[-2]
\stackrel{0}\longrightarrow \Omega^{\frac d2 +1} M
\stackrel{\extd}\longrightarrow \Omega^{\frac d2 +2} M
\stackrel{\extd}\longrightarrow
$$
\end{proposition}

\begin{remark}
In odd dimensions, the ``detour'' at $w=-k=1-\frac d2$ is avoided and one simply has an equivalence between the cohomology of $\D$ acting on $\ker(\Ds,\Xs)$ for $w+k=0$ 
$$
\cdots \xrightarrow{\D}\  \scalebox{.9}{$\ker(\Ds,\Xs)\subset\Gamma \ct^{k} M[-k]$}\ \xrightarrow{\D}\cdots
 $$
and de Rham cohomology
$$
\stackrel{\extd }\longrightarrow \Omega^{k} M
\stackrel{\extd}\longrightarrow
\, .
$$
\end{remark}

\begin{proof}
This follows directly from Lemma~\ref{maplem}, its accompanying Remark~\ref{remd} and a computation of~$\D$ at weights $w=-k=1-\frac d2,-\frac d2$
for some $g\in c$ using~\nn{D}.
\end{proof}

Rather than connecting the de Rham complex via the detour operator $\cod\extd$ followed by the zero map to the de Rham complex
again, as in Proposition~\ref{half de Rham}, a canonical man\oe uvre is to continue on with the dual de Rham complex. Pictorially this gives
the {\it Maxwell detour complex}
 \begin{equation}\label{detourmax}
\stackrel{\extd}\longrightarrow \Omega^\bullet M \stackrel{\extd} \longrightarrow\cdots
\stackrel{\extd}\longrightarrow\Omega^{\frac{d}{2}-1}M \xrightarrow{\ \cod\extd\ }
\Gamma\ce^{\frac{d}{2}-1}M[-2]
\stackrel{\cod}\longrightarrow\cdots \stackrel{\cod}\longrightarrow \Gamma\ce^\bullet M [\, .\, ]\stackrel{\cod}\longrightarrow\, ,
\end{equation}
(where the chains in the outgoing dual de Rham complex belong to $\Gamma\ce^kM[2k-d]$).
This is an important but simplest case of a family of conformally invariant differential detour complexes~\cite{BrGoEM,BrGodeRham,AG-BGops} whose 
study we take up again in Section~\ref{OBSTQD}.

We complete this Section by showing how the Maxwell detour complex arises in the current setting.
First we rely on a Corollary of Proposition~\ref{maplem} and the southern and eastern Lemmas~\ref{southlemma},~\ref{East}.
\begin{corollary}
Let $d\in2{\mathbb N}$. Then there is a well-defined, conformally invariant,  canonical ``detour'' map
$$
 \ker(\Ds,\Xs)\subset \Gamma\ct^{\frac d2-1} M\big[1-\frac d2\big]\longrightarrow \ker(\D,\X)\subset\Gamma \ct^{\frac d2+1} M\big[-1-\frac d2\big]\, ,$$
given by the composition of maps
$$
q^{\phantom{-1}\!\!}_E\circ q_S^{-1}\circ\D\, .
$$
Moreover,
$$
\Ds\circ q^{\phantom{-1}\!\!}_E\circ q_S^{-1}\circ\D=0\, .
$$
\end{corollary}
\begin{proof}
Only the statement that the composition of the interior tractor D-operator and the detour map vanishes requires further elaboration:
The range of the differential $\cod \extd$ in Proposition~\ref{half de Rham} at $\Omega^{\frac d2} M$ is mapped to the  eastern slot
of a section of $\ct^{\frac d2+1} M\big[-1-\frac d2\big]$. Then the eastern Lemma~\ref{East} implies that a dual version of Remark~\ref{remd}
holds: the action of $\Ds$ on $\ker(\D,\X)\subset \Gamma\ct^{k}M[-d-2+k]$ induces the  map $$\cod :\Gamma\ce^kM[2k-d]\to \Gamma\ce^{k-1}M[2k-d-2]\, ,$$ so long as $k\neq \frac d2 +2$
(this is the point in the $(w,k)$-plane dual to the one where the map $q^{\phantom{-1}\!\!}_E q_S^{-1}$ was needed to replace the $\D$ operator). The proof is now complete since $\cod^2=0$.
\end{proof}

\begin{remark}
Acting on weight $-\frac d2$, degree~$\frac d2$ southern tractors (in even dimensions), the composition of operators $q^{\phantom{-1}\!\!}_E\circ q_S^{-1}=-2\, \wD$, as can be easily verified by a direct computation.
\end{remark}

We have therefore by now established the following result (depicted schematically in Figure~\ref{detourgraph})
\begin{proposition}\label{Maxwelldetour}
Let $d\in 2{\mathbb N}$. Then the differential complex
\begin{align*}
\cdots & \xrightarrow{\D}\  \scalebox{.9}{$\ker(\Ds,\Xs)\subset\Gamma \ct^{\frac d2-2} M\big[2-\frac d2\big]$}
\\
&  \xrightarrow{\D}\   \scalebox{.9}{$\ker(\Ds,\Xs)\subset\Gamma \ct^{\frac d2-1} M\big[1-\frac d2\big]$}
\ \xrightarrow{q_E^{{}}q_S^{-1} \D}\ 
\scalebox{.9}{$\ker(\D,\X)\subset\Gamma \ct^{\frac d2+1} M\big[-1-\frac d2\big]$}
\\&\xrightarrow{\Ds}\ 
\scalebox{.9}{$\ker(\D,\X)\subset\Gamma \ct^{\frac d2} M\big[-2-\frac d2\big]$}
\ \xrightarrow{\Ds}\ \scalebox{.9}{$\ker(\D,\X)\subset\Gamma \ct^{\frac d2-1} M\big[-3-\frac d2\big]$}
\ \xrightarrow{\Ds}\ \cdots
\end{align*}
is equivalent to the Maxwell detour complex (of Equation~\nn{detourmax}).
\end{proposition}

\begin{figure}\begin{center}
\scalebox{1}{
\begin{picture}(200,200)(0,-120)
 \thicklines
 {\put(0,0){\vector(1,0){225}}}
\put(0,-200){\vector(0,1){230}}
\thinlines
\color{black}{\put(179,20){\line(0,-1){220}}}
%\color{black}{\put(42,0){\line(0,-1){3}}}

\put(-4,32){$w$}
\put(226,-3){$k$}
\put(-24,-180){$\scriptscriptstyle -d-2$}
\put(181,4){$\scriptscriptstyle d+2$}
\put(31,4){$\scriptscriptstyle \frac d2-1$}
\put(-18,-40){$\scriptscriptstyle 1-\frac d2$}

%\put(2,-20){\scalebox{.7}{$\ker(\Ds,\Xs)$}}
\put(14,-39){\scalebox{.8}{$\ker(\Ds,\Xs)$}}

\put(70,-98){\scalebox{.8}{$\ker(\D,\X)$}}

\put(43,-127){\scalebox{.8}{$\ker(\D,\X)$}}
\put(61,-56){\scalebox{.8}{$q_E^{{\color{white}{-1}}}\! q_S^{-1}\D$}}
\put(22,-18){\scalebox{.8}{$\D$}}

\put(74,-114){\scalebox{.8}{$\Ds$}}
\put(46,-143){\scalebox{.8}{$\Ds$}}
\put(0,-10){\rotatebox{-9}{$\ddots$}}
\put(2,-165){\rotatebox{-99}{$\ddots$}}
\thicklines
\put(14,-14){\vector(1,-1){15}}
\put(43,-43){\vector(1,-1){46}}

\put(48,-131){\vector(-1,-1){15}}
\put(77,-102){\vector(-1,-1){15}}

\color{gray}
\put(102,-102){\vector(1,-1){15}}
\put(132,-132){\vector(1,-1){15}}
\put(104,-127){\scalebox{.8}{$\ker(\Ds,\Xs)$}}
\put(164.5,-175){\rotatebox{-9}{$\ddots$}}

\put(105,-74){\vector(-1,-1){15}}
\put(136,-43){\vector(-1,-1){15}}
\put(104,-69){\scalebox{.8}{$\ker(\D,\X)$}}

\put(112,-108){\scalebox{.8}{$\D$}}

\put(143,-139){\scalebox{.8}{$\D$}}
\put(103,-85){\scalebox{.8}{$\Ds$}}
\put(134,-54){\scalebox{.8}{$\Ds$}}
\put(162,-5){\rotatebox{-99}{$\ddots$}}

\end{picture}
}
\vspace{3.5cm}
\caption{\label{detourgraph}
The Maxwell detour complex from exterior and interior Thomas-D cohomologies is schematically
$$
\cdots\xrightarrow{\ \D_{{}} \ } {\rm W}\xrightarrow{\ \D_{{}} \ } {\rm W}
 \xrightarrow{q_{_{\scriptscriptstyle E}} q_S^{-1} \D}
 {\rm E} \xrightarrow{\Ds_{{}}} {\rm E} \xrightarrow{\Ds_{{}}}\cdots\qquad\qquad\qquad\qquad
$$
for degrees and weights as depicted above.
}
\end{center}

\end{figure}
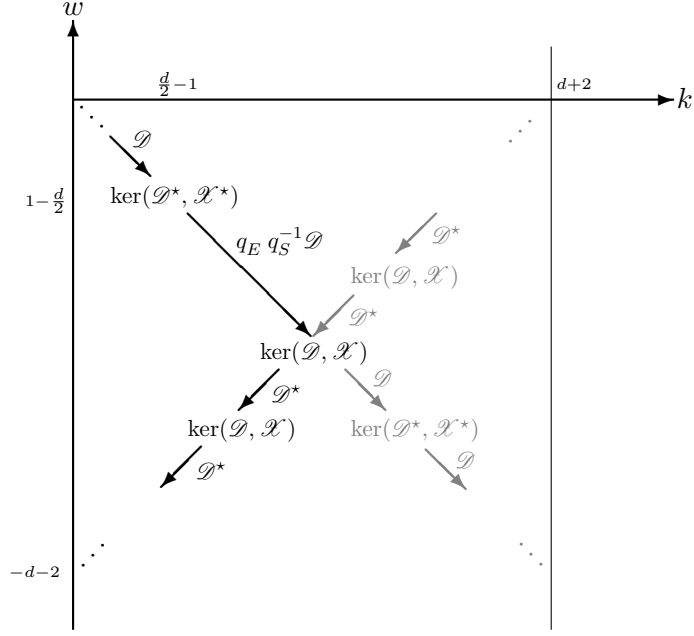

\section{The exterior calculus of scale}
\label{extcalcsc}

We now come to a central point of our development, namely that there is a canonical way to introduce a (generalized) scale
into the conformal calculus and algebra. Most of the structure is available in the general almost Riemmanian setting, so we
treat this first before refining to  Poincar\'e--Einstein structures.

Let~$(M,c,\sigma)$ be an almost Riemannian structure.
The
scale tractor defines canonical maps on $\Gamma\ct^\bullet M[\, . \,]$ by exterior and interior multiplication
$$
\I:=\varepsilon(I):\Gamma\ct^k M[w]\rightarrow \ct^{k+1} M[w]\, ,\qquad \Is:=\iota(I):\Gamma\ct^k M[w]\rightarrow \ct^{k-1} M[w]\, .
$$
We may explicate these operators for a choice of $g\in c$ by
\begin{equation}\label{Iext}
\I
\stackrel{g}=
\begin{pmatrix}
-\varepsilon(n)&\sigma&0&0\\
0&\varepsilon(n)&0&0\\
-\rho&0&\varepsilon(n)&\sigma\\
0&\rho&0&-\varepsilon(n)
\end{pmatrix}\, ,\qquad
{\mathscr I}^\star\stackrel{g}=
\begin{pmatrix}
-\iota(n)&0&-\sigma&0\\
\rho&\iota(n)&0&\sigma\\
0&0&\iota(n)&0\\
0&0&\rho&-\iota(n)
\end{pmatrix}\, ,
\end{equation}
where $I_A\stackrel{g}{:=}(\rho,n,\sigma)$. From equations~\nn{Dform} and~\nn{scaletractor} we have
$$
n=\nabla \sigma\, ,\qquad \rho=-\frac 1d \big(\Delta^g+\J\big)\sigma\, .
$$

The following identities result trivially from standard properties of interior and exterior multiplication
as well as the definition of the maps $\I$ and $\Is$
\begin{center}
%\shabox{
\begin{tabular}{c}
$\{\I,\X\}=0=\{\Is,\Xs\}\, ,$\\[1mm]
$\{\Is,\X\}=\sigma=\{\I,\Xs\}\, ,$\\[1mm]
$[\N,\I]=\I\, ,  \quad[h,\I]=0=[h,\Is]\, ,\quad [\N,\Is]=-\Is$\, .
\end{tabular}
%}
\end{center}

 As always, we denote by $x$ the map
$\Gamma\ct^{k}M[w]\to\Gamma\ct^{k}M[w-1]$ obtained by multiplying by the
scale~$\sigma$; thus
 $x=\{\Is,\X\}=\{\I,\Xs\}\, .$
%Also, at weights for which the operation of dividing by $h(h-2)$ is legal, the left hand sides of the
%last four identities can be converted to commutators in terms of composite operators
%$(\frac 1h \D, \frac 1h \Ds,\X\frac 1h,\Xs\frac 1h)$. For example $[\frac 1h \D,x]=2\X\frac 1{h^2} y+\I$.\edz{A: Our firend $\Dt$ could help here...}

Finally we give identities that include the exterior and interior Thomas D-operators in our calculus.

\begin{proposition}\label{Iformalg}
Let $I^A$ be a scale tractor for an almost Riemannian conformal structure~$(M,c,\sigma)$,
and define the operator $y:=-I_A \slashed D{}^A$ that maps $\Gamma\ct^kM[w]\to\Gamma\ct^kM[w-1]$. Then, for $d\neq 4$, the following operator identities hold.
\begin{center}
%\shabox{
\begin{tabular}{c}
$[h,x]=2\, x\, ,\quad$ $[x,y]=h\, ,\quad$ $[h,y]=-2\, y\, ,$
\\[3mm]
$
(h-2)\, \Ds x - h\  x \Ds \ =\ 2\, \Xs y +h(h-2)\Is
$,
\\[2mm]
$
(h-2)\, \D x - h\  x \D \ =\  2\, \X y +h(h-2)\I
$,
\\[3mm]
$
(h-2)\, y\X-h\, \X y\ =\ 2\,x\D-h(h-2)\I
$,
\\[2mm]
$
(h-2)\, y\Xs-h\, \Xs y\ =\ 2\, x\Ds-h(h-2)\Is
$.
\end{tabular}
%}
\end{center}
When $d=4$, the above identities hold for any almost Einstein structure.
\end{proposition}
\begin{proof}
The first three identities were proven already in~\cite{GWasym}. The remaining identities are obtained contracting equation~\nn{genrel}
with the scale tractor on one of its indices and then performing either exterior or interior multiplication with the other. For the $d=4$ almost Einstein case, the same proof applies using also Remark~\ref{Wtilde} and the proof of Proposition~\ref{Yform}.
\end{proof}

\subsection{Poincar\'e--Einstein structures} \label{solnalgform}

We now specialise to Poincar\'e--Einstein structures.
For concreteness, we recall some basic definitions. A Riemannian metric $g^o$ on the interior $M^+$ of a compact manifold $M$ with boundary $\Sigma:=\partial M$
is said to be conformally compact if it extends to~$\Sigma$ by
$
g=r^2 g^o
$,
with~$g$ non-degenerate up to~$\Sigma$ equaling the zero locus of a defining function~$r$; that is~$\Sigma$ is the zero locus ${\mathcal Z}(r)$
and $\extd r|_\Sigma\neq 0$. 
If the normal to $\Sigma$ is nowhere null, then $g$  determines a conformal structure~$c_{_\Sigma}$. In this case~$(\Sigma,c_{_\Sigma})$ 
is called the conformal infinity of~$M^+$.
If the defining function obeys
$$
|\extd r|^2_g=1\, ,
$$
along~$\Sigma$, the sectional curvatures of $\g$ tend to $-1$ at infinity and the structure is
said to be asymptotically hyperbolic (AH)~\cite{Ma-hodge}.

Tractor calculus enables a treatment of   any
  conformally compact structure~\cite{Goal}.
%For~$\si$ a defining scale, on any conformal manifold, 
%we call
%\begin{equation}\label{scaletractor}
%I^A:=\frac{1}{d}D^A \si
%\end{equation}
%the corresponding {\em scale tractor}. Note that $\si=X^AI_A$. 
A very strong indication that  conformal geometries and their tractor treatment is fruitful  for the study of physical models, is their strong predilection for Einstein metrics~\cite{Sasaki}, as partly captured by the following result~\cite{GoNur}. 
\begin{theorem}\label{cein}
  On a conformal manifold $(M,c)$  there is a 1-1
  correspondence between conformal scales $\si\in \Gamma\ce M[1]$, such
  that $g^\si=\si^{-2}\bg$ is Einstein, and parallel standard tractors
  $I\in \Gamma \ct M$ with the property that $X_A I^A$ is nowhere vanishing. The
  mapping from Einstein scales to parallel tractors is given by
  $\si\mapsto \frac{1}{d}D^A \si$ while the inverse is $I^A \mapsto
  X^A I_A$.
\end{theorem} 
 In the above, $X^A\in \ct^AM[1]$ is  the {\it canonical tractor}--a distinguished invariant tractor; see \hyperlink{canonical tractor}{Section~\ref{algebra}} for further details. The statement of the Theorem  is easily verified
using~\nn{trconn}, or may be viewed as an easy consequence of the {\em
  definition} of the tractor connection from~\cite{BEG}. 
  
 For concreteness and later use, we explicate in tensor terms the parallel conditions $$\nabla^\ct I^A=0$$ for  the scale tractor
for some $g\in c$:
\begin{equation}\label{almostEinstein}
\left\{
\begin{array}{ccl}
\nabla_a \sigma &=&\ \  n_a\, ,\\[1mm]
\nabla_a n_b & =& \! -\sigma \Rho_{ab} -\rho g_{ab}\, ,\\[1mm]
\nabla_a \rho &=&\ \ \Rho_{ab} n^b\, .
\end{array}\right.
\end{equation}

 In light of the above Theorem and in line with~\cite{GoIP}, we will  say that a conformal manifold $(M,c)$, is {\it almost Einstein} if it is equipped with a non-zero parallel standard tractor $I$. This notion slightly enlarges the standard Einstein condition. Indeed, from the Theorem above it follows that the \hypertarget{scale}{{\it defining scale}}~$\sigma$ is non-zero on an open dense set~\cite{GoNur}. Moreover, if non-empty, the zero locus of the defining scale~$\sigma$
 is a conformal infinity. {\it I.e., the almost Einstein condition extends the standard Einstein one to describe manifolds with a conformal infinity.} 
 This boundary structure is precisely the one required to study a wide range of physical applications. Let us spell out some pertinent details:

%A Riemannian metric $g^o$ on the interior $M^+$ of a compact manifold $M$ with boundary~$\Sigma:=\partial M$
%is said to be {\it conformally compact} if it extends to~$\Sigma$ by
%$$
%g=r^2 g^o\, ,
%$$ 
%with $g$ non-degenerate up to~$\Sigma$ equaling the zero locus of the so-called {\it defining function}~$r$ of~$\Sigma$ subject to $\extd r|_\Sigma\neq 0$. If $\Sigma$ does not possess null directions then $g$, independently of~$r$, determines a conformal structure~$c_{_\Sigma}$. In this case
%$(\Sigma,c_{_\Sigma})$ is called the {\it conformal infinity} of~$M^+$.
%
%If the defining function obeys
%$$
%|\extd r|^2_g=1\, ,
%$$
%along~$M$, the sectional curvatures of $\g$ tend to $-1$ at infinity and the structure is
%said to be {\em asymptotically hyperbolic (AH)}~\cite{Ma-hodge}. 

Recall that  an AH manifold
which  is Einstein in its interior is called Poincar\'e--Einstein.  
 Following~\cite{GoPrague} and~\cite{GoIP}, this   fits precisely  in the almost  Einstein picture and provides the first
 exterior identity specialized to this setting.
\begin{proposition}
  A Poincar\'e--Einstein manifold is an almost
  Einstein manifold $(M,c,\sigma)$ with boundary $\Sigma$ equaling the zero locus of~$\sigma$
   such that $I^2:=I_A I^A=1$; thus
\begin{center}
%\shabox{
$
\{ \Is,\I\}=1\, .
$
%}
\end{center}
\end{proposition}
In view of this observation, it makes sense to treat the almost Einstein setting generally.

\subsubsection{Exterior calculus of almost Einstein scales}

%We will from here on in  focus on the case where $(M,c,\sigma)$ is a Poincar\'e--Einstein  structure,
%which for simplicity is assumed to be smooth to the boundary,
%although many of our results apply also to the more general  almost Einstein case. 
Here we develop exterior identities
extending the solution generating algebra~\nn{sl2} to the almost Einstein setting.

Using that $I$ is parallel we come to the next exterior identities.
\begin{proposition}\label{yproof}
If $(M,c,\sigma)$ is almost Einstein, then
\begin{center}
%\shabox{
\begin{tabular}{c}
$\{\I,\D\}=0=\{\Is,\Ds\}\, .$
\end{tabular}
%}
\end{center}
\end{proposition}
\begin{proof}
 This can easily be directly verified by anticommuting the 
matrix expressions for the \hyperlink{exterior D}{exterior} and \hyperlink{interior D}{interior} Thomas D-operators of Section~\ref{algebra}
with those for the exterior and interior scale tractors~$\I$ and $\Is$ in Equation~\nn{Iext}, 
employing the almost Einstein identities~\nn{almostEinstein}.
A slicker argument is to note that since $I$ is parallel
$$
[D^A,I^B]=0\, ,
$$
 it immediately follows that $\{\I,\varepsilon(D)\}=0=\{\Is,\iota(D)\}$ so it only remains to verify that $\I$ and $\Is$ anticommute 
with $\X\Omega^{\hash\hash}$ and $\Xs\, \Omega^{\hash\hash}$, respectively. Away from dimension four, this holds trivially since for almost Einstein structures the $W$-tractor obeys $I_A W^{ABCD}=0$. In dimension four, $\I$ and $\Is$ still commute with the operator $\frac{1}{d-4}\widetilde  W^{\hash\hash}$ of~\cite[Section 4]{powerslap} which is also discussed  in the proof of Proposition~\ref{Yform} below. From this $\slashed D{}^A$ can still be defined, as commented upon in Remark~\ref{Wtilde}, and  we again obtain the result.
\end{proof}

Next we consider anticommutators of $\Is$ and $\D$ (or $\I$ and $\Ds$).
\begin{proposition}\label{Yform}
The map
$$
y:\Gamma\ct^kM[w]\rightarrow \Gamma\ct^kM[w-1]
\quad\mbox{where}\quad 
y:=-I_A\slashed D{}^A
$$
obeys
\begin{center}
%\shabox{
$
\{\Is,\D\} = -y = \{\I,\Ds\}\, .
$
%}
\end{center}

Moreover
\begin{center}
%\shabox{
$
[\D,y]=[\I,y]=0=[\Is,y]=[\Ds,y]\, .
$
%}
\end{center}
\end{proposition}

\begin{proof}
As usual, a rudimentary proof is to evaluate the matrix expression for each of the three operators for a given $g\in c$. The result agrees for
each of these, and since the operator~$y$ is a central player in this Article, we record the explicit result:
$$
-y\stackrel{g}=(d+2w-2) \delta_R - \sigma \Big({\square_Y} +\frac{\J}2 (d+2w-2) {\mathbb 1}\Big)\, ,
$$
where $\mathbb 1$ is the identity (matrix), and along~$\Sigma$,  $\delta_R$ is a conformal Robin-type operator for forms. It is given in general by
\begin{equation}\label{Robinform}
\delta_R:=\begin{pmatrix}
\nabla_n+w\rho&-\iota(n) & \varepsilon(n) & 0\\[1mm]
\varepsilon(\nabla\rho)&\nabla_n+w\rho&0&\varepsilon(n)\\[1mm]
-\iota(\nabla\rho)&0&\nabla_n+w\rho&\iota(n)\\[1mm]
0&-\iota(\nabla\rho)&-\varepsilon(\nabla\rho)&\nabla_n+w\rho
\end{pmatrix}\, .
\end{equation}
At weight  $w=1-d/2$, $\square_Y$ is a conformally invariant Yamabe-type operator for forms. It is given at general weights by
\begin{center}
\scalebox{.85}{
$
\!\square_Y\!:=\begin{pmatrix}
\!\FL \!+\!(\degree-\frac{d}2)(\J-2\P)\!\!\!\!\!\!\!&-2\cod&2\extd&2\degree-d\\[2mm]
-[\J-2\P,\extd]&\!\!\!\!\!\!\!\FL\! +(\degree-\frac{d}2+1)(\J-2\P)\!\!\!\!\!\!\!&0&2\extd\\[2mm]
-[\J-2\P,\cod]&0&\!\!\!\!\!\!\!\FL \!+\!(\degree-\frac{d}2-1)(\J-2\P)\!\!\!\!\!\!\!&2\cod\\[2mm]
\!-\Rho_a^b\Rho^a_b+2\End(\Rho.\Rho)+\frac2{\sigma}\slashed{\mathbb B}\!\!\!
&-[\J-2\P,\cod]&[\J-2\P,\extd]&\!\!\!\FL\! +\!(\degree-\frac{d}2)(\J-2\P)\!
\end{pmatrix}\, .
$}
\end{center}
In this formula the 
slashed Bach endomorphism~$
\slashed{\mathbb B}$ is defined in {\it any dimension} in terms of the Cotton tensor by
$\big(n^c(\nabla_c \Rho_{a}^b-\nabla_a \Rho_{c}^b )\big)\, .
$
In dimensions not equal to four, on almost Einstein structures,  it is related to the standard Bach tensor by
$$
\slashed{\mathbb B} = \frac{\sigma{\mathbb B}}{d-4}\, .
$$ 
Indeed, in dimensions other than four, the $W$-tractor {\em divided by $d-4$}---from which one builds $\slashed D{}^A$---is a sum of Weyl and Cotton tensor terms, plus the Bach tensor over $d-4$. Replacing the Bach over $d-4$ contribution by $\slashed {\mathbb B}/\sigma$, yields the tensor $\widetilde  W/(d-4)$. It equals the standard $W/(d-4)$ in any dimension not equal to four, but is also a well-defined tractor on four dimensional conformally Einstein manifolds~\cite{powerslap}. In this way we have defined
and computed the third expression $-I\cdot \slashed D$ in arbitrary dimensions.

An alternative proof is to
define $\slashed D{}^A$ in this way, and then use that $[I^A,\slashed D{}^B]=0$. Thus the equalities $\{\Is,\D\}=\{\I,\Ds\}=I\cdot\slashed D$,
follow immediately from the properties of exterior and interior multiplication.  The remaining commutation relations quoted are also trivial,
for example
$$
[\D,y]=[\{\D,\Is\},\D]=\big(\Is\D+\D\Is\big)\D-\D\big(\Is\D+\D\Is\big)=\D\Is\D-\D\Is\D=0\, ,
$$
because $\D$ is nilpotent.
\end{proof}

\subsection{Boundary tractors}\label{tractorsontheedge}

Let us first recall some general facts concerning hypersurface tractors~\cite{BEG,Grant,Goal,GoIP} before specializing to Poincar\'e--Einstein structures
and tractor forms in order to develop natural boundary conditions for the extension problems studied in the next Section. Assume, therefore,
that $\Sigma$ is a boundary component of $(M,c)$ and the conformal structure extends smoothly to a collar neighborhood of~$\Sigma$.

Assume~$\Sigma$ is smooth and is the zero locus of a defining density~$\sigma$ (see Section~\ref{CGeom})
which is also a defining scale for our structure.
 Let $n_a\in\ce_a[1]$ be a unit conormal so that, along~$\Sigma$
$$
\bg^{ab} n_a n_b = 1\, .
$$
In a scale $g\in c$, the mean curvature of~$\Sigma$ is
$$
H^g=\frac{1}{d-1}\, \nabla^T_a n^a\, ,\qquad \nabla^T:=\nabla-n \nabla_n\, .
$$
From this data, we can build the {\it normal tractor} $N\in \Gamma\ct M|_\Sigma$ of~\cite{BEG}
$$
N^A\stackrel{g}= \begin{pmatrix}0\\n_a\\[1mm]-H^g\end{pmatrix}\, .
$$
The normal tractor satisfies $N_A N^A=1$ and is linked to the
scale tractor when $(M,c,\sigma)$ is an almost scalar constant structure.
By observing
$$
I|_\Sigma\stackrel{g}=\begin{pmatrix}0\\\nabla_a \sigma\\[1mm]-\frac{1}{d}\Delta \sigma\end{pmatrix}\, ,
$$
and using $I^2=1$ we have the following result~\cite[Proposition 3.5]{Goal}.
\begin{lemma}
If $(M,c)$ is an almost scalar constant structure with defining scale singularity set~$\Sigma$ and scale tractor $I$, then
the normal tractor of $\Sigma$ satisfies
$$
N=I|_\Sigma\, .
$$ 
\end{lemma}
  
The intrinsic tractor bundle $\ct\Sigma$ of $(\Sigma,c_{_\Sigma})$, where the conformal structure $c_{_\Sigma}$ is the one induced by~$c$,
is related to the standard  tractor bundle  $\ct M$ along $\Sigma$ ({\it i.e.} $\ct M|_\Sigma$)~\cite{BrGoOps}. Indeed, there is a canonical, conformally invariant isomorphism between the canonical, rank $d+1$ subbundle $N^\perp$ of $\ct M|_\Sigma$ orthogonal (with respect to the tractor metric) to the normal tractor $N$,
and the intrinsic boundary tractor bundle 
\hypertarget{isomb}{} 
$$
N^\perp \cong \ct\Sigma\, .
$$
We shall use this isomorphism to identify these spaces.
The map between their respective section spaces, calculated in a scale $g\in c$ (and therefore $g_{_\Sigma}=g|_\Sigma$) is~\cite{Grant,GoIP,Stafford}
$$
\begin{pmatrix}
\nu\\ \mu_a\\ \rho
\end{pmatrix}
\mapsto
\begin{pmatrix}
1&0&0\\
-H^g \, n_a & 1 & 0\\
-\frac12 (H^g)^2 & H^g \, n^b & 1
\end{pmatrix}
\begin{pmatrix}
\nu\\\mu_b\\\rho
\end{pmatrix}\, ,
$$  
where $n^a \mu_a = H^g \nu$ because the left hand side is a section of $N^\perp$.

This {\it boundary splitting isomorphism} can be extended to tractor tensor bundles, in particular for tractor forms the section space map from the subbundle of $\ct^\bullet M[\,  .\, ]|_\Sigma$
orthogonal to $M$ to $\ct^\bullet \Sigma[\,  .\, ]$~is 
\begin{equation}\label{boundaryisom}
\begin{pmatrix}
F^+ \\ F \\ F^{+-} \\ F^-
\end{pmatrix}
\stackrel{g}\longmapsto
\begin{pmatrix}
1&0&0&0\\
-H^g\, \varepsilon(n)&1&0&0\\
\ \ H^g\, \iota(n)&0&1&0\\
\frac{(H^g)^2}{2}\, [\varepsilon(n),\iota(n)]&H^g\, \iota(n)&H^g\, \varepsilon(n)&1
\end{pmatrix}
\begin{pmatrix}
F^+ \\ F \\ F^{+-} \\ F^-
\end{pmatrix}\, ,
\end{equation}
where here the orthogonal condition says (along~$\Sigma$) that
$$
\iota(n) F^+ = 0 = \iota(n) F - H^g\, F^+ = \iota(n) F^{+-} = \iota(n) F^{-} + H^g\, F^{+-}
\, .$$
In invariant terms, this can be expressed as~$\iota(N)\F=0$ for $\F\in\Gamma\ct^\bullet M[\,  .\, ]|_\Sigma$.

Recall that for Riemannian geometries, the Gau\ss\  formula relates the interior and boundary covariant derivatives.
In particular, if $u,v\in \Gamma TM$ are local extensions of $u_{_\Sigma},v_{_\Sigma}\in \Gamma T\Sigma$, then the Levi--Civita
connection $\nabla$  with respect to $g$ on $M$ and the Levi--Civita connection $\nabla^\Sigma$ on $\Sigma$ with respect to the metric $g_{_\Sigma}$ induced by $g$ {\it agree},
in the sense
$$
\big(\nabla_u v\big)^T|_\Sigma = \nabla^\Sigma_{u_{_\Sigma}} v_{_\Sigma}\, . 
$$
Here we denote the tangential component of $v\in T_{x\in \Sigma}M$ by $v^T$. Moving to the almost scalar constant setting, choosing a $g\in c$ and setting $n:=\nabla\sigma$, in the above formula we can replace the connection $\nabla$
by $\nabla^T=\nabla-n\nabla_n$. In this way the independence on the left hand side from the choice of local extension is manifest because of the operator statement 
$$
\nabla^T \sigma = 
O(\sigma)\, .
%\sigma\,  \nabla^T\, .
$$

There exists a natural and canonical tangential operator related to the Thomas D-operator that we shall only need
in the Poincar\'e--Einstein setting. In fact, a main point we wish to emphasise here is that in that setting this operator
is holographic formula for the boundary Thomas D-operator.
\begin{definition}
Let $(M,c,\sigma)$ be a Poincar\'e--Einstein manifold. Then we define the 
{\it tangential Thomas D-operator} defined acting on tractors in $\Gamma\ct^\Phi M[w]$ with $w\neq 1-\frac d2,1-\frac n2$ by
$$
D^T_A:=D_A -I_A I\cdot D +\frac{1}{(h-1)(h-2)}\, X_A \big(I\cdot D\big)^2\, .
$$
\end{definition}
\begin{remark}
Along~$\Sigma$,~the tangential Thomas D-operator may be viewed as a
 tractor analog of the Gau\ss\ formula.
\end{remark}

\begin{definition}
Suppose $\sigma$ is a defining density for a hypersurface $\Sigma$.
Let $P$ be a differential operator acting on the section space of a  vector bundle~$\cf$.
We say~$P$ acts {\it tangentially along~$\Sigma$} (or informally ``$P$ is tangential'') if
$$
P\circ \sigma = \sigma \circ \widetilde P \, ,
$$
for $\widetilde P$ some  smooth linear operator acting on $\Gamma(\cf\otimes\ce M[-1])$ in the same neighbourhood.
\end{definition}

\begin{remark}
It is straightforward to  verify that $D^T_A$ is tangential by employing the identity
$$
(h-2) D^A \sigma = h\sigma D^A -2 X^A I\cdot D + h(h-2) I^A\, , 
$$
which follows, for Poincar\'e--Einstein structures,  directly from equation~\nn{ureq}.
\end{remark}

\begin{remark}
It is  useful to define  $\Gamma\F\big|\!\big|_\Sigma$ the space 
of equivalence classes of sections 
\begin{equation}\label{barbar}\A\sim \A + O(x)\, ,\qquad \A \in \Gamma\F \, .\end{equation}
The space $\Gamma\F\big|\!\big|_\Sigma$ is naturally isomorphic to $\Gamma\F\big|_\Sigma$.
Note that tangential operators $P:\Gamma\F\to \Gamma\F'$ act canonically on $\Gamma\F\big|\!\big|_\Sigma$
 by $$\Gamma\F\big|\!\big|_\Sigma\ni[\A]\longmapsto [P\A]\in \Gamma\F'\big|\!\big|_\Sigma\, .$$
\end{remark}

The tangential Thomas D-operator and the (boundary) Tho\-mas D-operator on the intrinsic tractor bundle of $(\Sigma,c_{_\Sigma})$ are related as follows.
\begin{proposition}\label{Thomtang}
Let $(M,c)$ be Poincar\'e--Einstein with defining scale singularity set $\Sigma$ and let $U,V$ be local extensions of $U_{_\Sigma} \in \Gamma \ct \Sigma$,
$V_{_\Sigma}\in \Gamma \ct\Sigma[w]$ with $w\neq 1-\frac d2,2-\frac d2$ and subject to $I\cdot\,  U=0=I\cdot\,  V$. Then
\begin{equation}\label{shortimag}
D^\Sigma_{U_{_\Sigma}} V_{_\Sigma} = \left.\left(\frac{d+2w-3}{d+2w-2} \, D_U V+\frac{1}{(d+2w-2)(d+2w-4)}\, X\cdot\,  U \big(I\cdot D\big)^2 \, V\right)\right|_\Sigma\, .
\end{equation}
\end{proposition}
The proof of Proposition~\ref{Thomtang} is given in Appendix~\ref{AMBIENT}. There we demonstrate that $D^T_A = h(h-1)^{-1} D_A^\Sigma$ along~$\Sigma$ which suffices to establish the result. It can also be proven by a  tedious explicit computation for a given $g\in c$.
\begin{remark}
Let us  clarify the meaning of the formula~\nn{shortimag}:
\begin{itemize}
\item Because $I$ is (tractor) parallel, the right hand side of the above display is manifestly an element of $\Gamma \big(N^\perp\otimes\ce M[w]\big)$
so equality is in the sense of the natural extension of the isomorphism $N^\perp\cong \ct\Sigma$ (explained \hyperlink{isomb}{above}) to weighted
tractors. We will often use this isomorphism without further comment where appropriate.
%\end{remark}
%\begin{remark}
\item
Given any extension $\widehat U$ of $U_{_\Sigma}$, we can construct another extension $U:=\widehat U-I\cdot \, \widehat U\, I$ satisfying $I\cdot\,  U=0$.
%\end{remark}
%\begin{remark}
\item
So long as $w\neq \ell-\frac d2$, $\ell=2,3,4$, the tractor expression on the right hand side of the display in Proposition~\ref{Thomtang} is
$$
\frac{d+2w-3}{d+2w-2} \, U\cdot D^T V\, .
$$
\item
At the  boundary Yamabe weight $w=\frac 32-\frac d2 = 1-\frac n2$, the Proposition  states $$D^\Sigma_{U_{_\Sigma}} V_{_\Sigma} = -X\cdot\,  U \big(I\cdot D\big)^2 \, V\big|_\Sigma\, ,$$
and so recovers the holographic formula for the Yamabe operator of~\cite{GWasym}. 
%\end{remark}
%\begin{remark}
\item
At the (excluded) bulk Yamabe weight  $w=1-\frac d2$ there is a version of the Proposition applying to the double D-operator $D^{AB}$. 
We will develop this for the specialization of the above Proposition to tractor forms below.
\item
At the  excluded weight  $w=2-\frac d2$, the  residue of the pole is the operator $(I.D)^2$ which is $\sigma^2$ times the bulk Paneitz operator~\cite{powerslap}, and therefore
vanishes along~$\Sigma$. In fact acting on tractor forms, this singularity is removable as we shall also  show below. 
\end{itemize}
\end{remark}

For the purposes of this Article we need to develop a variant of Proposition~\ref{Thomtang} that uses the exterior tractor D-operator. To that 
end, on Poincar\'e--Einstein structures we introduce the {\it tangential exterior tractor D-operator}
\begin{equation}\label{DT}
%\shabox{
\scalebox{1.05}{$
\D^T:=\D + \I y + \frac{\textstyle 1}{\textstyle(h-1)(h-2)}\, \X y^2\, .
$}
%}
\end{equation}
Here $\D^T:\ct^kM[w]\to\ct^{k+1}M[w-1]$ and is defined whenever $w\neq \frac32-\frac d2,2-\frac d2$.
When $w=2-\frac d2$ we define
\begin{equation}\label{special}
\D^T:=\D + \I y + \big(\Is\X\Dt-\X \Dt \Is\big)\, y\, .
\end{equation}
The tangential exterior tractor D-operator will play a central {\it r\^ole}  in our later study of detours and gauge operators thanks to the following result.

\begin{theorem}\label{dt}
Let $(M,c,\sigma)$ be Poincar\'e--Einstein with defining scale singularity set $\Sigma$ and let 
$\A\in\Gamma\ct^kM[w]$ be a local extension of $\A_{_\Sigma} \in \Gamma \ct^k\Sigma[w]$, 
with $w\neq \frac 32-\frac d2$ and subject to  $\Is \A=0$. Then
$$
\big(\D^T \A\big)\big|_\Sigma = \frac{d+2w-2}{d+2w-3} \, \D_{_\Sigma} \A_{_\Sigma}\, ,
$$
where $\D_{_\Sigma}$ is the exterior tractor D-operator of~$\ct^k\Sigma[w]$. 
When $w=\frac32-\frac d2$ and all other preconditions as above hold, then
$$
\big(\X y^2 \A\big)\big|_\Sigma =  -\D_{_\Sigma} \A_{_\Sigma}\, .
$$
\end{theorem}
The proof of this Theorem is given in Appendix~\ref{AMBIENT}. 
 
 \begin{remark}\label{tangDD}
Similar remarks apply as for Proposition~\ref{Thomtang}: 
\begin{itemize}
\item
The equalities are in the sense of the isomorphism between the orthogonal component
of the bundle $\ct^\bullet M[\, .\, ]$ along $\Sigma$ and $\ct^\bullet \Sigma[\, .\, ]$. The condition $\Is \A=0$ along~$\Sigma$ implies
$\iota(N)\A|_\Sigma=0$. Moreover, given any extension of $\A$ of $\A_{_\Sigma}$, we can always construct another extension in the kernel of~$\Is$
by multiplying by the projector $\Is\I$.
%\end{remark} 
%\begin{remark}
\item
%As for the exterior tractor D-operator $\D$, here 
%Away from bulk Yamabe weight we define $$\Dt^T:=\frac 1h \D^T\, .$$ 
The {\it tangential exterior double D-operator} is defined by $$D_{[2]}^T:=-\X \Dt^T\, .$$ 
%then has a removable singularity at $w=1-\frac d2$. 
With the same conditions as the Theorem and defining the $(\Sigma,c_{_\Sigma})$ version of 
$\Dt_{_\Sigma}$
%:=\frac1{h_{_\Sigma}} \D_{_\Sigma}$ 
also using   Definition~\ref{WHAT}, we have
$$
\big(\D^T_{[2]} \A\big)\big|_\Sigma=-\big(\X\Dt^T \A\big)\big|_\Sigma =  -\X_{_\Sigma} \Dt_{_\Sigma} \A_{_\Sigma}=\D_{[2]}^{\scriptscriptstyle \Sigma} \A_{_\Sigma}\, .
$$
%For later use we record the expressions for the tangential interior and exterior double D-operators for some choice of $g\in c$
%\begin{equation}
%\end{equation}
%\end{remark}
%\begin{remark}
\item
The formula~\nn{special} for $\D^T$ at the bulk Paneitz weight $w=2-\frac d2$ can be understood by noting that away from $w=2-\frac d2$
$$
\D^T:=\D + \I y + \frac1{h-1}\big(\Is\X\Dt-\X \Dt \Is\big)\, y-\frac{1}{(h-1)(h-2)}\, x\, \D\,  y\, ,
$$
where the singularity has been removed by discarding the last term which vanishes along~$\Sigma$.
\end{itemize}
\end{remark}
For later use, it is convenient to make the following definition.
\begin{definition}\label{baronD}
Acting on $\Gamma\ct^k M[w]$, we define  the operator
$$
\oD := \left\{\begin{array}{cc}(h-1) \Dt^T\, ,&\quad w\neq 1-\frac d2,1-\frac n2\\[1mm]
-\X y^2 \, ,& \quad w=1-\frac n2\, .\end{array}\right.
$$ 
\end{definition}
\noindent
This is useful because $\big(\oD\A\big)\big|_\Sigma = \D_{_\Sigma} \A_{_\Sigma}$ for $\A\in \ker \Is$ an extension of $\A_{_\Sigma}\in\Gamma\ct^k\Sigma[w]$.

\subsection{Boundary conditions for tractor forms}\label{bctractors}

An almost Einstein structure provides a simple construction of natural boundary conditions
for differential forms. These can be classified by  insertions of forms in tractors according
to the analysis of Section~\ref{compinsert}. Here we focus on the ``western case'' needed later in the Article.

Suppose $A_{_\Sigma}\in\Gamma\ce^k\Sigma[w+k]$ is some boundary form which we shall view as the boundary data
for our problem.
The aim is to canonically %(although {\it not} uniquely) 
extend this to an interior tractor $\A\in\Gamma\ct^k M[w]$. In Section~\ref{formssol} we will study an extension problem $y\A=0$
(see also Problem~\ref{canextprob}) as far as possible uniquely determining an $\A$ in the space of canonical extensions.

There are two obvious approaches to extending $A_{_\Sigma}$ to a bulk form tractor:
First extend to an interior differential form and thereafter embed in a tractor form, or secondly
embed~$A_{_\Sigma}$ 
 in a boundary tractor and then 
extend to an interior tractor form. 
\begin{equation}\label{Bryan}
\begin{array}{rcccl}
&\Gamma\ce^k\Sigma[w+k] & \stackrel{\rm ext}{-\!\!\!-\!\!\!-\!\!\!-\!\!\!\longrightarrow} & \Gamma\ce^kM[w+k]&\\
&|&&|&\\[-3mm]
q_W^{\Sigma}\hspace{-13mm}&|&&|&\hspace{-14mm}\, q_W^{\phantom{\Sigma}}\\[-4mm]
&|&&|&\\[-5mm]
&\downarrow & & \downarrow &\\
&\Gamma\ct^k\Sigma[w] & \stackrel{\rm Ext}{-\!\!\!-\!\!\!-\!\!\!-\!\!\!\longrightarrow} & \Gamma\ct^kM[w]&\\
\end{array}
\end{equation}

The key to writing tractor problems that precisely encode differential form problems is constructing 
extensions so this diagram commutes. For encoding boundary conditions we must find extensions so 
that this holds along~$\Sigma$.
\begin{proposition}\label{extExt}
Suppose $w\neq k-d,k-n$. Then
any extensions ${\rm  ext}$ and ${\rm Ext}$ 
$$
{\rm ext}:\Gamma\ce^k\Sigma[w+k]\longrightarrow \ker \wiota\subset \Gamma\ce^kM[w+k]
$$
and
$$
{\rm Ext}:\Gamma\ct^k\Sigma[w] \longrightarrow \ker(\Is,\wDs,\Xs)\subset \Gamma\ct^kM[w]\, ,
$$
obey
$$
r\circ q_W \circ {\rm ext} = r\circ {\rm Ext} \circ q_W^{\Sigma}\, ,
$$
where ${\rm r}$ denotes restriction to $\Sigma$.
\end{proposition}

\begin{proof} 
We begin by setting up the structures involved and start with ext.
Since $\Lambda^\bullet \Sigma$ is naturally a subbundle of $\Lambda^\bullet M|_\Sigma$,
so too is $\ce^k \Sigma[w]$ of $\ce^k M[w]|_\Sigma$.
Thus we choose any smooth weighted form $A\in \Gamma\ce^k M[w+k]$ such that
$$
A|_\Sigma = A_{_\Sigma}\, .
$$
Calculating in some choice of $g\in c$, this extension of $A_{_\Sigma}$  obeys
\begin{equation}\label{iotaA}
\iota(n) A + \sigma \phi = 0\, ,
\end{equation}
for some $\phi\in \Gamma\ce^{k-1}M[w+k-2]$, since necessarily $\big(\iota(n) A\big)\big|_\Sigma = 0$.
By continuity $\iota(n)\phi$ is zero everywhere.

The normal component of $A$, {\it i.e.} $\iota(n)A$, vanishes along~$\Sigma$,  but its normal derivative 
does not and is encoded by $\phi$ along $\Sigma$ as follows from~\nn{iotaA}:
\begin{equation}\label{normcomp}
\phi_{_\Sigma}=-\Big(\nabla_n\big[\iota(n) A\big]\Big)\Big|_\Sigma\, .
\end{equation}

Looking ahead, we want to construct a west tractor. This suggests an ansatz for $\A\in \Gamma\ct^k M[w]$ by its expression in the scale $g\in c$
\begin{equation}\label{XAform}
\A\ :\stackrel{g}=\ \begin{pmatrix}0\\ \ A \ \\ 0 \\ \phi \end{pmatrix}\, ,
\end{equation}
which satisfies $\Xs\A=0$.
At weights $w\neq k-d$, $\phi$ can be tied to $A$  by imposing
$$
\wDs\A = 0\, .
$$
%on the extension $\A$,
%coupled  with the requirement $\Xs \A=0$ above. 
The west Lemma~\ref{West} now applies so
$$
\A=q_WA\stackrel{g}=\begin{pmatrix}0\\ A\\0\\-\frac{1}{d+w-k}\ \cod A\end{pmatrix} \, ,\quad w\neq k-d \, ,
$$
and indeed we must have $\phi=-\frac{1}{d+w-k}\ \cod A$. 

Compatibility of the above ansatz with the condition $\Is \A =0$,  suggested by commutativity of the diagram,  yields (via Equation~\nn{iotaA})
 %implies that $\phi=-\frac{1}{d+w-k}\ \cod A$ so 
\begin{equation}\label{COULOMB_GAUGE}
\big[\sigma\cod -(d+w-k) \iota(n)\big] A=\wiota A =0\, .
\end{equation}
We have arrived canonically at the generalised divergence equation of Problem~\ref{deltaprob} which was solved (to some order) for arbitrary boundary data $A_\Sigma$ with weights $w\neq k-n$  in Section~\ref{coulomb}. The compatibility with that problem is critical to subsequent developments.

To summarise, extending $A_{_\Sigma}\in\Gamma \ct^k\Sigma[w+k]$ to $A\in \Gamma\ct^kM[w+k]$ subject to Equation~\nn{COULOMB_GAUGE} and then inserting this in a west tractor $\A=q_W A\in \Gamma\ct^kM[w]$ produces a solution to
\begin{equation}\label{allbcs}
\Is\A = 0 = \wDs\A = \Xs \A\, ,\qquad \big(q^* \A\big)\big|_\Sigma = A_{_\Sigma}\, .
\end{equation}

To establish Proposition~\ref{extExt}, it only remains to show equality of $\A|_\Sigma$ and $q_W^\Sigma A_{_\Sigma}$.
Along~$\Sigma$ Equation~\nn{COULOMB_GAUGE} implies $A|_\Sigma=A_{_\Sigma}\in\Gamma \ce^k\Sigma[w+k]$ 
and so 
$$
\left.\begin{pmatrix}0\\ A\\0\\-\frac{1}{d+w-k}\ \cod A\end{pmatrix}\right|_{\scalebox{1.1}{$\Sigma$}}=
\begin{pmatrix}0\\ A_{_\Sigma}\\0\\-\frac{1}{n+w-k}\ \cod_{_\Sigma} A_{_\Sigma}\end{pmatrix}\, ,
$$
which is a consequence of  Lemma~\ref{deltadelta}, which follows. This completes the proof since the inverse of the boundary splitting isomorphism, given in Equation~\nn{boundaryisom}, acts as the identity 
on such sections.

\end{proof}

\begin{lemma}\label{deltadelta}
Let $A\in \Gamma \ce^k M[w+k]$ be an extension 
%\stackrel{g}=q^*\A$ (for some $g\in c$) 
%be the western slot of a local extension $\A\in \Gamma\ct^kM[w]$ 
of $A_{_\Sigma}\in \Gamma \ce^k\Sigma[w+k]$,
subject to Equation~\nn{COULOMB_GAUGE}. Then
$$
(n+w-k)\, \big(\cod A\big)\big|_\Sigma
=(d+w-k)\,\cod_{_\Sigma} A_{_\Sigma}\,  .
$$
%When $w=k-n$, $\cod_{_\Sigma} A_{_\Sigma}=0$.\edz{A: FIX LOGIC}
\end{lemma}
\begin{proof}
When $w=k-d$, Equation~\nn{COULOMB_GAUGE} immediately says $\cod A=0$, so we may assume $w\neq k-d$.
We perform  a direct computation of $\cod A$ along~$\Sigma\, $
%, firstly avoiding the case $w=k-n$:
\begin{eqnarray*}
 \big(\cod A\big)\big|_\Sigma
&=&\big(\cod\, (\iota(n)\varepsilon(n) + \varepsilon(n)\iota(n))\, A\big)\big|_\Sigma\\
&=&\cod_{_\Sigma} A_{_\Sigma} + \frac1{d+w-k}\big(\iota(n) \varepsilon(n) \cod A \big)\big|_\Sigma\\[1mm]
&=&\cod_{_\Sigma} A_{_\Sigma} + \frac1{d+w-k}\big(1-\varepsilon(n) \iota(n) \big)\cod A \big|_\Sigma\, .
%&=&\cod_{_\Sigma} A_{_\Sigma} -  \big(\iota(n)  \nabla_n A\big)\big|_\Sigma\\[1mm]
%&=&\cod_{_\Sigma} A_{_\Sigma} -  \big(\iota(n) \, (\cod \varepsilon(n) + \varepsilon(n) \cod) \, A \big)\big|_\Sigma\\
%&=&\frac{d+w-k}{n+w-k}\, \cod_{_\Sigma} A_{_\Sigma}\, .
\end{eqnarray*}
On the first line we inserted $1=2\rho \sigma+ \iota(n)\varepsilon(n) + \varepsilon(n)\iota(n)$
and used that $\big(\iota(n)A\big)\big|_\Sigma=0$.
To obtain the second line we used that $\iota(n) A = \frac1{d+w-k}\, \cod A$ as well as the relationship between
bulk and boundary codifferentials
$$
\big(\cod\,  \iota(n) \varepsilon(n) A\big)\big|_\Sigma= \cod_{_\Sigma} A_{_\Sigma}\, .
$$
Finally, $\iota(n) \cod A = 0$ as can be seen by acting upon Equation~\nn{COULOMB_GAUGE} with $\iota(n)$ which completes the proof.
\end{proof}

\begin{remark}
The weight $w=k-n$ corresponds to a special case of the west Lemma~\ref{West} along the boundary. Indeed the above Lemma then imposes
the condition $\cod_{_\Sigma}A_{_\Sigma}=0$ which leads to a qualitatively different natural boundary problem.\end{remark}

Now we turn to distinguished weights.
Consider first a tractor $\A$ satisfying~\nn{allbcs} at the value 
 $w=k-d$. Then from Lemma~\ref{West}, we have
$$
\ker(\wDs,\Xs)\subset\ct^kM[w]\ni\A\stackrel{g}{:=} \begin{pmatrix}0\\ A\\0\\\phi\end{pmatrix} \, ,\mbox{ such that } \cod A = 0 =\cod\,\!  \phi\, .
$$
Remarkably the above display reduces to a west tractor along~$\Sigma$. To see this we need to show $\phi|_\Sigma =\phi_{_{\Sigma}} =
\cod_{_{\Sigma}} A_{_{\Sigma}}$. This follows from a rapid computation using only $\Is \A=0$:
\begin{eqnarray}
0 &=& \big(\cod A\big)\big|_\Sigma\nonumber\\
&=&\big(\cod\,  (2\rho\sigma + \iota(n)\varepsilon(n) + \varepsilon(n)\iota(n))A\big)\big|_\Sigma\nonumber\\
&=&\cod_{_{\Sigma}} A_{_{\Sigma}} - \big(\cod \varepsilon(n)\sigma\phi\big)\big|_\Sigma\nonumber\\
&=&\cod_{_{\Sigma}} A_{_{\Sigma}} - \phi_{_\Sigma}\, .\label{phideltaA}
\end{eqnarray}
This result is encapsulated by the following version of Proposition~\ref{extExt}.
\begin{proposition}\label{bcs1}
Given as data $A_{_\Sigma}\in \Gamma\ce^k\Sigma[2k-d]$, the following constructions of a tractor along~$\Sigma$ agree:
\begin{enumerate}[(i)]
\item 
Take any coclosed extension of $A_{_\Sigma}$ to $A\in \Gamma \ce^k M[2k-d]$ and pair it with a coclosed
form $\phi\in\Gamma\ce^{k-1} M[2k-d-2]$ by requiring $\A:=q_W(A,\phi)\in \ker \Is$, which is well-defined because
$$(A,\phi) \in\ker(\cod,\cod) \subset\big(\Gamma \ce^k M[2k-d]\oplus\Gamma\ce^{k-1} M[2k-d-2]\big)\, .$$
Then take the restriction.
\item Map $A_{_\Sigma}$ to $\A_{_\Sigma} \in \Gamma \ct^k\Sigma[k-d]$ with the boundary insertion operator 
$q_W^{\scriptscriptstyle \Sigma}$.
\end{enumerate}
That is 
$$
\A\big|_\Sigma=\A_{_\Sigma}\in\ker\big(\wDs_{_\Sigma},\Xs_{_\Sigma}\big)\, .
$$
\end{proposition}

Next, observe that 
the case $w=k-n$ is also distinguished because it corresponds to a special case for the west Lemma along~$\Sigma$.
Again the details are interesting and surprising.
Suppose that $\A\in \ker(\Ds,\Xs,\Is)$
with $w=k-n$, then for $g\in c$ 
$$
\A\stackrel{g}=\begin{pmatrix}0\\A\\0\\\phi\end{pmatrix} \mbox{ with } \phi=-\cod A \mbox{ and } \iota(n) A +\sigma \phi = 0 = \iota(n) \phi\, .
$$
Write $A_{_\Sigma}:=A|_\Sigma$ and consider
\begin{eqnarray*}
\cod_{_\Sigma} A_{_\Sigma}=\big(\cod\, \iota(n)\varepsilon(n)A\big)\big|_\Sigma
= \big(\cod(-\varepsilon(n)\iota(n)  + 1-2\rho\sigma)A\big)\big|_\Sigma
=0\, .
\end{eqnarray*}
This implies we can no longer use the boundary form $A_{_\Sigma}$ as Dirichlet data.
Instead the extension~$A$ of $A_{_\Sigma}$ carries the independent Neumann data $\phi_{_\Sigma}$ according to~\nn{normcomp}.
Moreover, $\phi_{_\Sigma}$ is also coclosed along $\Sigma$ because
$$  
\cod_{_\Sigma} \phi_{_\Sigma}=\big(\cod\, \iota(n)\varepsilon(n)\phi\big)\big|_\Sigma
= \big(\cod(-\varepsilon(n)\iota(n)  + 1-2\rho\sigma)\phi\big)\big|_\Sigma
=0\, ,\quad
$$
(using $\cod \phi = - \cod^2 A=0$). But exactly because $w=k-n$, by virtue of Lemma~\ref{West}, the coclosed boundary forms $(A_{_\Sigma},\phi_{_\Sigma})$ are isomorphic to a boundary tractor $\A_{_\Sigma}\in \ker(\wDs_{_\Sigma},\Xs_{\Sigma})$.

Thus for this case we have a modified version of Proposition~\ref{extExt} giving a Dirichlet boundary condition for the pair of coclosed forms $(A_{_\Sigma},\phi_{_\Sigma})$.
\begin{proposition}\label{bcs2}
Given the data $$(A_{_\Sigma},\phi_{_\Sigma})\in \ker(\cod_{_\Sigma},\cod_{_\Sigma})\subset\big(\Gamma\ce^k\Sigma[2k-n]\oplus \Gamma\ce^{k-1}\Sigma[2k-n-2]\big)\, ,$$
the following constructions of a tractor form along~$\Sigma$ agree:
\begin{enumerate}[(i)]
\item Take any extension of  $A_{_\Sigma}$ to $A \in \Gamma \ce^k M[2k-n]$,
where $A_{_\Sigma}=A|_\Sigma$ is Dirichlet data and $\phi_{_\Sigma}=-\big(\nabla_n\big[\iota(n) A\big]\big)\big|_\Sigma$ is Neumann data, and
such that 
$
\A:=q_W A\in\Gamma\ct^k M[k-n]
$
satisfies $\Is \A=0$. Then take the restriction. 
\item Map $(A_{_\Sigma},\phi_{_\Sigma})$ to $\A_{_\Sigma} \in \Gamma \ct^k\Sigma[k-n]$ with the boundary insertion operator 
$q_W^{\scriptscriptstyle \Sigma}$.
\end{enumerate}
That is 
$$
\A\big|_\Sigma=\A_{_\Sigma}\in\ker\big(\wDs_{_\Sigma},\Xs_{_\Sigma}\big)\, .
$$
\end{proposition}

\subsection{Holographic boundary projectors}\label{holographicproj}

Our aim at this stage is to construct projectors which solve the extension problem to be introduced in the next Section.
A first step is to 
construct holographic formul\ae~for projectors implementing the tractor  boundary conditions introduced above, while at the same time solving the 
scale-transversality conditions. In equations, we seek a tractor $\A$ obeying
\begin{equation}\label{stc}
\Is \A=
\wDs \A=
\Xs \A=0\, 
%\big(\delta_R \A\big)\big|_\Sigma=0\, ,\quad
\end{equation}
subject to
$$
\A\big|_\Sigma=\A_{_\Sigma}\in\ker\big(\wDs_{_\Sigma},\Xs_{_\Sigma}\big)\, .
$$
%leaving the normal Robin condition until  when we solve the Laplace extension problem $y\A=0$ in Section~\ref{formssol}.
%Firstly, recall that for all weights the boundary data is a tractor $\A_{_\Sigma}\in\ker(\wDs_{_\Sigma},\Xs_{_\Sigma})$.
There are however three separate cases: $w=k-n$, $w=-k$ and $w$ generic. We begin with the last case.

\subsubsection{Holographic projectors at  generic weights}\label{generic}
Recall  from Section~\ref{compinsert} the projector $\Gamma\ct^k\Sigma[w]\to \ker(\wDs_{_\Sigma},\Xs_{_\Sigma})\subset \Gamma\ct^k\Sigma[w]$ for $w\neq -k,k-n$ given by
$$\Pi^{{}_{^\Sigma}}_W:
=\scalebox{.9}{$\frac{1}{(w+k)(n+w-k)}$}\, \wDs_{_\Sigma}\Xs_{_\Sigma}\wD^{\phantom{\star}}_{_\Sigma}\XsS \, ,
%=\left\{\begin{array}{cc}
%\scalebox{.85}{$-\frac{1}{(n+2w-2)(n+2w+2)(w+k)(n+w-k)}$}\, \Ds_{_\Sigma}\Xs_{_\Sigma} \X^{}_{_\Sigma} \D^{}_{_\Sigma}\, , & w\neq 1-\frac n2, -1-\frac n2\, ,\\[2mm]
%\scalebox{.85}{$\frac1{(n + 2 w - 2)^2(w+k)(n+w -k)}$}\,  \Xs_{_\Sigma}\Ds_{_\Sigma}\X^{}_{_\Sigma}\D^{}_{_\Sigma}\, , & w\neq 1-\frac n2\, ,\\[2mm]
%\scalebox{.85}{$\frac1{(n + 2 w + 2)^2 (w+k) (n +w -k)}$}\, \Ds_{_\Sigma}\Xs_{_\Sigma}\D^{}_{_\Sigma}\X^{}_{_\Sigma}\, , & w\neq -1-\frac n2\, ,
%\end{array}\right.
\quad\mbox{satisfying}\quad \big(\Pi^{{}_{^\Sigma}}_W\big)^2=\Pi^{{}_{^\Sigma}}_W\, .$$
%projects onto $\ker(\wDs_{_\Sigma},\Xs_{_\Sigma})$, or in other words
%(when in addition $w\neq -\frac n2,k-n$), 
%western tractors.
%\edz{A:Should augment to distinguished weights.....}
%\begin{remark} The third line of the above definition can be written more compactly as $$\Pi^{{}_{^\Sigma}}_W:
%=\scalebox{.9}{$\frac{1}{(w+k)(n+w-k)}$}\, \wDs_{_\Sigma}\Xs_{_\Sigma}\wD^{\phantom{\star}}_{_\Sigma}\XsS \, .$$
%This definition is then valid for all $w\neq -k,k-n$.
%\end{remark}\edz{A: At end of the day, the remark version of this formula should suffice....}
There is a holographic formula for this {\it boundary western projector}.
%To construct it we first define {\it (holographic) exterior} and {\it interior triple D-operators}.
First we develop some new tools.
\begin{definition}
Let $(M,c,\sigma)$ be an almost Riemannian structure. The {\it holographic interior} and {\it exterior triple D-operators} are
defined, respectively, by
$$ \D_{[3]} := \wD\,  \X \I\, \qquad\mbox{ and } \qquad \Ds_{[3]}:=\wDs\Xs\Is\, ,$$
where 
$$
\D_{[3]} :\Gamma\ct^kM[w]\longrightarrow \Gamma\ct^{k+3}M[w]\, \qquad \mbox{ and } \qquad \Ds_{[3]}:\Gamma\ct^kM[w]\longrightarrow \Gamma\ct^{k-3}M[w]\, .
$$
For a choice of $g\in c$ one has
\begin{equation}\label{tripleD}
\D_{[3]} \stackrel{g}{=} \begin{pmatrix}0&0&\:\: 0\:\:&\:0\:\\[1mm]
0&0&0&0\\[1mm]
-\varepsilon(n)\extd&\wepsilon&0&0\\[1mm]
0&\varepsilon(n) \extd& 0 & 0\end{pmatrix}\, \qquad
\mbox{ and }\qquad
\Ds_{[3]} \stackrel{g}{=} \begin{pmatrix}0&0&0&0\\[1mm]
-\cod\iota(n)&0&\wiota&0\\[1mm]
0&0&0&0\\[1mm]
0&0& \cod\iota(n) & \:0\:\end{pmatrix}\, ,
\end{equation}
where
$$\wepsilon=\varepsilon(n) (w+\degree) - \sigma \extd\qquad \mbox{ and } \qquad
\wiota = \iota(n) (d+w-\degree)-\sigma \cod
\, .
$$
\end{definition}
\begin{remark}\label{DDD}
Like their double D ancestors, the holographic triple D-operators obey a graded Leibnitz rule.
Furthermore, we view these as holographic formul\ae\  for  the boundary double {\it and} (extrinsic) triple D-operators because along~$\Sigma$ they restrict to 
 $\varepsilon(N) \wD_{_\Sigma}\X_{_\Sigma}$ and  $\wDs_{_\Sigma}\Xs_{_\Sigma} \iota(N)$. 
\end{remark} 

From their definitions, or the explicit expressions displayed for a choice of $g\in c$, we immediately see that these operators are  tangential because they commute with the scale operator ~$x=\sigma$
\begin{center}
%\shabox{
$
\big[\D_{[3]}, x\big] = 0 = \big[\Ds_{[3]}, x\big]\, .
$
%}
\end{center}
Moreover $\operatorname{ran}\D_{[3]}\subset\ker(\wD,\X)$ and $\operatorname{ran}\Ds_{[3]}\subset\ker(\wDs,\Xs)$.
On an almost Riemannian structure we may  write $\D_{[3]}=\I \wD \X$ and $\Ds_{[3]}=\Is \wDs \Xs$. This is obvious for almost Einstein structures, but is also easily verified in the more general almost Riemannian setting.
Thus
$$\operatorname{ran}\D_{[3]}\subseteq \ker(\I,\wD,\X)\qquad \mbox{ and } \qquad \operatorname{ran}\Ds_{[3]}\subseteq\ker(\Is,\wDs,\Xs)\, .$$
 For most 
weights $\subseteq$ may be replaced by equality, see Section~\ref{coho}.

\begin{definition}
Let $w\neq-k,k-n$. 
%and let $\A\in\Gamma\ct^kM[w]$ with $\A_{_\Sigma}=\A|_\Sigma$ and $\big(\Is \A\big)\big|_\Sigma=0$.
We  define the {\it holographic boundary (west) projector}
$$\Pi:\Gamma\ct^kM[w]\longrightarrow\ker(\Is,\wDs,\Xs)\subset\Gamma\ct^kM[w]\, ,$$
by
$$
\Pi:=\scalebox{.9}{$\frac{1}{(w+k)(n+w-k)}$}\, \Ds_{[3]}\,  \D_{[3]}\, .
$$
\end{definition}

Although we term $\Pi$  the  holographic boundary projector,
it is only idempotent along~$\Sigma$: we view $\Pi$ as a holographic formula for the boundary west projector~$\Pi_{_\Sigma}$. In Lemma~\ref{failure} we shall show that bulk idempotence holds on the kernel of the  extension operator $y=-I\cdot \slashed D$.

\begin{proposition}\label{projprop}
%Suppose $w\neq-k,k-n$. 
%and let $\A\in\Gamma\ct^kM[w]$ with $\A_{_\Sigma}=\A|_\Sigma$ and $\big(\Is \A\big)\big|_\Sigma=0$.
On a conformally compact manifold,
 the holographic boundary  operator~$\Pi$
%$$\Pi:\Gamma\ct^kM[w]\longrightarrow\Gamma\ct^kM[w]\, ,$$
%with
%$$
%\Pi:=\scalebox{.9}{$\frac{1}{(w+k)(n+w-k)}$}\, \Ds_{[3]}\,  \D_{[3]}\, ,
%$$
obeys
\begin{equation}\label{pihat}
\Is \Pi = 0=\wDs\Pi =\Xs \Pi \, .
\end{equation}
If $\A\in\ker(\Is,\wDs,\Xs)\subset\Gamma\ct^kM[w]$, $w\neq-k,k-n$, then $$\big(\Pi \A\big)\big|_\Sigma=\A_{_\Sigma}\, .$$
Moreover $\Pi$ is tangential and if $\A\in\Gamma\ct^kM[w]$ with $\big(\Is \A\big)\big|_\Sigma=0$ then
$$\qquad
\big(\Pi \A\big)\big|_\Sigma = \Pi_W^{\scriptscriptstyle \Sigma} \A_{_\Sigma}\, ,\qquad \A_{_\Sigma}:=\A|_\Sigma.
%\quad \big(\Is \A\big)\big|_\Sigma=
%\big(\wDs \A\big)\big|_\Sigma=
%\big(\Xs \A\big)\big|_\Sigma=0\, .
$$
\end{proposition}
\begin{proof}
The first set of equalities are obvious using the discussion above the Proposition and nilpotency of the operators $(\Is,\wDs,\Xs)$.
Tangentiality follows by construction because  $\Ds_{[3]}$ and $\D_{[3]}$  are  separately tangential. The remaining two claims are verified by
choosing a scale $g\in c$ and employing the explicit matrix expressions~\nn{doubleD} and~\nn{tripleD} for the double and triple D-operators.
The last display also follows immediately from Remark~\ref{DDD} and the fact that the product $\iota(N)\, \varepsilon(N)$ is unity on boundary tractors.
\end{proof}

Later we will need finer information about the failure of the holographic boundary projector to be a linear projection in the bulk.
\begin{lemma}\label{failure}
Let $(M,c,\sigma)$ be almost Einstein. Then, acting on $\A \in\ker(\Is,\wDs,\Xs)\subset \Gamma\ct^kM[w]$,  $w\neq -k,k-n$, we have
$$
\Pi \A = \Big(1+ \frac{1}{(w+k)(n+w-k)}\, xy\Big)\A\, .
$$
More generally, at any weight $w$
$$
\Ds_{[3]}\, \D_{[3]} \A = \big((w+k)(n+w-k) + xy\big)\A\, .
$$
\end{lemma}
\begin{proof} In order to prove the second identity one has to push all the interior operators to the right where they annihilate $\A$.
First we write
\begin{eqnarray*}
\Ds_{[3]}\, \D_{[3]} \A&=&\Ds_{[2]}\, \Is\!\I \, \D_{[2]} \A\\[1mm]
&=&\big[\Ds_{[2]},\Is\!\I\big]\D_{[2]} \A + \Is\!\I \, (w+k)(d+w-k)\A\\[1mm]
&=&\big[\Ds_{[2]},\Is\!\I\big]\D_{[2]} \A +  (w+k)(d+w-k)\A\, ,
\end{eqnarray*}
where the second line employed Proposition~\ref{WESTPROJECTOR}
while the third  used that $\Is\! \I \A= \A$.

% To this end, from Proposition~\ref{XDalgebra} we write two further identities
%\begin{eqnarray*}
%\Xs\Dt&=&-\Dt\Xs+\X\wDs-\X\Dts-\Big(\frac{d-h}{2}-\N\Big)\, ,\\
%\wDs\X&=&-\X\wDs-\wD\Xs+\Dt\Xs+\Big(\frac{d+h}{2}-\N\Big)\, ,
%\end{eqnarray*}
%which we use as follows\edz{A: I had to play a lot with this proof to get it right. Double check......\\PB checked}
%{hatstotilde}
%\begin{eqnarray*}
%\Ds_{[3]}\D_{[3]} \A&=&-\wDs\Xs\Is\I\X\Dt \A\\[1mm]
%&=&\: \wDs\Xs\I\Is\X\Dt \A\: +\: \:\wDs\X\Xs\Dt \A\\[1mm]
%&=&\: \wDs\Xs\I\Is\X\Dt \A\: +\, (w+k)(d+w-k)\A\, .
%\end{eqnarray*}
%We now need to analyze the first term of the line above. Observing that
%\begin{eqnarray*}
%[\wDs\Xs,\I]&=&x\wDs+\frac 1 h\,  y\Xs+\Is\:=\: x\wDs-\{\Is,\Dt\}\Xs+\Is\, ,
%\end{eqnarray*}
%and
%\begin{eqnarray*}
%{}[\Is,\X\Dt]&=&x\Dt+\frac{1}{h-2}\, \X y\: =\: x\Dt-\X\{\Is,\Dt\} \, ,
%\end{eqnarray*}
%then upon performing some simple algebra one finds
%$$
%\wDs\Xs\I\Is\X\Dt \A=\big(xy-(w+k)\big)\A
%$$
%and thus 
%$$
%\wDs\Xs\Is\I\wD\X \A=\big(xy+(w+k)(d+w-k-1)\big)\A\, .
%$$
%ALTERNATIVE PROOF...\\
%We need to compute now $\wDs\Xs\I\Is\X\Dt \A$. Observe now that, using the result obtained above $\Ds\X\Xs\Dt \A=(w+k)(d+w-k)\A$, we have
%$$
%\wDs\Xs\I\Is\X\Dt \A=[\wDs\Xs,\I\Is]\X\Dt \A
%$$
We claim that the first term above equals  $(xy-(w+k))\A$. To demonstrate this we first use the algebra of Proposition~\ref{Iformalg}, acting on generically weighted tractors, to find
$$
[\Ds_{[2]},\Is\!\I]= \Is \Big(\wDs x+y\frac 1h \Xs\Big)=\Is\Big(x\wDs+\frac 1 h\,   y \Xs\Big) \, .
$$
%(Separately each term on the far right hand side of this identity is singular acting on tractors of weight $w=-\frac d2$ but this singularity is removable for their sum; we may ignore this point in what follows by a simple polynomial continuation argument.)
Recall from Proposition~\ref{hatstotilde} that the exterior double D-operator obeys $\D_{[2]}=\wD\X=-\X\Dt$ so that 
$\wDs\D_{[2]}=-\wD\wDs\X$ and  $\Xs\D_{[2]}=\X\Xs\Dt$. This allows us to use Proposition~\ref{diffsplit}, at generic weights $w$,  to obtain
$\wDs\D_{[2]} \A = -(d+w-k) \wD\A$ and $\Xs \D_{[2]}\A= (w+k)\X \A$. Thereafter using the obvious identities $\Is \D \A = -y \A$, $\Is \X \A = x \A$   and the fundamental identity $[x,y]=h$
gives the result for generic weights. We need the result for all weights. In the process of the above computation, there are terms which could become singular.
However, since the left hand side of the identity to be proved is manifestly a natural formula with coefficients polynomial in the weight, any singularities are removable.
%$$
%[\wDs\Xs,\I\Is]\X\Dt \A=(-x\Is\wDs-\frac 1 h \Is y \Xs )\X\Dt \A=-x\Is\wD\wDs\X+\frac{w+k }{h} y\Is \X
%$$
%Using now that $\Xs\Dt\A=(w+k)\A$ and $\wDs\X\A=(d+w-k)\A$ we get
%$$
%[\wDs\Xs,\I\Is]\X\Dt \A=(xy-(w+k))\A
%$$
\end{proof}

\begin{remark}\label{explicitpi}
It is interesting to explicate the holographic boundary projector $\Pi$ for some $g\in c$ to see how it manages to solve the scale-transversality conditions $\Is \A =\wDs \A=\Xs \A =0$.
(The latter two are no mystery, since they are the content of the west Lemma~\ref{West}.) Taking $\A\in \Gamma\ct^kM[w]$  for $g\in c$ to be
$$
\A\stackrel{g}{:=}\begin{pmatrix}\psi\\ A+\frac1{w+k}\,  \extd \psi\\[1mm] B \\ \phi\end{pmatrix}\, ,\qquad\qquad
$$
an explicit computation (valid away from $w=-k,k-d,k-n$) gives 
\begin{equation}\label{PIA=qA}\qquad\qquad
\Pi \,  \A \stackrel{g}=\ 
\begin{pmatrix}
0\\ 
\widetilde A
\\
0\\
-\frac1{d+w-k}\, \cod \, \widetilde A
\end{pmatrix}\ = q_W\big(\widetilde A\, \big)
\in \ker(\Is,\wDs,\Xs)\, ,
\end{equation}
where 
\begin{equation}\label{6letters}
\widetilde A = 
\big(\iota(n) -\frac{1}{n+w-k}\, \sigma \cod\big)\big(\varepsilon(n)  -\frac1{w+k} \sigma \extd \big)A\, .
\end{equation}
The operator appearing above is exactly the holographic projector solution to the Coulomb gauge extension problem
given in Proposition~\ref{holpro1}. 
\end{remark}

\subsubsection{True forms}
When $w=-k$ the holographic boundary projector~$\Pi$ can no longer be used to solve the scale-transversality conditions.
So instead we solve a weaker problem that suffices for later purposes based around an operation we call $\widehat \Pi$.
\begin{equation}\label{notwell}
 \Omega^kM \ni A
\stackrel{\widehat\Pi}
\longmapsto \,-\frac{1}{n-2k} \  \Ds_{[3]}\I\X q_{(N)} A\in \Gamma\ct^kM[-k]
\, , \quad k\neq \frac n2\, ,
\end{equation}
Here $q_{(N)}:\ce^{k-1}M[w+k]\to \coker(\Xs,\X;\Gamma\ct^kM[w])$ via ({\it cf.} Remark~\ref{obvious})
$$
q_{(N)}A\stackrel{g}=\begin{pmatrix}A\\ * \\ * \\ *\end{pmatrix}\, .
$$
So the map~\nn{notwell} depends on the choice of a coset representative for the $(\Xs,\X)$ cokernel.
However, upon
composition with the canonical map 
$$
\pi_{_\Sigma}:
\Gamma\ct^kM[-k]\longrightarrow \Gamma\ct^kM[-k]\big|\!\big|_{\Sigma}\, ,
$$
this determines a well defined map $\pi_{_\Sigma}\circ \widehat\Pi$.
(Recall that the notation $\bullet\, \big|\!\big|_{\Sigma}$ denotes equivalence classes of sections as in~\nn{barbar}.)
To show this 
we calculate on $A\in\Gamma\ce^k M[-k]$,  $k\neq \frac{n}2$, for some $g\in c$ using
the explicit expressions~\nn{tripleD},~\nn{Iext} and~\nn{X} for the operators in~$\widehat \Pi$,
 and find
$$
\widehat{\Pi}\, A\, 
\stackrel{g}=\, 
\begin{pmatrix}
0\\ \iota(n)\varepsilon(n) A \\ 0 \\ -\frac{1}{n-2k}\ \cod\,  \iota(n)\varepsilon(n) A
\end{pmatrix}\, \quad \mbox{along~$\Sigma$.}
$$
This  is precisely $q^{\scriptscriptstyle\Sigma}_W A_\Sigma$ with $A_{_\Sigma}=A|_\Sigma$.

On the other hand, given a choice of cokernel representative determining $q_{(N)}A$, then~$\widehat{\Pi} A$ gives a representative of  $\Gamma\ct^kM[-k]\big|\!\big|_{\Sigma}$ 
that obeys the analog of~\nn{pihat}
$$
\Is \widehat\Pi  A= 0=\wDs\widehat\Pi A=\Xs \widehat\Pi A \, .
$$
So we have now established an analog of Proposition~\ref{projprop}.

\begin{proposition}\label{projprophat}
Suppose $k\neq \frac n2$. 
%and let $\A\in\Gamma\ct^kM[w]$ with $\A_{_\Sigma}=\A|_\Sigma$ and $\big(\Is \A\big)\big|_\Sigma=0$.
 Then 
$$\pi_{_\Sigma}\circ\widehat\Pi:\Omega^kM\longrightarrow\Gamma\ct^kM[-k]\big|\!\big|_\Sigma\, ,$$
is a well defined map where
$$
\widehat\Pi:=-\scalebox{.9}{$\frac{1}{n-2k}$} \  \Ds_{[3]}\I\X q_{(N)}\, ,
$$
obeys
\begin{equation*}
\Is \widehat\Pi = 0=\wDs\widehat\Pi =\Xs \widehat\Pi \, .
\end{equation*}
Moreover $ \widehat\Pi$ is tangential and if $A\in\Omega^kM$ with $A|_\Sigma=:A_{_\Sigma}\in \Omega^k \Sigma$, then
$$\qquad
\big(\widehat\Pi A\big)\big|_\Sigma = q_W^{\scriptscriptstyle \Sigma} A_{_\Sigma}\, .%\quad \big(\Is \A\big)\big|_\Sigma=
%\big(\wDs \A\big)\big|_\Sigma=
%\big(\Xs \A\big)\big|_\Sigma=0\, .
$$
\end{proposition}
\noindent
Note that a true scale $\tau$ determines a {\it map} $\widehat \Pi_\tau :\Omega^kM\longrightarrow\Gamma\ct^kM[-k]$
with \begin{equation}\label{pihattau}\widehat \Pi_\tau= -\frac{1}{n-2k}\, \Ds_{[3]}\I\X q_{(N)}^\tau\end{equation} where
\begin{equation}\label{qNtau}
q_{(N)}^{\, \tau} A \:\stackrel{\tau}=\: \begin{pmatrix}\, A\, \\0\\0\\0\end{pmatrix}\, .
\end{equation}
The above Proposition applies to the map $\widehat \Pi_\tau$ {\it mutatis mutandis}.

\subsubsection{Dual weight true forms}\label{dforms}
The remaining problem weight is $n+w-k=0$ which is the special case of  the  west Lemma, see~\ref{West} as specialized to the boundary.
Based on Proposition~\ref{bcs2}, at this weight one must consider coclosed boundary data $(A_{_\Sigma},\phi_{_\Sigma})$.

Locally along $\Sigma$ we may write
$$
A_{_\Sigma}=\cod_{_\Sigma} B_{_\Sigma}\qquad\mbox{ and }\qquad 
\phi_{_\Sigma}=\cod_{_\Sigma} \psi_{_\Sigma}\, ,
$$
where $[B_{_\Sigma}]\in\coker\big(\cod_{_\Sigma},\Gamma\ce^{k+1}\Sigma[2k-n+2]\big)$
and $[\psi_{_\Sigma}]\in\coker\big(\cod_{_\Sigma},\Gamma\ce^{k}\Sigma[2k-n]\big)$.
To simplify the discussion, we will assume that this holds globally.
So we take $\big([B_{_\Sigma}],[\psi_{_\Sigma}]\big)$ as our boundary data. We proceed by working with representatives
$(B_{_\Sigma},\psi_{_\Sigma})$; our solutions to the Proca system of the next Section will not depend on this choice.
We then extend these to bulk forms $B\in\Gamma\ce^{k+1}M[2k-n+2]$ and $\psi\in\Gamma\ce^{k}M[2k-n]$ and insert them in a bulk tractor
$$
\B\stackrel{g}{=}q_{(N)} B + q_{(E)} \psi \in \Gamma \ct^{k+2} M[k-n]\, .
$$
(In fact, in the above, we really view  $(B_{_\Sigma},\psi_{_\Sigma})$ as components, in a scale $g_{_\Sigma}\in c_{_\Sigma}$, of a representative 
section of $\coker\!\big(\Xs_{_\Sigma},\Gamma\ct^{k+2}\Sigma[k-n]\big)$. Our final solutions will not  depend on this choice of scale.)
Now let us compute the following using~\nn{tripleD}
$$
\Ds_{[3]}\I\B\stackrel{g}=
\begin{pmatrix}
0\\
\cod\, \iota(n)\varepsilon(n) B\\
0\\
\cod\, \iota(n)\varepsilon(n) \psi
\end{pmatrix}
+ O(\sigma)\, .
$$
Here we have used that, acting on $\Gamma\ce^{k+1}M[2k-n+1]$, the operator $\wiota=-\sigma\cod$. Moreover we have the identity $\cod \iota(n) B = O(\sigma)$.
The latter holds because
 $B$ is a smooth extension of~$B_{_\Sigma}$, whence $\iota(n)B=\sigma C$ for some $C$ so $\cod \iota(n) B = \iota(n) C+O(\sigma)$,
which vanishes along~$\Sigma$ by continuity. 

Along~$\Sigma$, the operator $\cod\, \iota(n)\varepsilon(n) = \cod_{_\Sigma}$ so  the above display equals $q^{\scriptscriptstyle \Sigma}_W (A_{_\Sigma},\phi_{_\Sigma})$ there.
Thus, we see that the above display, modulo $O(\sigma)$, is also independent of our choice of representatives $(B_{_\Sigma}, \phi_{_\Sigma})$.
Hence we  have established the following result.
\begin{proposition}\label{superprop}
Let $(A_{_\Sigma},\phi_{_\Sigma})\in\ker(\cod_{_\Sigma},\cod_{_\Sigma})\subset\Gamma\ce^{k}\Sigma[2k-n]\oplus\Gamma\ce^{k-1}\Sigma[2k-n-2]$
and $\A\in \Gamma\ct^kM[k-n]$ be any extension of $q_W^{\scriptscriptstyle \Sigma}(A_{_\Sigma},\phi_{_\Sigma})$. Then 
$$
\A = \Ds_{[3]}\I\B + O(\sigma)\, ,
$$
for some $\B\in\Gamma \ct^{k+2} M[k-n]$. Moreover, modulo $\sigma$, we may view $\B$ as an element of $\coker\Xs$. 
\end{proposition}

\section{Higher form Proca equations}\label{Procasect}

The Proca equation~\cite{Proca}
$$
\cod\extd A - m^2 A = 0\, ,
$$
for a one-form $A$, first arose in the 1930's as a relativistic extension of Maxwell's equations to describe massive vector excitations.
If the constant $m\neq 0$, then acting with the codifferential $\cod$ immediately yields a constraint
$$
\cod A = 0\, .
$$
Often it is convenient to consider, therefore, the equivalent system
$$
\big(\FL-m^2\big) A = 0 = \cod A\, .
$$
At $m^2=0$ and in Lorentzian signature, these can be viewed as Maxwell's equations in a Feynman gauge choice.

These equations generalise immediately to the case where $A$ is a form of arbitrary degree. They do not enjoy a conformal invariance,
apart from {\it Maxwell systems} (meaning $m^2=0$ and arbitrary degree) in even dimensions with form degree $\frac d2-1$. 
Remarkably, it is possible to unify these systems using the
conformal tractor calculus by coupling to scale through the scale tractor. In~\cite{Gover:2008pt}, the authors considered a tractor vector
 $V^A\in \ct^AM[w]$ satisfying the equations
 \begin{equation}\label{MAX}
 I_A F^{AB} = 0 = D_A V^A\, ,
\end{equation}
where $F^{AB}:=D^A V^B - D^B V^A$ was called a {\it tractor Maxwell field strength}. Both equations enjoy the gauge invariance
$$
V^A\sim V^A + D^A \alpha\, , \qquad \alpha \in \ce M[w+1]\, .
$$
For generic weights and for Einstein structures, the above system of equations describes massive Proca excitations
with masses dictated by the tractor weight~$w$. At $w=1-\frac d2$ the Branson--Deser--Nepomechie equation arises while Maxwell's
equations appear at $w=-1$, see the work~\cite{Gover:2008pt}. 
%There the above system was also generalised to  higher spins~$s$, including massive, massless and partially
%massless excitations (see as well~\cite{Grigoriev:2011gp}). For example, the $s=2$ model describes linearised metric fluctuations: the massless, massive and partially massless gravitons of~\cite{Deser:1983tm}. The higher $s$ systems were studied from this viewpoint  in~\cite{Deser:2001us}.

Here
%rather than totally symmetric, rank~$s$ tractors, 
we study the generalization of the system~\nn{MAX} to higher rank tractor forms on an almost Einstein structure. We start with a tractor field strength formulation of
the higher form Proca systems by letting $\F\in \ct^{k+1}M[w-1]$. 
% we focus on the tractor Maxwell field strength and replace it with a section 
%$\F\in \ct^{k+1}M[w-1]$. Just as Maxwell's equations enjoy both potential and field strength formulations, we will find a similar story here, so 
%take $\F$ to be an  independent field. 
%In the vector example above, $D_A F^{AB}=0=D^{[A} F^{BC]}$.\edz{A: MIght need a slash or constant curvature here...\\PB I vote for slashing} Here square brackets denote antisymmetrization over indices. 
The natural higher form generalization of the field strength formulation of the vector model of~\cite{Gover:2008pt} is
\begin{equation}\label{DDsF}
\Ds \F = 0 = \D \F\, .
\end{equation}
The coupling to scale is then achieved by generalizing the relation $I_A F^{AB}=0$ to
\begin{equation}\label{IsF}
\Is \F=0\, .
\end{equation}
Since $\{\Is,\D\}=-y$,  an integrability condition follows.
\begin{proposition}\label{ynorth}
Equations~\nn{DDsF} and~\nn{IsF} imply
$$
y\F=0\, .
$$
\end{proposition}

By calculating away from distinguished weights, in the preferred interior scale  with
metric $g^o=\sigma^{-2}\bg$ (away from~$\Sigma$), we
can easily characterise the above system, $\F\in\ker(\D,\Ds,\Is)$,
as a Proca one. We record this in the following.
\begin{proposition}\label{northproca}
Let $\F\in\ct^{k+1}M[w-1]$ with $w\neq 2-\frac d2, -k$, and $(M,c,\sigma)$ be almost Einstein. Then, away from~$\Sigma$ in the Einstein scale $g^o$, the equations
\begin{center}
%\shabox{
$
\D \F = \Ds \F = \Is \F =0\, ,
$
%}
\end{center}
capture precisely the Proca equation $\cod \extd A - m^2 A =0$ with $\cod A=0$ when $m=0$,
$A\in\Gamma\ce^{k}M[w+k]$ and   mass--Weyl weight relationship
\begin{equation*}
m^2=-\frac{2\J}{d}\, (w+k)(n+w-k)\, .
\end{equation*}
\end{proposition}

%\edz{A: Note that\\  $(w+k)(d+w-k+1)$\\ $=$ $w(d+w+1)$\\  $-k(k-d-1)$\\
%$=$ $(w+1)(d+w)$\\$-(k-1)(k-d)$\\
%The latter is the famous $s(n-s)$+spin\\
%with $s=-w-1$}
\newcommand{\Js}{\scalebox{.8}{\!\J}}

\begin{proof}
For the preferred interior scale~$g^o\in c$, we have
%\edz{A: Make sure we defined the einstein scale $g^o$ properly}
$$\Is\stackrel{\, g^o}= \begin{pmatrix}
0&0&-1&0\\
-\frac\Js d&0&0&1\\
0&0&0&0\\
0&0&-\frac\Js d&0
\end{pmatrix}\, .
$$
Moreover, in the preferred scale the almost Einstein equations~\nn{almostEinstein}
tell us
$$
\P\, \stackrel{g^o}=\ \frac{\J}{d}\, \degree\, .
$$
Thus, together with the result for $\ker(\D,\Ds)$ given in the northern Lemma~\ref{northlemma} (and calculating carefully the special cases not covered there) we find
$$\cod A =0\, ,$$
and
$$
\Big(\FL+\frac{2\J}{d}\, (w+k)(n+w-k)\Big)\, A = 0\, ,
$$
which completes the proof. 
\end{proof}

\begin{remark}
Since the independent field content of the above field strength tractor formulation of the Proca system appears in the northern slot,
we have depicted it and its Hodge dual model obtained by replacing $\F\to \star \F$, at the northern point of the compass in 
Figure~\ref{stopsign}.
\end{remark}

\begin{figure}\label{stopsign}
\begin{center}
\scalebox{.75}{
\begin{picture}(300,300)(40,40)
\put(80,300){$\ker(\D,\Ds,\Is)$}
\put(220,300){$\ker(\D,\Ds,\I)$}
\put(182,307){$\star$}
\put(160,303){\vector(1,0){50}}
\put(170,303){\vector(-1,0){10}}
\put(110,290){\vector(-3,-4){45}}
\put(255,290){\vector(3,-4){45}}
\put(0,210){$\ker(\Ds\D,\Is\D)\big/\im\D$}
\put(260,210){$\ker(\D\Ds,\I\Ds)\big/\im\Ds$}
\put(299,197){\vector(-3,-4){45}}
\put(66,197){\vector(3,-4){45}}
\put(35,120){$\ker(\Is\D\Ds)\big/\im(\D,\Ds)$}
\put(220,120){$\ker(\I\D\Ds)\big/\im(\D,\Ds)$}
\put(182,127){$\star$}
\put(160,123){\vector(1,0){50}}
\put(170,123){\vector(-1,0){10}}
\put(138,172){\includegraphics[scale=.41]{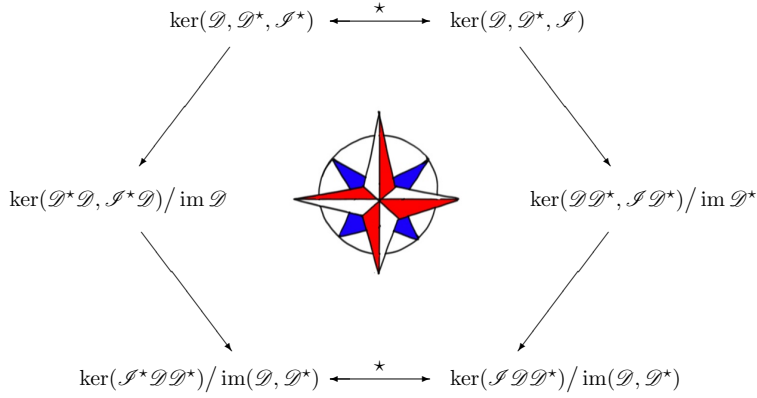}}
\end{picture}
}
\end{center}
\caption{Proca models with independent field content classified by the points of a compass.}\label{stopsign}
\end{figure}

The main focus of the remainder of the Article is a potential formulation of the Equations of Proposition~\ref{northproca}.
Recall that the tractor Maxwell field strength $\F\in\ct^{k+1}M[w-1]$ is subject to~$\D\F=0$.
Hence, by the cohomology result of Proposition~\ref{nocoho},  we can write
$$
\F=\D\A \, ,
$$
for some $\A\in \Gamma\ct^{k}M[w]$ so long as $w\neq1-\frac d2,2-\frac d2,-k,-k+2$. Viewing the potential~$\A$ as the independent field 
content, our system of equations now becomes
$$
\Ds\D \A = 0 =\Is \D \A\, .
$$
Since $\D$ is nilpotent, solutions are only defined up to the gauge invariance
\begin{equation}\label{GF}
\A\sim \A + \D \B \, ,
\end{equation}
for some $\B\in \Gamma\ct^{k-1}M[w+1]$.
Moreover, as the independent field content for the system was shown in Proposition~\ref{northproca} to be described by a degree~$k$ differential
form, this information must now reside in the western slot of the potential~$\A$; see Figure~\ref{stopsign}. 
To capture this  precisely we  firstly remove the gauge freedom~\nn{GF} by requiring 
$$
\Xs \A = 0\, .
$$
For generic weights, there always exists a suitable~$\B$ that achieves this, and the resulting~$\A$ is unique.
Thus, for generic weights, the system $\D \F = \Ds \F = \Is \F = 0$ is equivalent~to
\begin{equation}\label{second}
%\shabox{
y\A =\Is \A=\wDs \A = \Xs \A  = 0\, ,
%}
\end{equation}
which we shall term the {\it tractor Proca equations}.
Note that the first of these is  a Laplace--Robin equation of the type solved for general tractors in~\cite{GWasym} and
Section~\ref{products}. The latter three equations \begin{equation}\label{scaletrans}\Is \A=\wDs \A = \Xs \A  = 0\, ,\end{equation} are the scale-transversality conditions
encountered in our study of boundary conditions for tractor forms in Section~\ref{bctractors}. They underly the
 transversality condition $\cod A=0$.
Further note that the conditions $\wDs \A=0=\Xs \A$ are exactly those of the western Lem\-ma~\ref{West}.
This  reflects Remark~\ref{spinningcompass} which shows that a northern tractor~$\F$ can be written 
as the exterior tractor D-operator acting on a western tractor.
Also, observe that~$\wDs$ is well defined at all weights acting on $\A\in
\ker \Xs$.

We omit the proof of the above equivalence, because on the one hand the details are somewhat involved, and on the other we shall simply
adopt~\nn{second} as the primary system of interest. As we shall see, it uniformly recovers the {\it (higher form)  Proca system}
\begin{equation}\label{higherformproca}
\big(\FL-m^2\big) A = 0 =\cod A\, ,\qquad A\in \Gamma\ce^kM[w+k]\, .
\end{equation}
Thus the tractor Proca equations~(\ref{second}) 
will be the main focus of the latter part of this Article where we study solutions and obstructions.
This is because they uniformly describe the Proca system (or its massless Maxwell limit), as we shall now chronicle.

\begin{proposition}\label{westproca}
For $\A\in\Gamma\ct^{k} M[w]$ and $(M,c,\sigma)$ almost Einstein, away from~$\Sigma$ in the Einstein scale $g^o$, the tractor Proca equations
$
y \A = \Ds \A = \Xs \A   = \Is \A =0\, ,
$
precisely capture the Proca equations $\cod \extd A - m^2 A =0=\cod A$
for $A\in\Gamma\ce^{k}M[w+k]$ with a mass--Weyl weight relationship
\begin{equation}\label{mWw}
m^2=-\frac{2\J}{d}\, (w+k)(n+w-k)\, .
\end{equation}
\end{proposition}

\begin{proof}
The proof is identical in method to that of Proposition~\ref{northproca} except that the western Lemma~\ref{West} is called upon
instead of its northern counterpart.
\end{proof}

\begin{remark}
An important feature of the
tractor Proca equations~\nn{second} is their conformal invariance in the interior. The two equations $y\A=0=\Is\A$ 
prescribe how the system couples to the defining  scale;  these arise from a canonical conformally invariant pairing of scale (mediated by the scale tractor)
and derivatives of the fundamental fields. This is striking, since the Proca system, being massive, is traditionally viewed as a non-conformally
invariant system. 

Furthermore, the tractor system automatically includes massless and massive models:
When $w=-k$, Equation~\nn{mWw} gives $m^2=0$ so $\FL A = 0 = \cod A$ which is the potential version of the higher form Maxwell system in a Feynman gauge choice. 
\end{remark}

\begin{remark}
Notice, that $w=k-n$ also recovers the Maxwell system. This underlies an important duality of the tractor approach that we will utilize when studying solutions.
\end{remark}

\newcommand{\cc}{{\mathcal C}}

\begin{remark} 
Corollary~\ref{northcor} of the northern Lemma~\ref{northlemma} shows that, for generic weights,  $\F\in\Gamma\ct^{k+1}M[w-1]$ subject to $\D\F=0=\Ds\F$ can be written as $\F=\Ds\D\B$ for some $\B
\in\Gamma\ct^{k+1}M[w+1]$. Since $\Is\F=0$, we obtain the model depicted at the southern compass point in Figure~\ref{stopsign}
with only a single equation of motion
$$
\Is\Ds\D\B=0\, .
$$
This model is  strikingly similar to the model of $(p,q)$ form K\"ahler electromagnetism proposed in~\cite{Cherney:2009vg} where the Dolbeault operators
$(\partial,\overline\partial)$ 
play the {\it r\^ole}  of $\D$, $\Ds$ and the K\"ahler trace corresponds to $\Is$.
Just as in that case, the model here has a pair of gauge invariances 
$$
\B\sim\B + \D \cc+\Ds\cc'\, ,
$$
with $\cc\in \Gamma\ct^{k}M[w+2]$ and $\cc'\in \Gamma\ct^{k+2}M[w+2]$.
The same Corollary also shows that we may generically fix this invariance by choosing $\X\B=0=\Xs\B$. In that case the equations
of motion are 
$$\Is \Ds\D\B=\Xs\B=\X\B=0\, .$$
Then, in the Einstein scale $g^o$ and away from~$\Sigma$, one recovers the Proca system with mass--Weyl weight relationship~\nn{mWw}.

\end{remark}

\begin{remark}
To complete the classification of independent field content describing Proca systems according to east side of the diagram in Figure~\ref{stopsign},
one applies tractor Hodge duality.
\end{remark}

Having surveyed formulations of the  model we now turn to detailed solutions which are facilitated by the tractor Proca equations.

\subsection{Solution generating operators for differential forms}\label{formssol}

We seek to solve the following problem on a Poincar\'e--Einstein structure $(M,c,\sigma)$.
\begin{problem}\label{laylanguage}
Given $A|_\Sigma=A_\Sigma\in \Gamma \ce^k \Sigma[w_0+k]$, $w_0\neq k-n$,  and an arbitrary extension $A_0\in \Gamma\ce^kM[w_0+k]$
find $A_i\in \Gamma\ce^kM[w_0+k-1]$  such that
$$
A^{(\ell)}= A_0 + \sigma A_1 + \sigma^2 A_2 +\cdots + O\big(\sigma^{\ell+1}\big)
$$
solves asymptotically the Proca system 
 \begin{equation}\label{Procao}
\big(\sigma^2\, \cod^o \extd^o\ - (w_0+k)(n+w_0-k)\big) A = 0 = \sigma\cod^o A\, ,
\end{equation}
off $\Sigma$, for $\ell\in {\mathbb N}\cup \infty$ as high as possible.
\end{problem}

Equation~\nn{Procao} is the higher form Proca system~\nn{higherformproca}
\begin{equation}\label{interiorProca}
\big(\cod\,\! \extd -(w_0+k)(n+w_0-k)\big) A = 0 = \cod A\, ,
\end{equation}
where $\cod$ and $\extd$ are given in the Einstein scale.
Moreover recall that  $\cod^o$ and  $\extd^o$ are the usual interior and exterior differential operators on form densities determined by the Levi-Civita connection of the Poincar\'e--Einstein scale defined in~\nn{d0delta0}. Although these recast $\cod$ and $\extd$ in terms of a scale that extends to the boundary,  they are nevertheless singular along~$\Sigma$. However the
operators $\wiota$ and $\wepsilon$ of Definition~\ref{wdef} are well-defined everywhere and agree with  $-\sigma \cod^o$ and $-\sigma \extd^o$  off~$\Sigma$.
Therefore, 
Equations~\nn{Procao}  extend to~$\Sigma$ as
\begin{equation}\label{usemelater}
\big(\wiota\, \wepsilon - (w_0+k)(n+w_0-k)\big) A = 0 = \wiota A\, .
\end{equation}
In Lemma~\ref{aerm} these are proven to be equivalent to the tractor Proca equations~\nn{second}.

Evaluated along~$\Sigma$, these equations say 
\begin{equation}\label{alongsigma}(w_0+k)(n+w_0-k) 
\big[\big(\iota(n)\varepsilon(n)-1 \big)A\big]\big|_\Sigma=0 =
(d+w_0-k)\big[\iota(n) A\big]\big|_\Sigma\, .\end{equation}
This is consistent with 
the requirement that $A_0$, and therefore $A^{(\ell)}$, extends $A_\Sigma\in \Gamma\ce^{k}\Sigma[w_0+k]$ to~$M$.

The normal derivatives of $A$ along~$\Sigma$ are also
determined; understanding the details is critical to linking Problem~\ref{laylanguage} to its equivalent tractor formulation in 
Problem~\ref{solprob} below.  
\begin{proposition}
Solutions $A\in \Gamma\ce^kM[w_0+k]$ to Problem~\ref{laylanguage} satisfy
\begin{equation}\label{normalcomp}
\Big(\nabla_n\big[\iota(n) A\big]\Big)\Big|_\Sigma = \frac{1}{n+w_0-k}\, \cod_{_\Sigma} A_{_\Sigma}\, ,
\end{equation}
and, when $w_0\neq 1-\frac d2$, 
\begin{equation}\label{formrobin}
\big[\big(\nabla_n-\varepsilon(n) \nabla_n\,  \iota(n) - w_0 H^g) A\big]\big|_\Sigma=0\, .
\end{equation}
\end{proposition}

\begin{proof}
In essence,
equation~\nn{normalcomp} was already derived in Section~\ref{bctractors}; indeed at $w_0=k-d$ it follows by combining Equations~\nn{normcomp} 
and \nn{phideltaA}. At all other weights $w_0\neq k-n$, by virtue of $\wiota A=0$, we have
$$
\Big(\nabla_n\big[\iota(n) A\big]\Big)\Big|_\Sigma=\frac{1}{d+w_0-k}\, \big(\cod  A\big)\big|_\Sigma= \frac{1}{n+w_0-k}\, \cod_{_\Sigma} A_{_\Sigma}\, ,
$$
where the last equality used Lemma~\ref{deltadelta}.

To derive equation~\nn{formrobin}, we first
use that $\iota(n)A = - \sigma \phi$, for some $\phi$, to rewrite $\wiota A = 0$ as
$$
\cod A = -(d+w_0-k) \phi\, .
$$
Then using these facts as well as $(\nabla_n \sigma)|_\Sigma=1$,
we act on the first equation of~\nn{Procao} with~$\nabla_n$ and evaluate the result along~$\Sigma$.
 Employing the operator identities
\begin{eqnarray*}
 \{\cod,\varepsilon(n)\}&=&\nabla_n- \sigma(\P-\J)+\rho(\degree-d)\, ,\nonumber\\[2mm]
 \{\iota(n),\extd\}&=&\nabla_n- \sigma\P-\rho\degree\, ,
\end{eqnarray*}
which are readily obtained from the parallel conditions~\nn{almostEinstein}, plus $\rho|_\Sigma=-H^g$, allows that equation to be written as
\begin{equation}\label{yform}
(d+2w_0-2)\big[(\nabla_n-w_0H^g)A+ \varepsilon(n)\phi\big]\big|_\Sigma=0\, .
\end{equation}
In Section~\ref{bctractors} we showed that $\phi|_\Sigma$ was the southern slot of a boundary west tractor so that
$$
\phi|_\Sigma = -\frac{1}{n+w_0-k}\, \cod_{_\Sigma} A_{_\Sigma} = 
-\frac{1}{n+w_0-k}\big(\cod\,  \iota(n) \varepsilon(n) A\big)\big|_\Sigma\, .
$$
This allows us to compute the last term of Equation~\nn{yform} as follows
$$
\big(\varepsilon(n) \cod \iota(n) \varepsilon(n) A\big)\big|_\Sigma\!=\!
\big(\varepsilon(n) \big[\cod   - \{\cod,\varepsilon(n)\}\iota(n) \big]A\big)\big|_\Sigma
\!=\!
-\big(\varepsilon(n)\big[(d+w_0-k)\phi   + \nabla_n\iota(n) A\big] \big)\big|_\Sigma\, .
$$

The combination of the first operator identity above and Lemma~\ref{deltadelta} can be used to show
$$
\big[\varepsilon(n)\cod A_0\big]_\Sigma=(d+w_0-k) \big[\varepsilon(n)\nabla_n\iota(n) A_0\big|_\Sigma\, .
$$
In the first step we used $\{\iota(n),\varepsilon(n)\}=1-2\rho\sigma$ while for the second we used the first operator
identity above as well as the formula given for $\cod A$. Thereafter, elementary algebra gives the quoted result for Equation~\nn{formrobin}.
\end{proof}

\begin{remark}
By construction $\nabla_n-\varepsilon(n) \nabla_n\,  \iota(n) - w_0 H^g$ as an operator on forms in the kernel of $\iota(n)$  is conformally invariant 
along~$\Sigma$ so gives a differential 
forms generalisation of the conformal Robin operator~\cite{cherrier,BrGoOps}. Moreover, in the form given by Equation~\nn{yform}
of the above proof, it can be written as
$$
(d+2w_0-2)\delta_R \A\big|_\Sigma=-y\A\big|_\Sigma=0\, ,
$$
where $\A\in \ker\Xs \subset \Gamma\ct^k M[w_0]$ is a tractor given  by
the expression~\nn{XAform}
and $\delta_R$ is the Robin operator on tractor forms given in Equation~\nn{Robinform}.
Note that
the weight $w_0=1-\frac d2$ is the first in a series of exceptional weights that will be studied 
in detail.
\end{remark}

\begin{remark}\label{wtd}
In Equation~\nn{interiorProca}
%Away from~$\Sigma$ in the $g^o$ scale, Equation~\nn{Procao} gives
%$\big(\FL^{g^o}-m^2) A = 0= \cod^{o}A$ with 
the parameter
$m^2$ is $(w_0+k)(n+w_0-k)$. If we trivialise the density bundles with respect to the Einstein scale so that $A$ is a differential form, then this interior system is
unchanged under the replacement
$$
w_0\mapsto -n-w_0\, .
$$
This {\it weight duality} leaves the parameter $m^2$ invariant.

Although weight duality is seen to be a symmetry of the interior
equations, it acts non-trivially on the boundary data.  Shortly we
will use this symmetry as a map between solutions with distinct
boundary behaviours.
\end{remark}

The case $w_0=k-n$ has been excluded from the statement of
Problem~\ref{laylanguage} because the boundary problem is canonically
of mixed Dirichlet-Neumann type, see Proposition~\ref{bcs2}. All
weights are, however, handled uniformly in the tractor statement of
our extension problem, to follow.

\begin{problem}\label{solprob}
Given $\A|_\Sigma$ isomorphic to $\A_{_\Sigma}\in\ker(\wDs_{_\Sigma},\Xs_{_\Sigma})\subset \Gamma\ct^k\Sigma M[w_0]$ and
an
arbitrary extension $\A_0\in \Gamma\ct^k M[w_0]$ of this subject to
$$
y\A_0=\Is \A_0 = \wDs \A_0 = \Xs \A_0 = O(\sigma)\, ,
$$
find $\A_i\in \Gamma\ct^k M[w_0-i]$ such that
$$
\A^{(\ell)}:= \A_0 + \sigma \A_1 + \sigma^2 \A_2 + \cdots + O\big(\sigma^{\ell+1}\big)
$$
solves the tractor Proca equations $$y \A=O(\sigma^{\ell})\, ,\qquad \Is \A = \wDs \A = \Xs \A = 0\, ,$$ off $\Sigma$, for $\ell\in {\mathbb N}\cup \infty$ as high as possible.
\end{problem}

Regarding the equivalence of  Problems~\ref{laylanguage} and~\ref{solprob}:
recall that Proposition~\ref{westproca}
showed that the tractor Proca equations  yield the Proca system  uniformly for all weights in the interior.
Moreover, the tractor Proca equations extend to the boundary and therefore determine boundary conditions
naturally associated with the interior equations.

We are now well-positioned to solve these Equations because
the first Proca equation, $y\A=0$, is solved by the solution generating operator technique of~\cite{GWasym} which is explained in Section~\ref{EXTENSION}.
Moreover,
the final three scale-transversality equations $\Is \A = \wDs \A = \Xs \A = 0$, are solved by the holographic projector~$\Pi$ of Proposition~\ref{projprop} for weights
$w_0\neq -k,k-n$. It remains to show compatibility of these results and handle the special weights. The latter encompass some of the most interesting features
of the theory related to obstructions to smoothness, see Section~\ref{OBSTQD}.

A critical ingredient, especially for establishing the abovementioned compatibility, is the following result. 

\begin{lemma}\label{moving}
$$
[x,\Ds_{[3]}]=0=[y,\Ds_{[3]}]\, ,
$$
\vspace{.1cm}
$$
[x,\D_{[3]}]=0=[y,\D_{[3]}]\, .
$$
Hence
$$
y \, \Pi = \Pi\,  y\, .
$$
 \end{lemma}

\begin{proof}
As observed in Remark~\ref{DDD}, the equalities on the left hand side  hold, because the exterior and interior double D-operators
obey the Leibnitz rule and their commutators with $x$  are proportional to $\Xs\Is$ and $\I\X$,
respectively. 
%This in principle proves also the equalities on the right hand side above by the symmetry
%of the algebra of Proposition~\ref{Iformalg} under interchange $x\leftrightarrow y$, $\X\leftrightarrow \D$
%and $\Xs\leftrightarrow \Ds$.

The commutators on the right hand side only require a simple application of the algebra of Proposition~\ref{Iformalg}. For example,
%\begin{eqnarray*}
%y\, \wDs\Xs\Is &=&\Ds\frac1{(h+2)(h-2)}\big(h\Xs y + 2 x \Ds \big) \Is\\[2mm]
%&=&\frac{h+2}{h(h+4)}\Ds\Xs\Is y-\frac{4}{h(h+2)(h+4)}\Ds\Xs \Is y 
%\: =\:  \wDs\Xs\Is y\, .
%\end{eqnarray*}
%\edz{A: Check this claim...}
\begin{eqnarray*}
y\, \Ds_{[3]} &=&-y\Xs\Dts \Is  \:=\: -\frac{h}{h-2}\, \Xs y\,\Dts\Is\\[2mm]
&=&-\frac{1}{h-2}\, \Xs y\,\Ds\Is
\:=\: -\Xs\Dts\Is y \:=\: \Ds_{[3]} y\, .
\end{eqnarray*}
The apparent singularities above for exceptional weights are all removable.
%The apparent singularity at $h_0=4$ is obviously removable; the operator
%$\frac{1}{h-2}$ is the one needed to put a tilde in the last step to $\Ds$
\end{proof}

\begin{remark}
The above Lemma is easily extended to demonstrate that the  commutators of the interior and exterior  triple D-operators
with $x^\alpha$ (for any $\alpha\in {\mathbb C}$) and the log density~$\log x$ all vanish.
\end{remark}

The above Lemma and pair of Remarks establish an all orders solution to  the tractor Proca equations for generic weights.

\begin{theorem}\label{BIGTHEOREM}
For $h_0:=d+2w_0\notin {\mathbb Z}_{\geq 2}$  and $w_0\neq -k, k-n$, Problem~\ref{solprob} has an order  $\ell=\infty$
solution given by
\begin{equation}\label{generic}
\A=\Pi \ \colon K^{h_0}(z) \colon\  \A_0\, .
\end{equation}
\end{theorem}

\subsubsection{Solutions of the second kind}\label{sol2}

The Proca equations, being second order, admit a second homogeneous solution. We consider therefore a more general problem $y\A=0$
where $\A=\sigma^\alpha \big(\overline \A_0+O(\sigma)\big)$ for some $\alpha\in {\mathbb C}$. Recall from Proposition~\ref{WESTPROJECTOR} that
\begin{equation}\label{free}
[y,x^\alpha]=-\alpha\, x^{\alpha-1}(h+\alpha-1)\, ,
\end{equation}
from which it follows, if $\A_0$ is of weight $w_0$, that 
$$
\alpha=0\quad \mbox{ or } \quad\alpha = h_0-1\, ,
$$
in order that $\overline \A_0$ is $O(\sigma^0)$.
Since the case $\alpha=0$ was solved above in Theorem~\ref{BIGTHEOREM},
this brings us to the following Problem.

\begin{problem}\label{solprob2}
Given $\overline\A_0|_\Sigma$ isomorphic to $\overline\A_{_\Sigma}\in\ker(\wDs_{_\Sigma},\Xs_{_\Sigma})\subset \Gamma\ct^k\Sigma [-w_0-n]$ and
an
arbitrary extension $\overline\A_0\in \Gamma\ct^k M[-w_0-n]$ of this subject to
$$
y\, \overline\A_0=\Is \, \overline\A_0 = \wDs \, \overline\A_0 = \Xs \, \overline\A_0 = O(\sigma)\, ,
$$
find $\overline\A_i\in \Gamma\ct^k M[-w_0-n-i]$ such that
$$
\overline\A^{(\ell)}:= \sigma^{h_0-1}\Big(\, \overline\A_0 + \sigma \, \overline\A_1 + \sigma^2 \, \overline\A_2 + \cdots + O\big(\sigma^{\ell+1}\big)\Big)\in \Gamma\ct^kM[w_0]
$$
solves the tractor Proca equations $$y \, \overline\A=O(\sigma^{\ell+h_0-1})\, ,\qquad \Is \, \overline\A = \wDs \, \overline\A = \Xs \, \overline\A = 0\, ,$$ off $\Sigma$, for $\ell\in {\mathbb N}\cup \infty$ as high as possible.
\end{problem}

Problem~\ref{solprob2} amounts to 
Problem~\ref{laylanguage} but instead with a solution of the form
\begin{equation}\label{laidagain}
\overline A^{(\ell)}= \sigma^{h_0-1}\Big(\, \overline A_0 + \sigma \, \overline A_1 + \sigma^2 \, \overline A_2 +\cdots + O\big(\sigma^{\ell+1}\big)\Big)\, .
\end{equation}
Solutions to this problem can be obtained from solutions to Problem~\ref{laylanguage}
 by what we term the {\em scale duality map},
which is related to the weight symmetry alluded to above. The scale duality map couples
the defining scale $\si$ with an existing solution to yield a new
solution of the same mass, as in the following result. 
\begin{theorem}\label{swap}[Scale Duality.]
Suppose $\A\in\Gamma\ct^kM[-w_0-n]$, 
$w_0\neq -k,k-n$, solves
 $$y \A=O(\sigma^{\ell})\, ,\qquad \Is \A = \wDs \A = \Xs \A = 0\, .$$
 Then 
 $$\overline \A = \sigma^{h_0-1}\Pi\A$$
 solves
  $$y \, \overline\A=O(\sigma^{\ell+h_0-1})\, ,\qquad \Is \, \overline\A = \wDs \, \overline\A = \Xs \, \overline\A = 0\, .$$
\end{theorem}

\begin{proof}
Firstly from Equation~\nn{free}, acting on $\Pi\A$,  we have $[y,\sigma^{h_0-1}]=0$
and from Lemma~\nn{moving} $y\, \Pi=\Pi\, y$ and $x\, \Pi=\Pi\, x$ so
$$
y\, \overline \A \, = \, \sigma^{h_0-1} y \, \Pi\,  \A \, = \, \sigma^{h_0-1} \Pi\,  y\,  \A\,  = \, \sigma^{h_0-1} \Pi\,  O(\sigma^{\ell}) \, = \, O(\sigma^{\ell+h_0-1})\, .
$$
Also
  we have $\overline \A = \Pi\, \sigma^{h_0-1}\A$, so automatically $\Is \, \overline\A = \wDs \, \overline\A = \Xs \, \overline\A = 0$.
 \end{proof}
 
 \begin{corollary}\label{swapc}
When $w_0\neq -k,k-n$, there is a bijection between solutions of the
weight $-w_0-n$ version of 
Problem~\ref{solprob} and solutions of
Problem~\ref{solprob2} given by
$$
\A\longmapsto \, \overline \A = \sigma^{h_0-1}\Pi\A\, .
$$
 \end{corollary}
 
 \begin{proof}
 It only remains to establish that the boundary conditions are correctly mapped from those of one Problem to the other.
 Observing that $\overline \A_0=\Pi \A_0$ (because $\Pi \, x = x\,  \Pi$) the conditions $\Is \, \overline \A_0 = \wDs\, \overline \A_0 = \Xs \, \overline \A_0=O(\sigma)$ hold.
 Similarly $y \, \overline \A_0 =y\, \Pi \A_0 = \Pi \, y \A_0 = \Pi \, O(\sigma)=O(\sigma)$.
 
 Remembering that solutions are defined up to $O(\sigma^{\ell})$, the inverse map is 
 $$
 \overline \A \longmapsto \A = \sigma^{1-h_0} \Pi \, \overline \A\, ,
 $$
 because
 \begin{equation}
 \sigma^{1-h_0}\, \Pi\: \sigma^{h_0-1}\, \Pi \A \, =\,  \Pi^2 \A =
 \Big(1+ \frac{1}{(w+k)(w+k-n)}\, xy\Big)^2\A= \A+O(\sigma^{\ell+1})\, ,
 \end{equation}
 where the second equality used Lemma~\ref{failure}.
  \end{proof}
 
 Although not strictly needed for our subsequent discussion, it is worth noting that a strong global statement is available.
 \begin{theorem}\label{swapg}[Global scale duality.]
 Given a global solution~$\A$ to the tractor Proca equations  of weight $w_0\neq -k,-k-n$, then a solution of weight $-w_0-n$ is
 $$
 \overline \A = \sigma^{1-h_0}\A\, .
 $$
 \end{theorem}
 
 \begin{proof}
 The equation $y\, \overline \A=0$ is again obvious from
 Equation~\nn{free}. Moreover, by virtue of Lemma~\nn{failure}, $\Pi
 \A = \A$ on global solutions. Thus $\sigma^{1-h_0}\A=
 \sigma^{1-h_0}\Pi \A = \Pi \sigma^{1-h_0} \A$ so $\Is \, \overline \A =
 \wDs \, \overline \A = \Xs\, \overline \A =0$.
 \end{proof}
 
\begin{remark}\label{drem}
Note the solutions $\A$ and $\overline \A = \sigma^{1-h_0}\A$ in the
Theorem have the {\em same mass}, in that they solve the Proca system
\nn{proca1} for the same fixed parameter value $m^2$, see Remark
~\ref{wtd}. It follows from the analysis in this work that on a Poincar\'e--Einstein manifold,
for generic $m^2$ (and from the established theory for problems of
this sort~\cite{MaMe,AG-BGops}), one expects global solutions of the
form
$$
\A_{\rm global}=\A + \si^{h_0-1}{\, \overline{\A}}\, ,
$$ where $\A|_\Sigma$ may be viewed as the ``Dirichlet data'' and
$\overline{\A}|_\Sigma$ the ``Neumann data'' of the solution~$\A_{\rm
  global}$. Thus the scale duality map takes the weight $w_0$ solution  $\A_{\rm
  global}$ of the Proca system \nn{proca1} to the weight $-w_0-n$ solution
$$
\A'_{\rm global}=\si^{1-h_0}\A_{\rm global}={\overline{\A}} + \si^{1-h_0}\A,
$$ and we see the that  {\it r\^oles} of the Dirichlet and Neumann parts are swapped
in the new solution (of the {\em same} interior Proca system \nn{proca1}).

Note that away from the boundary we may work in the defining scale
$g^o=\si^{-2}\bg$ and trivialise density bundles using $\si$. Then, in
this trivialisation, $\si$ is represented by the constant function 1 and
so in the interior the two solutions $\A_{\rm global}$ and
$\A'_{\rm global}$ (which are sections of true form bundles) then
appear indistinguishable. The solutions $\A_{\rm global}$ and
$\A'_{\rm global}$ differ by their boundary behaviour, but this
information is lost when we work on the interior in the ``$\si=1$'' scale.
\end{remark}

 \subsubsection{Log solutions}\label{logsect}

To treat  weights $w_0$ such that $h_0\in{\mathbb N}$ we need to draw log-type solutions into the picture.
This is captured by the following Problem; see~\cite{GWasym} for the definition of the log densities $\log\sigma$
and $\log \tau$.

\begin{problem}\label{logfire}
Let $h_0=d+2w_0\in {\mathbb Z}_{\geq 2}$ and $w_0\neq -k,k-n$.	
Then, given $\A|_\Sigma$ isomorphic to $\A_{_\Sigma}\in\ker(\wDs_{_\Sigma},\Xs_{_\Sigma})\subset \Gamma\ct^k\Sigma [w_0]$ and
an
extension $\A_0\in\Gamma\ct^kM[w_0]$ of this satisfying
$$
y\A_0= \Is \A_0 = \wDs \A_0 = \Xs \A_0 = O(\sigma)\, ,
$$
find $\A_i\in \Gamma\ct^k M[w_0-i]$ and $\overline \A_i\in \Gamma\ct^k M[-n-w_0-i]$ such that
\begin{eqnarray*}
\A^{(\ell)}&:=& \big(\A_0 + \sigma \A_1 + \sigma^2 \A_2 + \cdots\big) + \sigma^{h_0-1}(\log\sigma-\log\tau)\big(\, \overline{\A}_0+\sigma \, \overline{\A}_1+\sigma^2 \, \overline{\A}_2+\cdots \big)
\\[2mm]&+& O\big(\sigma^{\ell+1}\big)\, +\, O\big(\sigma^{\ell+1}\log(\sigma/\tau)\big)
\end{eqnarray*}
solves the Proca equations $$y \A=O\big(\sigma^\ell\big)+O\big(\sigma^\ell\log(\sigma/\tau)\big)\, ,\qquad\Is \A = \wDs \A = \Xs \A = 0\, ,$$ off $\Sigma$, for $\ell\in {\mathbb N}\cup \infty$ as high as possible.

When $h_0=1$, we set $\A_0=0$ and take non-vanishing initial data $\overline\A_0|_\Sigma$.
\end{problem}

A key aspect of this Problem is solved in~\cite{GWasym}, which explains the set up and its special case $h_0=1$.
Recall from there (see also  Equation~\nn{besselop}) that the function $K^{h_0}(z)$ 
characterizing the solution generating operator $\colon K^{h_0}(z) \colon$ can be obtained by solving the ordinary differential equation
$$
z K''-(h_0-2) K' + K=0\, ,
$$
in a series expansion by the Frobenius method. However, for weights $h_0=2,3,\ldots$ the power series solution breaks down. For $h_0=2,4,\ldots$ 
the obstruction to a power series solution $\big[y^{h_0-1} \A_0\big]\big|_\Sigma$ vanishes for almost Einstein structures~\cite{GWasym}. When $h_0=3,5,\ldots$,
power series solutions are no longer possible, but solutions beyond the obstruction are obtainable by introducing a second, nowhere vanishing,  scale $\tau\in 
\Gamma\ce M[1]$ and including terms $\sigma^k (\log \sigma - \log \tau)$ in  the series expansion. Since~$\tau$ is arbitrary, there is limited
control over its algebra with differential operators $y, \D, \Ds$. In~\cite{GWasym}, this difficulty was circumvented by carefully 
ordering operators. The extension problem $y\A=0$ was solved via 
$$
\A={\mathcal O} \A_0\, ,
$$
where the new
 solution generating operator~$\cO$ is given by
\begin{eqnarray}\label{O}
\cO&:=&\colon F_{h_0-2}(z)\colon -\frac{\colon z^{h_0} \, \H(z)\colon}{(h_0-1)!(h_0-2)!}\nonumber\\
&-&\frac
{x^{h_0-1} \log x \ \colon K^{\oh_0}(z)\colon\  y^{h_0-1}\  -\ x^{h_0-1} \ \colon K^{\oh_0}(z)\colon\, \big(  y^{h_0-1}\,\log \tau\big)_{\rm W}}{(h_0-1)!(h_0-2)!} \, .
\end{eqnarray}
Explicit expressions for the order $h_0-2$ polynomial~$F_{h_0-2}$ and power series $B(z)$ can be found in~\cite{GWasym}.
The notation $\big( \bullet  \big)_W$ denotes the Weyl ordering $\frac12 \big\{  y^{h_0-1},\log \tau\big\}$.
In~\cite{GWasym} it was shown that the above operator only depends on  the log densities $\log\sigma$ and $\log\tau$
in the combination $\log\sigma-\log\tau=\log(\sigma/\tau)$ and $\sigma/\tau$ is a $C^\infty$ defining function for~$\Sigma$. This implies
that the operator $\cO$ maps smooth sections of $\ct^kM[w]$ to sections of  $\ct^kM[w]$ whose failure to be smooth is controlled by the term of order $(\sigma/\tau)^{h_0-1}\log(\sigma/\tau)$.

The key feature of $\cO$ is that keeping terms up to order in $x^\ell$ and $x^\ell\log x$, and denoting this $\cO^{(\ell)}$,  one has the operator statement
$$
y\, \cO^{(\ell)} = O(\sigma^\ell)+O(\sigma^\ell\log\sigma)\, .
$$
(Note the log term is only present for $\ell\geq h_0-1$.) This machinery can now be combined with the tools developed above to handle log solutions for forms.
We are first focussing on the cases $w_0\neq -k,k-n$ so that a version of the holographic projector can still be employed.

\begin{theorem}\label{mainl}
For weights $w_0\neq -k,k-n$ and $h_0=2,3,4,\ldots$, Problem~\ref{logfire} has an $\ell=\infty$ solution given by
$$\A=\scalebox{.9}{$\frac{1}{(w+k)(n+w-k)}$}\ \Ds_{[3]} \, \cO\, \D_{[3]} \A_0 \, .$$
\end{theorem}

\begin{proof}
To show $y\A=O(\sigma^\ell)+O(\sigma^\ell\log\sigma)$, we employ Lemma~\ref{moving} to write
\begin{eqnarray*}
y\ \Ds_{[3]} \, \cO^{(\ell)}\, \D_{[3]} \A_0
&=& \Ds_{[3]} \ y \,  \cO^{(\ell)}\ \D_{[3]} \A_0\\[1mm]
&=&\Ds_{[3]} \, \big(O(\sigma^\ell)+O(\sigma^\ell\log\sigma)\big)\\[1mm]
&=&O(\sigma^\ell)+O(\sigma^\ell\log\sigma)\, .
\end{eqnarray*}
By construction $\A\in \ker (\Is,\wDs,\Xs)$ so it remains only to verify that $\A|_\Sigma = \A_0|_\Sigma$. This is clear from Lemma~\ref{moving} 
along with  the form of $\cO$ in Equation~\nn{O}.
\end{proof}

\begin{remark}
As shown in~\cite{GWasym}, when $h_0=2,4,\ldots$ and $(M,c,\sigma)$ is almost Einstein (as here), the coefficients of the log terms in $\cO$ vanish.
So the solutions are still of the type given by~\nn{O} without the terms displayed on the second line.
\end{remark}

A solution of the form $\A=\Pi\, \cO\A_0$ could also have been used in the above, but the expression given adapts easily to the exceptional cases $w_0=-k$ or $w_0=k-n$. These solutions necessarily differ by some amount of a solution of the second kind which we now consider:
The point here is that at the dual weight $-w_0-n$, we have dual $h$-weight $1-h_0$ which is necessarily negative. 
For those values, there is no obstruction to Dirichlet-type solutions. {\it I.e.}, Problem~\ref{solprob} admits an $\ell=\infty$ solution at the dual weight.
Hence Theorem~\ref{swap} applies and generates an $\ell=\infty$ solution of the second kind.

It remains to discuss the case of weights $w_0$ such that  $h_0=1$. This value of $h_0$  is invariant under the weight duality $w_0\to -n-w_0$.
As found in~\cite{GWasym}, there are two solutions. The one of the first kind has now  as its leading behaviour $\log(\sigma/\tau)$. 
Thus, there is no interesting boundary operator appearing as an obstruction to smooth solutions. The solution of the second kind now has leading behaviour
$\sigma^{h_0-1}=1$ here, so is in fact Dirichlet.
For the case of true forms $w_0=-k$, (and their duals at $w_0=k-n$) this weight corresponds to middle boundary forms of degree $-n/2$ for $n$ even.

\subsubsection{True forms}

We now treat the case where our Dirichlet boundary data is given by  a true form  $A_\Sigma\in \Omega^k \Sigma=\Gamma\ce^k\Sigma[0]$ so $k\in \{0,1,\ldots , n\}$ which 
corresponds to a west tractor of weight $w_0=-k$ and $h_0\in\{n+1,n,\ldots,1-n\}$. 
Thus the cases where the degree $k\in \{0,1,\ldots, \big\lfloor\frac d2\big\rfloor\}$ potentially involve log terms.
Since the boundary data is now given in terms of a true form, we modify our problem slightly.

\begin{problem}\label{trueproblem}
Let $h_0=d+2w_0\in {\mathbb Z}_{\geq 2}$ and $w_0= -k$.	
Take $A_{_\Sigma}:=A|_\Sigma\in\Omega^k\Sigma$ and
an extension $A_0\in\Omega^kM$ of this.
% satisfying\edz{A: Are these bcs forced? What about the other half non-log true forms?}
%$$ \big[\big(\nabla_n-\varepsilon(n) \nabla_n\,  \iota(n) +k H^g) A_0\big]\big|_\Sigma=0=\big[\iota(n)A_0\big]\big|_\Sigma\, .$$
Find $\A_i\in \Gamma\ct^k M[w_0-i]$ such that
\begin{eqnarray*}
\A^{(\ell)}&:=& \big(\A_0 + \sigma \A_1 + \sigma^2 \A_2 + \cdots\big) + \sigma^{h_0-1}(\log\sigma-\log\tau)\big(\, \overline{\A}_0+\sigma \, \overline{\A}_1+\sigma^2 \, \overline{\A}_2+\cdots \big)
\\[2mm]&+& O\big(\sigma^{\ell+1}\big)\, +\, O\big(\sigma^{\ell+1}\log(\sigma/\tau)\big)\, ,
\end{eqnarray*}
where along~$\Sigma$
$$
\A_0=q_W^{\scriptscriptstyle \Sigma}(A_{_\Sigma})\, ,
%\A_0:=q_W(A_0)+O(\sigma)\, ,
$$
and $\A^{(\ell)}$ solves the tractor Proca equations $$y \A=O(\sigma^\ell)+O\big(\sigma^\ell\log(\sigma/\tau)\big)\, ,\qquad\Is \A = \wDs \A = \Xs \A = 0\, ,$$ off $\Sigma$, for $\ell\in {\mathbb N}\cup \infty$ as high as possible.
\end{problem}

The main ingredients for solving this problem are again the solution generating operator $\cO$ including log terms of above and 
the operator $\widehat \Pi_\tau$ of Equation~
\nn{pihattau},
the analog of the holographic boundary projector at these weights. Combining these gives our result.

\begin{theorem}\label{tpsol}
For weights $w_0=-k\in\{0,-1,\ldots, -\lfloor \frac {n-1}2\rfloor\}$, Problem~\ref{trueproblem} has an $\ell=\infty$ solution given by
$$
\A=-\scalebox{.9}{$\frac{1}{n-2k}$} \  \Ds_{[3]}\, \cO\, \I\X q_{(N)}^\tau A_0\, .
$$
\end{theorem}

\begin{proof}
Along $\Sigma$ we have 
$$
\A= \widehat \Pi_\tau A_0\, ,
$$
so by Proposition~\ref{projprophat}, $\A|_\Sigma= q_W^\Sigma(A_{_\Sigma})$.

Given the formula for $\A$, it only remains to check that $y\A =0$.
Using 
Proposition~\ref{moving} twice
we have
\begin{eqnarray*}y\,   \Ds_{[3]}  \cO^{(\ell)} \I\X q_{(N)}^\tau A_0&=&
 \Ds_{[3]}\,y\,  \cO^{(\ell)}\, \I\X q_{(N)}^\tau A_0\\
&=&\Ds_{[3]} \Big(O(\sigma^\ell)+O\big(\sigma^\ell\log(\sigma/\tau)\big)\Big)\\
& =& O(\sigma^\ell)+O\big(\sigma^\ell\log(\sigma/\tau)\big)\, .
\end{eqnarray*}
This  shows that $\A$ solves the tractor Proca equations to any given order.
\end{proof}

\noindent
For the remaining true form weights the argument simplifies considerably as there are no log terms, we may use Propostition~\ref{projprophat}, so the solution is simply $\A =\colon K^{h_0}(z) \colon\,  \widehat \Pi \, A_0$. 
 
An important feature of the log solution is the appearance of the second solution generating operator multiplied by $y^{h_0-1}$, see Equation~\nn{O}.
In Section~\ref{OBSTQD} we shall show  that the operator $y^{h_0-1}$  yields a holographic formula for the BG operators of~\cite{BrGodeRham}. In particular, 
this means that $\big(y^{h_0-1}\A_0\big)\big|_\Sigma$ consists of a pair of boundary forms in the range of the codifferential $\cod_{_\Sigma}$ for any smooth western tractor $\A_0$.

\subsubsection{Weight dual true forms}

We now treat the case $w_0=k-n$, in which case  we take boundary data  given by  a pair of coclosed  weighted  forms  
 $(A_{_\Sigma},\phi_{_\Sigma})\in \ker(\cod_{_\Sigma},\cod_{_\Sigma})\subset\big(\Gamma\ce^k\Sigma[2k-n]\oplus \Gamma\ce^{k-1}\Sigma[2k-n-2]\big)$.
 We  focus on degrees $k\in\{0,1,\ldots, \lfloor \frac {n-1}2\rfloor\}$, {\it i.e.,} $h_0=-n+1,-n,\ldots,0$, to avoid the log solutions. 
 The latter can be obtained by applying, in the obvious way,  the weight/scale duality map to true form log solutions.
 On the other hand, the same map, applied to the solutions we derive in this Section, generates solutions of the second kind for the true form problem.
 The problem we solve is thus stated as follows.
 \begin{problem}\label{solprobdualtrue}
Given $(A_{_\Sigma},\phi_{_\Sigma})\in \ker(\cod_{_\Sigma},\cod_{_\Sigma})\subset\big(\Gamma\ce^k\Sigma[2k-n]\oplus \Gamma\ce^{k-1}\Sigma[2k-n-2]\big)$,  $k\in\{0,1,\ldots, \lfloor \frac {n-1}2\rfloor\}$,
consider an extension of these
$A_0\in \Gamma\ce^kM[2k-n]$, satisfying   the mixed Dirichlet--Neumann conditions
$$
A_0|_\Sigma = A_{_\Sigma}\, \qquad \mbox{ and }\qquad 
\big[\nabla_n \big(\iota(n) A_0\big)\big]\big|_\Sigma =-\phi_{_\Sigma}\, .
$$
Let
$$
\A_0:=q_W(A_0)\in \Gamma\ct^kM[2k-n]\, .$$
Find $\A_i\in \Gamma\ct^k M[k-n-i]$ such that
$$
\A^{(\ell)}:= \A_0 + \sigma \A_1 + \sigma^2 \A_2 + \cdots + O\big(\sigma^{\ell+1}\big)
$$
solves the tractor Proca equations $$y \A=O(\sigma^{\ell})\, ,\qquad \Is \A = \wDs \A = \Xs \A = 0\, ,$$ off $\Sigma$, for $\ell\in {\mathbb N}\cup \infty$ as high as possible.
\end{problem}

\begin{theorem}\label{tpdsol}
Problem~\ref{solprobdualtrue} has an $\ell=\infty$ solution given by
\begin{equation}\label{Kqsol}
\A = %-\scalebox{.95}{$\frac{1}{n-2k}$}\, \Ds_{[3]} \ \colon K^{h_0}(z) \colon\,  \D_{[3]}\,  q A_0\, , 
\colon K^{h_0}(z) \colon\,  q_W A_0\, .
\end{equation}
where $h_0=2k-n+1$.
\end{theorem}

\begin{proof}
Firstly we calculate the southern slot of $q_W A_0$ along~$\Sigma$ (see Lemma~\ref{West})
$$
-\cod A =\! -\cod \big(\iota(n) \varepsilon(n) + \varepsilon(n) \iota(n) + 2\rho\sigma\big) A_0
\stackrel{\Sigma}= -\cod_{_\Sigma} A_{_\Sigma} + \big(\cod\varepsilon(n)\sigma \phi \big)\!\big|_\Sigma
=\big(\iota(n)\varepsilon(n) \phi \big)\!\big|_\Sigma=\phi_{_\Sigma} .
$$
This establishes that $\big(q_W A_0\big)\big|_\Sigma = q^{\Sigma}_W (A_{_\Sigma},\phi_{_\Sigma})$. Thus we may 
employ Proposition~\ref{superprop} from which we have $\A = \colon K^{h_0}(z) \colon \, \Ds_{[3]} \I \B = \Ds_{[3]} \, \colon K^{h_0}(z) \colon  \I \B$
because $\colon K^{h_0}(z) \colon$ annihilates terms $O(\sigma)$. (Lemma~\ref{moving}
was used for the last equality.) Thus, this expression 
 is  annihilated by $y,\Is,\Xs,\Ds$. Also by virtue of Proposition~\ref{superprop},  along~$\Sigma$ we have $\Ds_{[3]} \I \B=q_W^{\scriptscriptstyle \Sigma}(\A_{_\Sigma},\phi_{_\Sigma})$. This agrees with $q_W A_0$ thanks to Proposition~\ref{bcs2}.
\end{proof}

One might wonder how general the simple form of the solution~\nn{Kqsol} for $w_0=-k-n$ is.
By construction, $q_W A_0\in \ker(\wDs,\Xs)$ (when $w_0\neq k-d$) so let us consider Equation~\nn{Kqsol} at generic weights. 
In fact,  for $A_0$ an extension of $A_{_\Sigma}$, we have $q_W A_0+\sigma {\mathcal C} \in \ker \Is$ for some smooth tractor ${\mathcal C}$. Thus
$\colon K^{h_0}(z) \colon\,  q_W A_0=\colon K^{h_0}(z) \colon\, \Pi\, q_W A_0$. So our generic solution~\nn{generic} amounts to $\A=\colon K^{h_0}(z) \colon\,  q_W A_0$.
Hence, the solution to Problem~\ref{laylanguage} is obtained by extracting the western slot of this expression
$$
A=q^* \, \colon K^{h_0}(z) \colon\,  q_W A_0\in \Gamma\ce^kM[w_0+k]\, ,\quad w_0\neq -k,k-d,\quad  h_0\notin{\mathbb Z}_{\geq 1} \, .
$$ 

\subsection{The product solution}\label{prodf}

In this Section, we show generically that  there is a simple and explicit formula for solutions
of the Proca system which can be stated without recourse to tractor formalism. 
Nevertheless, the tractor machinery plays
the central {\it r\^ole}  in obtaining this.  These solutions   use the product formula for the solution generating
operator developed in Section~\ref{products} specialised to  weighted tractor forms.

The main technical tool to translate between western tractors and weighted differential forms is as follows.
Recall from the west
Lemma~\ref{West} that, for tractor weights $w\neq k-d$, a solution~$\A\in\Gamma\ct^kM[w]$ to $\wDs \A = \Xs \A =
0$ is isomorphic to a weighted differential form
$A\in\Gamma\ce^kM[w+k]$ via
$$
\A = q_W(A)\, .
$$

We now investigate how key operators between western tractors are intertwined by~$q_W$.
Let us assume tractor weights $w_0\neq -k,k-n$ so that the scale-transversality equations $$\Is\A=0=\wDs\A = \Xs \A$$ are solved via
$$
\A=\Pi\, \A_0\, ,\qquad \A_0\in \Gamma\ct^kM[w_0]\, .
$$
Now we consider 
$
\A_0=q_W(A)
$
so that
$$
\A=\Pi\,  q_W(A)\, ,
$$
for some $A\in\Gamma \ce^kM[w+k]$.
Putting together the result of Proposition~\ref{holpro1} and the computation given in Remark~\ref{explicitpi}
gives the following result.
\begin{proposition}\label{qWPi}
Suppose $w\neq -k,k-d,k-n$ and $A\in \Gamma\ce^kM[w+k]$, then
$$
\Pi\,  q_W(A)=\scalebox{.9}{$\frac1{(w+k)(n+w-k)}$}\,q_W(\wiota\, \wepsilon A)\, .
$$
\end{proposition}
\noindent
So the map 
$$
A\mapsto \scalebox{.9}{$\frac1{(w+k)(n+w-k)}$}\,\wiota\, \wepsilon\,  A
$$
solves the scale-transversality equations via the ``intertwiner''~$q_W$.

It remains to solve the Laplace--Robin equation $y\A=0$. For that we use the following result capturing~$y$ in terms of an operator on weighted forms,
which also demonstrates the equivalence of the tractor Proca equations~\nn{second} with the  Proca system in Equations~\nn{Procao} and~\nn{usemelater}.
\begin{lemma}\label{aerm}
Let $A\in \Gamma\ce^k [w+k]$ with $w\neq k-d$ and $\A=q_W (\wiota\, \wepsilon A)$. Then
$$
xy \A = q_W \Big[\big(\wiota \, \wepsilon  -(w+k)(n+w-k) \big)\wiota \, \wepsilon A \Big]\, .
$$
\end{lemma}

\begin{proof}
Since $\Is \A=0$, it follows that $y\A = -\Is \D \A$, so it is a simple matter of writing out $\Is$ and $\D$ explicitly for some $g\in c$
to verify the above. An alternate proof is to apply Lemma~\ref{failure} to $ \A=q_W(A)$ and then use Proposition~\ref{qWPi} above.
\end{proof}

Combining this result with the product solution of Proposition~\ref{firstsolprop}, we immediately have a product solution to the tractor Proca equations 
on weighted forms.
\begin{proposition} \label{fsol}
For 
$w_0\neq -k, k-d,k-n$ Problem~\ref{solprob} has a solution 
to order 
$$
\ell=\left\{\begin{array}{cl}\infty\, ,&h_0\neq 2,3,\ldots\\[1mm]h_0-2\, ,&h_0=2,3,\ldots\, ,\end{array}\right.
$$
given by
$$
\A^{(\ell)}=q_W(A^{(\ell)})\, ,
$$
with
\begin{equation}\label{Al}A^{(\ell)}=\frac{\wiota\, \wepsilon}{(w_0+k)(n+w_0-k)}\, \Big[\prod_{j=1}^\ell \frac{\wiota\, \wepsilon -(w_0+k-j)(n+w_0-k-j)}{j(n+2w_0-j)}\Big]\, A_0\, .\end{equation}
\end{proposition}

Thus we have the following Theorem.
\begin{theorem}
The solution to the higher form Proca system of Problem~\ref{laylanguage}
with the same weight and order conditions as in Proposition~\ref{fsol}, is
given by~$A^{(\ell)}$ as in~\nn{Al}.
\end{theorem}

\begin{proof}
We give an alternative proof of this Theorem that does not rely on tractors, and in particular 
 without employing the intertwined version of the central relationship $[x,y]=h$ of Proposition~\ref{thesl2}.
 Instead, the key ingredient in the commutativity of $\sigma$ with the interior and exterior normals.
This  is as follows:
Firstly we have an intertwined formula for the operators $c_k$ of Equation~\ref{ci}
$$
c_j \, q_W\big(\wiota\, \wepsilon A_0\big) = q_W\big(\hat c_j\, \wiota\, \wepsilon A_0\big)\, ,
$$
where acting on $\Gamma \ce^k M[w]$ we have
$$
\hat c_j := \wiota\, \wepsilon - (w-j)(n+w-j-2k)\, .
$$
Acting on $\ker \wiota$, the operator $\hat c_0$ is the intertwined version of $xy$ and indeed, from Equation~\nn{alongsigma}, $\hat c_0 \wiota\, \wepsilon A_0=0$ along $\Sigma$.
Hence, on the kernel of $\wiota$, the operator
$$
\hat y := \sigma^{-1} \hat c_0
$$
extends smoothly to~$\Sigma$. 

Now, since $\sigma$ commutes with both $\wiota$ and $\wepsilon$ (see Corollary~\ref{walgebra}) and $\sigma$ has weight~$1$, we immediately
have the operator identity
$$
\sigma^\alpha \hat c_j = \hat c_{j+\alpha} \sigma^\alpha\, .
$$
This is in fact valid everywhere in the interior of $M$ for any $\alpha\in {\mathbb C}$. 

Since $\wiota\,{}^2 =0$, the second part of Equation~\nn{Procao} holds trivially. 
To complete the proof, we need to establish Equation~\nn{yleft}, namely $\hat y \, \hat c_1 \hat c_2\ldots \hat c_\ell =  \sigma^\ell \hat y^{\ell+1}$; here and henceforth we work on the 
kernel of $\wiota$. Using the above identity and calculating in the interior of~$M$ we have
\begin{eqnarray*}
\hat y\,  \hat c_1 \hat c_2 \ldots \hat c_\ell
&=&\sigma^{-1} \hat c_0\hat c_1 \hat c_2 \ldots \hat c_\ell\\
&=& \sigma^{-1} \hat c_1 \hat c_2 \ldots \hat c_\ell \hat c_0\\
&=&\hat c_0 \hat c_1 \ldots \hat c_{\ell-1} \hat y \\
&=&\sigma \hat y\,  \hat c_1 \ldots \hat c_{\ell-1} \hat y \\
&\vdots& \\
&=&\sigma^\ell  \hat y^{\ell+1}\, .
\end{eqnarray*}
Acting with $\hat y$ maps sections in $\ker \wiota$ to sections in $\ker \wiota$, so this 
identity extends to~$\Sigma$. 
\end{proof}

Let us end this Section with a rewriting of the product solution~\ref{Al} which is useful because it makes direct contact with the Laplace operator acting on true forms. Firstly we need  a pair of technical Lemmas.
\begin{lemma}\label{poly}
Let $P(z)$ be any polynomial and call $$\zeta:=\wiota\, \wepsilon\, \qquad\mbox{ and }
\qquad {\mathcal L}:= \{\wiota , \wepsilon\}\, .$$
Then we have the operator identity
$$
\zeta\,  P(\zeta) = \wiota\, P({\mathcal L})\, \wepsilon\, .  
$$
\end{lemma}
\begin{proof}
It suffices to prove the statement for a monomial $\zeta\,  \zeta^\ell$ for some integer $\ell$. Recalling from Proposition~\ref{walgebra} that $\wiota^{\, 2}=0$  we have, in parallel to the usual exterior calculus of $\cod$ and $\extd$, 
$$
\wiota \, {\mathcal L}^\ell\, \wepsilon =
\wiota \, \big(\wiota \, \wepsilon+\wepsilon\, \wiota\big) {\mathcal L}^{\ell-1}\, \wepsilon
=
\cdots 
 =\wiota \big(\wepsilon\,  \wiota\big)^\ell\wepsilon=\zeta\, \zeta^\ell\, .
$$ 
\end{proof}

\begin{lemma}\label{prodprod2}
Let $A\in \Gamma\ce^k M[w]$. Then, away from~$\Sigma$,
we have the following results. First 
$$
\wiota\,  A = -\sigma^{d+w-2k+1} \, \cod \, \sigma^{2k-w-d} A\qquad
\mbox{ and } \qquad
\wepsilon\,  A = - \sigma^{w+1}\, \extd \, \sigma^{-w}  A\, .
$$
Moreover
$$
{\mathcal L}\,  A := \{\wiota,\wepsilon\}\, A = \sigma^w\,  \widehat{\mathcal L} \ \sigma^{-w} A\, ,
$$
so that for a polynomial $P(z)$,
$$
P({\mathcal L}) \, A = \sigma^w\, P\big(\widehat{\mathcal L}\, \big)\, \sigma^{-w} \, A\, ,
$$
where the operator $\widehat{\mathcal L}:\Omega^kM\to\Omega^kM$
is given by
\begin{equation}\label{Lhat}
\widehat{\mathcal L}:= \sigma^2 \FL + 
(2k-d) \, \big[\sigma\pounds_n + \varepsilon(n)\iota(n)\big]
+2\sigma\,  \big[\varepsilon(n) \cod + \iota(n) \extd \big]\, .
\end{equation}
\end{lemma}
\begin{proof}
The first two identities follow immediately from first two relations in Equation~\nn{comms}.
Equation~\nn{Lhat} follows by a slightly more intricate  application of these as well as Cartan's magic formula
for the Lie derivative
$$
\{\extd,\iota(n)\}=\pounds_n\, .
$$
\end{proof}
\begin{remark}
In the formul\ae\ above, the right hand sides are clearly not defined along~$\Sigma$, however,
as is clear from Definition~\ref{wdef}, the left hand sides  are.
When using expressions, such as those on the right hand sides, to express objects defined everywhere on $M$
we will use the notation $\stackrel{\scriptstyle \rm lim}=$.
\end{remark}
Applying these two Lemmas to the solution displayed in Equation~\nn{Al} gives the following result.
\begin{proposition}\label{prodprod}
The solution to the higher form Proca system of Problem~\ref{laylanguage}
with the same weight and order conditions as in Proposition~\ref{fsol}, is
given by
\begin{equation}\label{Alprod}
A^{(\ell)}\, \stackrel{\scriptstyle \rm lim}=\, \sigma^{d+w_0-k}\, \cod\, \sigma^{2k-d+2}\, P^{(\ell)}\big(\widehat{\mathcal L}\, \big)\, 
 \extd\,  \sigma^{-w_0-k}\, A_0\, ,\end{equation}
where the polynomial $P^{(\ell)}$ is given by
$$
P^{(\ell)}(z)=\frac{1}{(w_0+k)(n+w_0-k)}\, \Big[\prod_{j=1}^\ell \frac{z -(w_0+k-j)(n+w_0-k-j)}{j(n+2w_0-j)}\Big]
$$
and 
$
\widehat{\mathcal L}
$ 
is as displayed in Equation~\nn{Lhat}.
\end{proposition}

\subsection{Solutions in Graham--Lee normal form}

To  make contact to a coordinate-based approach, let us explicate the solution generating operator  in its product form
in a choice of scale adapted to the Graham--Lee normal form for the interior metric~$g^o$.  The aim is to present explicitly the operator
of Lemma~\ref{next} that generates  $O(\sigma^{\ell+1})$ solutions from $O(\sigma^{\ell})$ solutions.

Let $\tau\in\ce M[1]$ be a scale that extends to the boundary which defines a metric $\overline g=\tau^{-2}{\bm g}\in c$.
Moreover, we take $r=\sigma/\tau$ to be the function that gives $\sigma$ in the scale~$\tau$. In that case $g^o=\frac{\overline g}{r^2}$ where $r$ is 
the defining function of~\cite{GrL,GrVol}. Working in terms of local coordinates in a collar neighbourhood $[0,\epsilon]\times\Sigma$ of the boundary with coordinate $r\in[0,\epsilon]$,
 $\overline g$ extends to~$\Sigma$ with form
\begin{equation}\label{GL}
\overline g = dr^2 + h\, .
\end{equation}
Here $h$ is a family of metrics on $\Sigma$ parameterized by~$r$.
In this choice of scale $\varepsilon(n) = {\extd r}\wedge$ so differential forms $A\in \Gamma\ce^kM[w]$ can be uniquely decomposed as
$$
A=A^\perp + \extd r \wedge A^\parallel\, , \qquad (A^\perp,A^\parallel)\in \ker \iota(n)\, .
$$
We will write this choice of splitting using a column vector notation
$$
A\stackrel{\overline g}= \begin{pmatrix}A^\perp\\ A^\parallel\end{pmatrix}\, .
$$
Tautologically then, in this  splitting we have
$$
\iota(n)\stackrel{\overline g}=\begin{pmatrix}0&1\\0&0\end{pmatrix}\, ,\qquad\varepsilon(n)\stackrel{\overline g}= \begin{pmatrix}0&0\\1&0\end{pmatrix}\, .
$$
Defining
$$
\extd^\perp:=\extd - \varepsilon(n)\frac{\partial}{\partial r}\, ,\qquad \cod^\perp := \cod -\iota(n)\frac{\partial}{\partial r}\, ,
$$
the exterior derivative and codifferentials become
\begin{equation}\label{GLddelta}
\extd\stackrel{\overline g}=\begin{pmatrix}\extd^\perp&0\\[1mm]\frac{\partial}{\partial r}&-\extd^\perp\end{pmatrix}\, ,\qquad\cod\stackrel{\overline g}=\begin{pmatrix}\cod^\perp&\frac{\partial}{\partial r}+\frac 1 2(\trH-2{\mathbb{H}})\\[1mm]0&-\cod^\perp\end{pmatrix}\, ,
\end{equation}
and the form Laplacian is the anticommutator of these.
Along $\Sigma$, for $A\in \ker \iota(n)$,  $$\big(\extd^\perp\! A\big)\big|_\Sigma = \extd_{_\Sigma} A_{_\Sigma}\, , \qquad\big(\cod^\perp \!A\big)\big|_\Sigma = \cod_{_\Sigma} A_{_\Sigma}\, .$$ 
The operator $\mathbb{H}$ is the natural endomorphism field on the subbundle of $\ker \iota(n)\in \Gamma T^*M$ coming from  $h^{-1} \frac{\partial h}{\partial r}$ in the local coordinates
and then extended in the usual way to an endomorphism on forms in $\ker \iota(n)\in \Omega^\bullet M$. We have denoted $\trH:={\rm tr}\,  h^{-1}\frac{\partial h}{\partial r}$.
Note the relationship between bulk and boundary differentials and codifferentials
$$
\big(\iota(n) \varepsilon(n) \, \extd A\big)\big|_\Sigma= \extd_{_\Sigma} A_{_\Sigma}\, ,\qquad
\big(\cod\,  \iota(n) \varepsilon(n) A\big)\big|_\Sigma= \cod_{_\Sigma} A_{_\Sigma}\, ,
$$
for any extension $A$  of $A_{_\Sigma}\in \Gamma\ce^k\Sigma[w]$, is  manifest in this splitting since
$$
\iota(n) \varepsilon(n) \, \extd \stackrel{\overline g}=\begin{pmatrix}\extd^\perp&0\\ 0&0\end{pmatrix}\, ,\qquad
\cod\,  \iota(n) \varepsilon(n) \stackrel{\overline g}=\begin{pmatrix}\cod^\perp&0\\ 0&0\end{pmatrix}\, .
$$
Moreover, the Lie derivative along $n$ is simply
$$
\pounds_n=\{\extd,\iota(n)\}=\begin{pmatrix}\frac{\partial}{\partial r} & 0 \\ 0 & \frac{\partial}{\partial r}\end{pmatrix}
\, .
$$

Inserting equations~\nn{GLddelta},~\nn{Lhat}  and the choice of scale $\sigma =r$ in the product  solution~\nn{Alprod} of Proposition~\ref{prodprod} gives an explicit formula for solutions in the Graham--Lee normal form for
the interior metric. Moreover, we can combine these results with Lemma~\ref{next} to obtain the operator that increases the order of a solution.
\begin{theorem}\label{GLtheorem}
Let $A^{(\ell)}\in\ker\wiota$ solve Problem~\ref{laylanguage} to order $\ell$ with $w_0\neq -k,k-d,k-n$ and $\ell\neq d+2w_0-2$. Then
an order $\ell+1$ solution is given, in the scale $\tau$ corresponding to the normal form~\nn{GL}, by
\begin{equation*}
%\label{LR}
A^{(\ell+1)}\stackrel{\rm lim}=
\frac{L\, R-(w_0+k-\ell-1)(n+w_0-k-\ell-1){\mathbb 1}}{(\ell+1)(d+2w_0-2-\ell)}
\: A^{(\ell)}\, ,
\end{equation*}
where $e:=r\frac{\partial}{\partial r}$,
$$
L:=\begin{pmatrix}
r\cod^\perp 
&
e-n-w_0+k+\frac r2(\trH-2{\mathbb{H}})
\\[2mm]
0&-r\cod^\perp \end{pmatrix}
\, ,\quad\mbox{and}\quad
R:=
\begin{pmatrix}r\extd^\perp&0\\[2mm]e-w_0-k&-r\extd^\perp\end{pmatrix}\, .
$$
\end{theorem}
\begin{proof}
First from $A^{(\ell)}$ we construct the tractor $\A^{(\ell)}=q_W(A^{(\ell)})$. This  obeys $y \A^{(\ell)}=O(\sigma^{\ell})$.
Hence, by Lemma~\ref{next},
$$
\A^{(\ell+1)}=\scalebox{.9}{$\frac{1}{(\ell+1)(d+2w_0-2-\ell)}$}\, \big[xy + (\ell+1)(d+2w_0-2-\ell)\big]\,  \A^{(\ell)}
$$
solves $y\A^{(\ell+1)}=O(\sigma^{(\ell+1)})$. 
Next we apply $q^*$ to the above expression after
re-expressing the operator  $xy$ in terms of $\wiota\, \wepsilon$ via  Lemma~\ref{aerm} applied to $A^{(\ell)}$ in place of $\wiota\,  \wepsilon A$, which is legal because $A^{(\ell)}\in\ker \wiota$. Then $q^* \A^{(\ell+1)}\propto \big(\wiota\, \wepsilon -(w_0+k-\ell-1)(n+w_0-k-\ell-1) \big)\A^{(\ell)}$.
This expression obeys the first equation of~\nn{Procao} by construction. Moreover, by virtue of the identity $\wiota^{\, 2}=0$,  we have $q^* \A^{(\ell+1)}\in \ker \wiota$, so the second of those equations also holds.
Thereafter we employ Lemma~\ref{prodprod2} to write the solution $q^* \A^{(\ell+1)}$ compactly. The last step is to use Equation~\nn{GLddelta}
 for the exterior derivative and codifferential in the scale $\sigma=r$.
\end{proof}

\begin{remark}
The condition that the order~$\ell$ solution solves the transversality condition $\wiota\,  A^{(\ell)}=0$ to all orders
is not an essential restriction since this can always be achieved using the technology of Section~\ref{coulomb}.
\end{remark}

\section{Obstructions, detours, gauge operators and $Q$-curvature}\label{OBSTQD}

Continuing in the Poincar\'e--Einstein setting,
we now consider obstructions to
smoothness of the  solutions of Section~\ref{Procasect}. This is  partly captured by the following result.
\begin{theorem}\label{fruit}
Let $(M,c,\sigma)$ be Poincar\'e--Einstein. Then 
for any weight $w_0\neq k-d$ 
the tractor Proca equations~\nn{second} have 
a smooth solution
\begin{equation}\label{geewhizz}
\A^{(\ell)} = 
\colon K^{(\ell)}(z)\colon\, 
q_W(A_0)
+O(\sigma^{\ell+1})\, ,
\end{equation}
to order 
$$
\ell=\left\{\begin{array}{cl}\infty\, ,&h_0\neq 3,5,7,\ldots\\[1mm]h_0-2\, ,&h_0=3,5,7,\ldots\, .\end{array}\right.
$$
Here $A_0\in \Gamma\ce^kM[w_0+k]$ is an extension of $A_{_\Sigma}\in \Gamma\ce^k\Sigma[w_0+k]$, and when $h_0\neq 2,4,6,\ldots $ the solution generating operator $\colon K^{(\ell)}(z)\colon$  is determined by  $\colon K^{h_0}(z)\colon$ as in~\nn{besselfn}. For the case $h_0=2,4,6,\ldots$, $\colon K^{(\ell)}(z)\colon$ is determined by the solution generating operator in~\nn{O}, omitting the log terms on the second line of that display.
\end{theorem}

\begin{remark}
Going beyond the space of smooth asymptotic solutions, for weights $h_0=3,5,7,\ldots$, the remainder in the expression~\nn{geewhizz}, as computed in Section~\ref{logsect}, is 
$O(\sigma^{\ell+1}$ $\log\sigma)$.
\end{remark}

\begin{proof}
For weights $w_0=k-n$, this is just a restatement of Theorem~\ref{tpdsol}. For the other cases we use that the holographic triple D-operators $\Ds_{[3]}$ and $\D_{[3]}$ commute with 
the solution generating operator $\colon K^{(\ell)}(z)\colon$. Moreover, when $w_0\neq -k,k-n$ one has $\Ds_{[3]}\, \D_{[3]}\A_0=(w_0+k)(n+w_0-k)\, \Pi \, \A_0$. The latter can be written as  $q_W(A_0)$  for some~$A_0$ (see Equation~\nn{PIA=qA}). For $w_0=-k$ we have $-\frac{1}{n-2k}\, \Ds_{[3]}
\I\X q^\tau_{(N)} A_0=q_W(A_0)+O(\sigma)$ which can be verified using the explicit expressions for $\Ds_{[3]}$ (see Equation~\nn{tripleD}), $q^\tau_{(N)}$, $\I$ and $\X$ (see Equation~\nn{qNtau} and Sections~\ref{algebra},~\ref{compinsert}).
Thus the remaining results follow from Theorems~\ref{BIGTHEOREM},~\ref{mainl} and~\ref{tpsol}.
\end{proof}

\begin{remark}
The restriction $w_0\neq k-d$ in the above Theorem is an inessential one. For that weight, the map $q_W$ is not defined in the bulk. Instead the boundary data $A_{_\Sigma}$ is mapped to a boundary tractor $q_{W}^{\Sigma} (A_{_\Sigma})$ which can be subsequently extended to a bulk tractor $\A_0$.
Thus $\A^{(\ell)} = 
\colon K^{(\ell)}(z)\colon\, 
{\rm ext}\circ  q^\Sigma_W(A_{_\Sigma})
+O(\sigma^{\ell+1})$ and the remainder of this case and its proof follows {\it mutatis mutandis}.
\end{remark}

Considering the cases $h_0=3,5,7,\ldots$, we could attempt to extend the above smooth solution to higher orders:
Using that $\Ds_{[3]}$ commutes with the operators $x$ and $y$, it would suffice to solve the $y\A=0$ problem to solve
all four tractor Proca equations. However, from~\cite{GWasym}, starting with any tractor boundary data $\A_0$ and
attempting to solve the corresponding $y\A=0$ extension problem at these weights, one encounters the obstruction 
$\big(y^{h_0-1}\A_0)\big|_\Sigma$; one succeeds in obtaining a formal smooth solution if and only if this vanishes. 
In fact, again from~\cite{GWasym},   
this is also the coefficient of the first log term $\sigma^{h_0-1}\log \sigma$ in solution generating operator~\nn{O}.
Specialising this result, we immediately have the following.
\begin{proposition}
For any weight $w_0\neq k-d$ such that $h_0=3,5,7,\ldots$, consider
the tractor Proca equations~\nn{second} with boundary data captured by $A_0\in \Gamma\ct^k M[w_0+k]$ satisfying $\big(\iota(n)A_0\big)\big|_\Sigma=0$.
The tractor field
\begin{equation}\label{obst}
y^{h_0-1} q_W(A_0)\, ,
\end{equation}
is the coefficient of the first log term $\sigma^{h_0-1}\log \sigma$ in the solution generated by ${\mathcal O}$ in~\nn{O}.
The vanishing of this obstruction along~$\Sigma$ is  necessary and sufficient to obtain a smooth formal solution.
\end{proposition}

We now study the {\it obstruction} to smoothness given in expression~\nn{obst}.

\subsection{Detour and gauge operators}

From essentially the same argument as in the proof of Theorem~\ref{fruit}, we see that the obstruction given in expression~\nn{obst} is a west tractor
along~$\Sigma$ so consists of two pieces of weighted differential form data; the part residing in the west slot is conformally invariant. It can be extracted along~$\Sigma$ using the operator~$q^*$
of Section~\ref{compinsert}.
% to yield $q^* y^{h_0-1} q_W(A_0)$. 
To rewrite $q_W(A_0)$ in a more convenient format, we need the following result.
\begin{lemma}\label{beach}
Let $A_0\in\Gamma\ce^kM[w_0+k]$ be an extension of $A_{_\Sigma}\in\Gamma\ce^k\Sigma[w_0+k]$. Then, if $w_0\neq k-d,k-n,-1-\frac n2$, along $\Sigma$
$$
q_W(A_0) =\scalebox{.95}{$\frac{1}{(n+2w_0+2)(n+w_0-k)}$}\,{\oDs}\! \X q(A_0)\, .
$$ 
\end{lemma}
\begin{proof}
Since $A_0$ extends $A_{_\Sigma}$ to $M$, it follows along~$\Sigma$ that $\iota(n)A_0=0$ and (by continuity) $\iota(n)\cod A_0=0$. Thus, along~$\Sigma$, we have
$\Is q_W(A_0)=0$. Now, focussing on the case $w_0\neq -k$, we can use Propositions~\ref{projprop} and~\ref{extExt} to write (along~$\Sigma$) 
$$q_W(A_0)=\Pi \, q_W(A_0)=\Pi_W^{\scriptscriptstyle \Sigma}  \big(q_W( A_0)\big)\big|_{\Sigma}
=\Pi_W^{\scriptscriptstyle \Sigma}\,  q_W^{\scriptscriptstyle \Sigma} A_{_\Sigma}= q_W^{\scriptscriptstyle \Sigma} A_{_\Sigma}\, .$$
The third equality uses that $\Pi_W^{\scriptscriptstyle \Sigma}$ only sees the western slot of $\big(q_W( A_0)\big)\big|_{\Sigma}$.
Noting the identity (again along~$\Sigma$, using also Proposition~\ref{diffsplit})
$$
q_W^{\scriptscriptstyle \Sigma} A_{_\Sigma}=\scalebox{.95}{$\frac{1}{(n+2w_0+2)(n+w_0-k)}$}\, \Ds_{_\Sigma}\X_{_\Sigma} q_W^{\scriptscriptstyle \Sigma} A_{_\Sigma}
=\scalebox{.95}{$\frac{1}{(n+2w_0+2)(n+w_0-k)}$}\,\oDs \X q(A_0)\, ,
$$
the result follows. For the case $w_0=-k$, the proof is almost identical except that one replaces $\Pi q_W(A_0)$ with $\widehat{\Pi}_\tau(A_0)$
and then relies upon Proposition~\ref{projprophat}.
\end{proof}

%Following~\cite{BrGodeRham} (or using the explicit formul\ae \, for $\Ds$ and $\X$ in Section~\ref{algebra}) one sees, away from weights $w_0=k-d$, that $q_W(A_0)$ is  proportional to~${\oDs}\! \X q(A_0)$. 
Writing $\ell=\frac{h_0-1}{2}$, this suggests the following definition.

\begin{definition}\label{LONG}
We define 
%$L_k:\Omega^k M\longrightarrow \Omega_k M\, ,$ 
${\mathbb L}_k^{\ell}:\Gamma\ce^kM[k+\ell-\frac n2]\longrightarrow 
%\Gamma\ce^kM[k-\ell-\frac n2]
\Gamma \ct^kM[-\ell-\frac n2]
$, $\ell=1,2,3,\ldots$
by the composition of tangential operators
$$
{\mathbb L}_k^{\ell}:= 
%q^* 
y^{2\ell} {\oDs} \X q \, .
$$
Along $\Sigma$, we define 
$$
L^\ell_k:=q^* {\mathbb L}^\ell_k\, .
$$
\end{definition}

As follows from the discussion above and as will be established in Theorem~\ref{fun},  the range of the map ${\mathbb L}^\ell_k$ lies in the range of $\Xs$ along~$\Sigma$, hence the composition~$L^\ell_k$
is well defined along~$\Sigma$.

\begin{proposition}
The operator $L^{\ell}_k$
%is tangential so 
defines a conformally invariant differential operator 
$$L^{\ell}_k:\Gamma\ce^k\Sigma[k+\ell-\frac n2]\longrightarrow \Gamma\ce^k\Sigma[k-\ell-\frac n2]\, ,$$
and a necessary condition for extending $A_{_\Sigma}\in\Gamma\ce^kM[k+\ell-\frac n2]$
to a smooth solution to the tractor Proca equations is
$$
L^{\ell}_k A_{_\Sigma}=0\, .
$$
\end{proposition}

\begin{proof}
Tangentiality of $L^{\ell}_k$ has already been established and  boundary forms canonically  embed in bulk forms restricted to the boundary. The remainder of the Theorem is established by combining Lemma~\ref{beach} and the expression Equation~\nn{obst} for the obstruction of~\cite{GWasym}.
\end{proof}

The operator ${\mathbb L}^{\ell}_k$ is well-defined in the bulk, and in fact holographically extends the higher order conformally invariant operators
on forms from~\cite{BrGodeRham}.

\begin{proposition}
As an operator on  boundary forms, the operator $L_k^\ell$ is the same  as that defined in~\cite[{\it Theorem 2.1}]{BrGodeRham}.
\end{proposition}
\begin{proof} 
First note that $y^{h_0-1}$ is the tractor form twisted GJMS operator as proved in~\cite{GWasym}.
On the other hand, 
${\oDs}$ is a holographic formula for $\Ds_{_\Sigma}$ and so along~$\Sigma$, Definition~\ref{LONG} coincides with the construction of~$L_k^\ell$
given in~\cite{BrGodeRham}.
\end{proof}

Further to our observations, the importance of the 
$L^{\ell}_k$ is that generically they completely control the obstructions to smoothly solving the Proca problem.

\begin{theorem}\label{66}
The Dirichlet Proca Problem~\ref{laylanguage} with boundary data  $A_{_\Sigma} \in \Gamma\ce^{k}\Sigma[w_0+k]$,  for $w_0$ such that $h_0=3,5,7,\ldots$ and
$w_0\neq -k, k-n$,
admits 
a  formal smooth solution to all orders iff 
$$
L^{\ell}_k A_{_\Sigma}=0\, .
$$
That is, for $w_0\neq -k, k-n$,   the space $\ker L^{\ell}_k$ uniquely parameterises the smooth solution space of the Proca Problem~\ref{laylanguage} modulo the addition of 
second solutions. 
%at weights $w_0$ such that $h_0=1,2,3,\ldots$.
\end{theorem}

\begin{proof}
Returning to display~\nn{obst} we note that along~$\Sigma$, $q_W(A_0) = \Ds_{[3]} \B$ for some $\B$ and so since $y$ commutes with the holographic interior
triple D-operator,  $y^{h_0-1}q_W(A_0)$ is a weight $-n-w_0$ boundary west tractor. Thus, provided  $-n-w_0\neq k-n$ ({\it i.e.}, $w_0\neq -k$), by the West Lemma~\ref{West}
it is completely determined by its western slot. 

Uniqueness of the parameterisation of the solution space follows by the iterative construction of the solution given in~\cite{GWasym}.
\end{proof}

To understand fully  smoothness, it remains now to study true forms.

\subsection{True forms} \label{truef}

Here we  focus on even boundary dimensions $n$ and  weights $w_0=-k\in\{0,-1,\ldots ,1-\frac n2 \}$ such that $h_0=3,5,7,\ldots$.
We return to the full obstruction displayed in~\nn{obst}, which using Lemma~\ref{beach}
can be expressed as the restriction to~$\Omega^k \Sigma$ of the tangential {\it obstruction operator} 
\begin{equation}\label{fullobst}
y^{h_0-1} {\oDs} \X q= {\mathbb L}_k^{\frac n2 -k}  =:{\mathbb L}_k \, .
\end{equation}
Along~$\Sigma$, the range of this operator consists of two parts; its western slot  
\begin{equation}\label{holL}q^*y^{h_0-1} {\oDs} \X q=L^{\frac{n}2-k}_k=:L_k\, ,\end{equation}
dealt with above and its southern slot
\begin{equation}\label{holG}
q^* \Ys y^{h_0-1} {\oDs} \X q=:G_k\, .
\end{equation}
(The operator $\Ys$ used here is defined in Equation~\nn{Y} below.)
We note that, in contrast to the other cases,  $G_k$ is here not determined by~$L_k$ since the obstruction operator displayed in~\nn{fullobst}
takes values in tractors of weight $k-n$ (see Lemma~\ref{West} applied to boundary tractors).

The operator $G_k$ depends on a second true scale $\tau\in \Gamma\ce M[1]$  (through $\Ys$) and for each such gives a canonical, tangential, 
linear differential operator, which upon restriction to~$\Sigma$ is a map
\begin{equation}\label{boundG}G_k:\Omega^k\Sigma\longrightarrow \Gamma\ce^{k-1}\Sigma[2k-n-2]\, .\end{equation}
As an operator on  boundary forms, the operator $G_k$ of  display~\nn{boundG} is the same  as that defined in~\cite[{Expression (3)}]{BrGodeRham}.
Agreement follows immediately from the proof of Theorem~2.8 of~\cite{BrGodeRham}.

By analogy with the proof of Theorem~\ref{66}, we obtain  the result characterising the obstruction to smooth solutions for true forms.

\begin{theorem}
In even boundary dimension~$n$, the
Dirichlet Proca Problem~\ref{laylanguage} with boundary data  $A_{_\Sigma} \in \Omega^{k}\Sigma$ with $k\in\{0,1,\ldots ,\frac n2-1 \}$
(so weights $w_0=-k$ thus $h_0\in\{3,5,7,\ldots\}$)
admits 
a  formal smooth solution to all orders iff 
$$
L_k A_{_\Sigma}=0=G_k A_{_\Sigma}\, .
$$

%at weights $w_0$ such that $h_0=1,2,3,\ldots$.
\end{theorem}

\begin{remark}\label{injectively}
The differential operator $G_k$ is a {\it gauge companion/fixing} operator for the detour operator $L_k$,  meaning that the pair $(L_k,G_k)$
is graded injectively elliptic and~${\mathcal H}^k_L$ is used to denote its kernel which is finite dimensional for compact boundary~$\Sigma$ (see~\cite{BrGodeRham}). 

Thus, we have an analog for true forms of the last part of Theorem~\ref{66}, in that ${\mathcal H}^k_L$ uniquely parametrises, modulo second solutions, the formal smooth solutions to the 
gauge fixed higher form Maxwell system given by the formal  Proca problem~\ref{laylanguage} with $w_0=-k$.

For the global version of our problem, the second solutions are determined by the Dirichlet problem modulo
the addition of solutions that vanish along~$\Sigma$. More precisely, using results of~\cite{Ma-hodge}, it was 
shown in~\cite{AG-BGops} that the following sequence is exact
$$
0\longrightarrow H^k(M,\Sigma) \stackrel{{\rm i}}{\longrightarrow} K^k_\infty(M)\stackrel{{\rm r}}\longrightarrow{\mathcal H}^k_L(\Sigma)\longrightarrow 0\, .
$$
Here, $H^k(M,\Sigma)$ is the relative de Rham cohomology of $M$ which was shown to be isomorphic to $\ker_{L^2}\FL$ in~\cite{Ma-hodge}. The space $K^k_\infty(M)=:\ker\FL$ is the space of smooth harmonic forms on $M$. 
Also, i and r denote, respectively, inclusion and restriction.

Displays~\nn{holL} and~\nn{holG} give holographic formul\ae \ for the conformally invariant (detour operator, gauge companion) pair.
\end{remark}

\subsection{Fundamental holographic identities} \label{funSec}

A main aim of this Section is to show, from its holographic formula above, that there exists a  factorisation of the boundary operator 
\begin{equation}\label{plag}
L_k=\cod_{_\Sigma} Q_{k+1}\,  \extd_{_\Sigma}\, :\Omega^k\Sigma\longrightarrow \Omega_k \Sigma\, ,
\end{equation} 
where 
$$
\Omega_k \Sigma:=\Gamma\ce^{k}\Sigma[2k-n]\, ,
$$
and the boundary {\it Q-operator} $Q_{k+1}$ of~\cite[Theorem 2.8]{BrGodeRham} is the form  analog of the Branson Q-curvature~\cite{tomsharp}.
Thus we have a holographic  proof of  a  theorem of~\cite{BrGodeRham}.
\begin{theorem}\label{detids}
On $(\Sigma,c)$, 
%re is, 
for each $0\leq k \leq n-2$, there is a 
%(intrinsic) conformally invariant long operator on true forms  
%$$
%L_k: \Omega^k \Sigma\longrightarrow \Omega_k \Sigma \, ,
%$$
%that, in any choice of scale, takes the form 
%\begin{equation}\label{dmd}
%L_k= \cod _{_\Sigma}Q_{k+1} \extd_{_\Sigma} ,
%\end{equation}
%for some differential operator $Q_{k+1}: \Omega^{k+1}\Sigma \to
%\Omega_{k+1}\Sigma$. Thus the $L_k$ yield 
detour complex as follows
\begin{equation}\label{detour}
\Omega^0\Sigma\stackrel{\scriptstyle\extdS}{-\!\!\!\longrightarrow}\cdots\stackrel{\extdS}{-\!\!\!\longrightarrow}\Omega^{k-1}\Sigma\stackrel{\extdS}{-\!\!\!\longrightarrow}
\Omega^k\Sigma\stackrel{\textstyle L_k}{-\!\!\!-\!\!\!\longrightarrow}\Omega_k\Sigma\stackrel{\codS}{-\!\!\!\longrightarrow}\Omega_{k-1}\Sigma
\stackrel{\codS}{-\!\!\!\longrightarrow}
\cdots\stackrel{\codS}{-\!\!\!\longrightarrow}\Omega_0\Sigma\, .
\end{equation} 
\end{theorem}

This will require some holographic identities which hold in greater generality than strictly required
for true forms. The first of these is the following.

\begin{theorem}\label{fun} Let $\A$ be a tractor $k$-form of weight $w_0$ and
  $\F$ a tractor $k$-form of weight $w_0-1$, with $h_0=3,5,7,\ldots$. Then along~$\Sigma$ the following identities hold.
  % with $w_0+\frac{d}{2}\not\in\mathbb{Z}_{\geq 1}$. Then along $\Sigma$
$$
\X y^{h_0-1}\A = - (h_0-2)^2\,  y^{h_0-3} \oD \A\, ,\quad  \quad \Xs y^{h_0-1}\A = - (h_0-2)^2 y^{h_0-3} \oDs\! \A\, ,
$$
and 
$$
y^{h_0-1}\X \F = - (h_0-2)^2 \, \oD\,  y^{h_0-3} \F\, ,\quad \quad
y^{h_0-1}\Xs\F = - (h_0-2)^2 \, \oDs y^{h_0-3} \F\, .
$$
\end{theorem}

Before proving this Theorem we establish a Lemma.

\begin{lemma}\label{superlemma}
Acting on $\ct^kM[w_0]$ for any $\ell\in{\mathbb Z}_{\geq 1}$ and weights $h_0\neq0,2,4,\ldots,2\ell$,
$$
\X y^\ell = \frac{1}{h+2\ell-2}\Big[
(h-2)\,  y^\ell \X +\ell\big((\ell-1)(h-2)-2\, xy\big)y^{\ell-2}\D
\Big]+\ell(h-2)\I y^{\ell-1}\, .
$$
\end{lemma}

\begin{proof}
The proof is by induction. The base case is the identity on the fourth line of Proposition~\ref{Iformalg}.
Thereafter only that identity, the solution generating algebra~\ref{sl2} and the fact that $[y,\D]=0$ from Proposition~\ref{Yform} is required to complete the induction.
\end{proof}

Armed with this Lemma we may prove Theorem~\ref{fun}.

\begin{proof}
The four identities are essentially equivalent. 
Since the algebra of $\Ds$, $\Xs$ and~$\Is$ satisfy the same identities as their unstarred
counterparts, it suffices to prove one identity from each row as displayed.
We will shortly prove the first of the four identities by simple application of known identities.
To obtain the fourth identity from the first, observe that the operators above acting on~$\A$ and $\F$ at quoted weights are all tangential.
Moreover, so are their separate pieces $\oD$, $\oDs$, $\X$, $\Xs$ and the given powers of~$y$.
Along the boundary the corresponding identities are in fact 
the  formal adjoints of one another. Of course, one can also verify the fourth identity algebraically.

Turning to the first identity,  note that from equation~\nn{DT}
$$
\X y^2 = (h-1)(h-2)\big(\D^T-\D-\I y\big) =(h-2)\big(h\oD-(h-1)(\D+\I y)\big) \, .
$$
Therefore one writes $\X y^{h_0-1}=\X y^{h_0-3} y^2$ and firstly applies Lemma~\ref{superlemma} at $\ell=h_0-3$.
The Theorem then follows by computing along~$\Sigma$, using only $[\D,y]=0$ and elementary algebra. 
\begin{comment}
Thus for $\F$ of weight $h_0$ 
\begin{eqnarray*}
\X y^{h_0-1}\F|_\Sigma&=&
\X y^{h_0-3} y^2 \F|_\Sigma\\[3mm]
&=&-\ \frac{h_0-2}{h_0-4}\  y^{h_0-3}\X y^2\F|_\Sigma\\
&-&(h_0-2)(h_0-3)(\D+\I y)\,  y^{h_0-3}\F|_\Sigma\\[3mm]
&=&-\ \frac{h_0-2}{h_0-4}\  y^{h_0-3}\, (h_0-4)\big[ -(h_0-3)(\D+\I y)+(h_0-2)\widehat\D^t\big]\F|_\Sigma\\
&-&(h_0-2)(h_0-3)(\D+\I y)\,  y^{h_0-3}\F|_\Sigma\\[3mm]
&=&-(h_0-2)^2\,  y^{h_0-3}\widehat\D^t \F|_\Sigma 
\end{eqnarray*}

We can also work in the opposite direction
$$
y^2\X =h(h+2) \widehat\D^t-(h+1)(h\D+(h+2)\I y)+\frac{4}{h-2}\ xy \D
$$
\end{comment}

\end{proof}

\newcommand{\bce}{\overline{\ce}}

 Let us sketch how the factorization~\nn{plag} of the long operator arises.
  Consider~$\F$ satisfying the usual identities \begin{equation}\label{usual}\Is \F = 0 =\Ds \F = \Xs \F\, ,\end{equation} so it has
  entries in the west and southern slots, but otherwise is zero. 
Then $\F=\Xs \A$ for some $k+1$ form tractor~$\A$. Thus 
$$
y^{h_0-1} \F= y^{h_0-1} \Xs \A\, .
$$ 
We assume that $h_0=3,5,7,\ldots$, and using the Theorem~\ref{fun} we see that this takes the form
$$
\oDs y^{h_0-3} \A\, ,
$$
where we have dropped a non-zero overall constant. Now using 
Theorem~\ref{fun} we see that
the North slot of this vanishes (in fact it is annihilated by $\Xs$). Then it follows easily by inspecting the formula for 
$\Ds_{_\Sigma}$ that, along~$\Sigma$, the western slot takes the form $\cod_{_\Sigma} N$, for some $N$.

On the other hand the western slot of $y^{h_0-1} \F$ is moved to the
southern slot by acting with $\X$.  Thus we may study 
$$
\X y^{h_0-1} \F\,  .
$$ 
Using once again Theorem~\ref{fun} and dropping the non-zero constant factor this becomes
$$
y^{h_0-3} \oD \F\,  .
$$
But now using that $\F$ has weight $-k$ and satisfies the system~\nn{usual},
it follows at once from the formula for $\oD$ that this factors through $\extd_{_\Sigma} F_{_\Sigma}$
where $F_{_\Sigma}=\Big[q^* \F\Big]\Big|_\Sigma$. 

To follow this through carefully and show that the $Q$ operator arises as in~\nn{plag} we 
need some preliminaries.
Of course the operator $Q_{k+1}$ is not determined
uniquely by  the formula~\nn{plag} for $L_k$. 
Thus, we use a
special choice of true scale $\tau\in\Gamma\ce_+ M[1]$
 to pin down a preferred version. 

Let us write \begin{equation}\label{Y}\Y:=\scalebox{.95}{$\frac{1}{d-2}$}\, \D \log \tau\in \Gamma\ct^1M[-1]\, ,\end{equation}
where 
$$\big(\delta_{\!\raisebox{-.5mm}{\it \tiny R}} \tau\big)\big|_\Sigma=0\, .$$
Note that this condition is easily solved given any $\tau|_\Sigma$.
 Then $$\{\Xs,\Y \}=1\, ,$$
 and, along~$\Sigma$, $$ \{\Is,\Y\}=0\, .$$
Now note  the following Lemma  for commutators of the tangential double D-operator (see Proposition~\ref{hatstotilde} and Remark~\ref{tangDD}) with $\Y$.
\begin{lemma}\label{trumpet}
Acting on $\Gamma \cT^kM[w]$ we have along~$\Sigma$
$$
 \Big[\X \Dt^T , \Y\Big]=  0 =\Big[ \Xs \Dts{}^T , \Ys \Big]\, .
$$
\end{lemma}
\begin{proof}
The proof amounts to recalling that $\X \Dt^T=\X(\Dt -\I \{\Is,\Dt\})$. Along $\Sigma$ the right hand side of this is $\X \Dt - \I\Is \X \Dt +\I
\X\Dt\Is $. But the double D-operator $\X\Dt$ is Leibnizian so obviously commutes with $\Y$. Moreover, $\Y$ was chosen such that $\{\Is,\Y\}=0$ along $\Sigma$. Note that $\I \X\Dt \{\Is,\Y\}$ vanishes along $\Sigma$ since an $x$ produced on the right by the anticommutator produces terms containing either a second  $\I$ or $\X$
when pushed through to the left. 
\end{proof}
We shall also need the following which is obtained by an easy computation using the tangential double D-operator version of Theorem~\ref{dt} given in Remark~\ref{tangDD}.
\begin{lemma}\label{boondoggled}
Acting on $A\in\Omega^kM $ an extension of $A_{_\Sigma}\in\Omega^k\Sigma$  we have
$$
 \Big[\X \Dt^T q A\Big]\Big|_\Sigma =  \X_{_\Sigma} q_{_ \Sigma} \extd_{_\Sigma} A_{_\Sigma} \, .
$$
\end{lemma}
%\end{document}
\begin{remark}
There is an obvious adjoint version of the above Lemma, 
$$
q^* \Xs_{_\Sigma} \Dts_{_\Sigma} = -\cod_{_{\Sigma}}q^* \Xs_{_\Sigma}\, .
$$
\end{remark}

We are now ready to give a holographic formula for the Q-operator in preparation for our main Theorem.
\begin{definition}\label{Qdef} Let $n$ be even, fix a true scale $\tau\in \Gamma\ce_+M[1]$ and
  denote by $g\in c$ the corresponding metric ({\it i.e.}, $g=\tau^{-2}\bg$).
 For each $k\in \{0,\ldots ,\frac n2 \}$. This determines the boundary differential {\em Q-operator}
$$
Q^{g_{_\Sigma}}_{k}: \Omega^k\Sigma\longrightarrow \Omega_k\Sigma \, ,
$$
determined by restriction of a bulk holographic formula,  called the {\it holographic Q-operator} 
$$
Q^g_k:= q^* \Ys\Xs y^{n-2k} \X\Y  q\ \, ,
$$
which acts on bulk true forms.
\end{definition}

This enables a holographic proof of the following Theorem.

\begin{theorem}
  On $(\Sigma,c_{_\Sigma})$ of even dimension $n\geq 4$ and given any choice of $g_{_\Sigma}\in c_{_\Sigma}$, the conformally
  invariant detour operators
$$
L_k: \Omega^k\Sigma\longrightarrow \Omega_k \Sigma \, ,\quad 0\leq k \leq \frac{n}{2}-1  \, ,
$$
can be expressed as the composition  
\begin{equation}\label{dmd}
L_k= \gamma_k\, \cod_{_\Sigma} Q^{g_{_\Sigma}}_{k+1} \extd_{_\Sigma} \, ,
\end{equation}
where 
$Q^{g_{_\Sigma}}_{k+1}$ is the Q-operator from Definition~\ref{Qdef} and
$$
\gamma_k=-(n-2k)(n-2k+2)(n-2k-1)^2\, .
$$
\end{theorem}
\begin{proof}
From Equation~\nn{holL}, $L_k$ is the composition of operators
$$
L_k= q^* y^{h_0-1} {\oDs} \X q \, ,
$$
evaluated along~$\Sigma$ with $h_0:=d-2k$.

This operator is tangential and throughout we will calculate along~$\Sigma$ without further comment.
Now using that $\{\X,\Y^* \}=1 $ and that the composition $q^* \X\Xs=0$ (recall from Theorem~\ref{fun} that   $y^{h_0-1} \oDs \X  q$ can be written  as a composition of $\Xs$ from the left)  we come to 
$$
L_k= q^*\Ys \X y^{h_0-1} {\oDs} \X  q \, .
$$
Now using Theorem~\ref{fun} we have 
$$
L_k= -(h_0-2)^2 q^*\Ys  y^{h_0-3} \oD\, \oDs \X  q  \, .
$$
Since along~$\Sigma$ we have $\oD\  \oDs=-\oDs\, \oD $ (because $\{\D_{_\Sigma},\Ds_{_\Sigma}\}=0$ and $y^{h_0-3}$ acts tangentially here), thus 
\begin{eqnarray*}
L_k&=&\ \  (h_0-2)^2 \, q^*\Ys  y^{h_0-3} \oDs\oD \X  q  \\[2mm]
&=& -(h_0-2)^2\, (h_0+1)\,  q^* \Ys y^{h_0-3} \oDs \X \Dt^T\,  q\, .
\end{eqnarray*}
In the last step we used that along~$\Sigma$ one has 
\begin{equation}
\label{slaphappy}
\oD \X = \D_{_\Sigma} \X_{_\Sigma}=-\frac{h_{_\Sigma}+2}{h_{_\Sigma}-2}\, \X_{_\Sigma} \D_{_\Sigma} =
-\frac{h+1}{h-3}\, \X \oD=-(h+1)\X\Dt^T\, .
\end{equation}
%\begin{eqnarray}
%\D^T \X  &=&\ \ \frac{h_{_\Sigma}+1}{h_{_\Sigma}} \, \D_{_\Sigma} \X_{_\Sigma} \nonumber\\
%&=&\ \ \frac{h_{_\Sigma}+1}{h_{_\Sigma}} \, \wD_{_\Sigma} h_{_\Sigma} \X_{_\Sigma}\nonumber \\
%&=&
%-\frac{(h_{_\Sigma}+1)(h_{_\Sigma}+2)}{h_{_\Sigma}} \, \X_{_\Sigma} \Dt_{_\Sigma}\nonumber\\
%&=&
%-\frac{h(h+1)}{h-1} \, \X\Dt^T\, .\label{slaphappy}
%\end{eqnarray}
Note that the singularity at $h_0=3$ in the fourth expression above is a removable one as evidenced by the last equality.

We may now employ Lemma~\ref{boondoggled} to generate a boundary exterior derivative on the right (along~$\Sigma$, $\X$ and $\X_{_\Sigma}$
agree):
$$
L_k=-(h_0-2)^2\, (h_0+1)\,  q^* \Ys y^{h_0-3} \oDs \X q \, \extd_{_\Sigma}\, .
$$
Using the second relation of Theorem~\ref{fun} we then have 
$$
L_k=  (h_0+1)\,  q^* \Ys \Xs y^{h_0-1}  \X q \, \extd_{_\Sigma}\, .
$$
Now we use  that $1=\{\Xs,\Y\}$ and $\X\Y\Xs q =0$ to obtain
$$
L_k=-(h_0+1)\,  q^* \Ys \Xs y^{h_0-1}  \Xs\X \Y q \, \extd_{_\Sigma}\, .
$$
Now the fourth relation in Theorem~\ref{fun} gives
$$
L_k=(h_0-2)^2\, (h_0+1)\,  q^* \Ys \Xs \oDs y^{h_0-3}  \X \Y q \, \extd_{_\Sigma}\, .
$$
Now we use the second relation in Lemma~\ref{trumpet} to move $\Ys$ to the right (being careful to replace $\oDs=(h-1)\, \Dt^{\star T}$).
$$
L_k=-(h_0-2)^2\, (h_0+1)(h_0-1)\,  q^*  \Xs {\Dt}^{\star T} \Ys\, y^{h_0-3}  \X \Y q \, \extd_{_\Sigma}\, .
$$
The final step is to employ the adjoint version of Lemma~\ref{boondoggled} to extract a boundary codifferential on the left.
$$
L_k=- (h_0-2)^2(h_0+1)(h_0-1)\,  \cod_{_\Sigma} q^*  \Ys\Xs \, y^{h_0-3}  \X \Y q \, \extd_{_\Sigma}\, .
$$
\end{proof}
\begin{remark}
The way the above holographic proof employs the fundamental Theorem~\ref{fun} mirrors the original proof of~\cite[Theorem 2.8]{BrGodeRham} based on the ambient Fefferman--Graham metric.
Similarly, consideration of~$Q_{k}\extd_{_\Sigma}$, and holographic arguments parallel to  those in the proof above,   recovers the conformal transformation 
formula of~\cite[Equation (1)]{BrGodeRham} for $Q_k$ as an operator on closed $k$-forms.
On the other hand, using again parallel arguments to above shows that  $\cod_{_\Sigma}Q_k$ is proportional to $G_k$  and recovers the conformal transformation law for $G_k$ found in~\cite{BrGodeRham}.
\end{remark}

\begin{remark}
We observed in Remark~\ref{injectively} that ${\mathcal H}^k_L$ parametrises the smooth, Dirichlet solution  space 
to the gauge fixed Maxwell system  (which amounts in the  preferred interior scale to equations $\FL A = 0 = \cod A$) up to the addition of second solutions. As a special case of this problem, one can begin with closed Dirichlet boundary data $A_{_\Sigma}$. It follows that solutions are then parametrised by $A_{_\Sigma}$
such that
$$A_{_\Sigma}\in \ker( \extd_{_\Sigma}, G_k)=: {\mathcal H}^k(\Sigma)\, .$$
This is exactly the conformal harmonic space for~$\Sigma$ defined in~\cite{BrGodeRham}.
It is interesting to understand what this special boundary data means in the bulk. This is answered in~\cite[Theorem 1.3]{AG-BGops}: Essentially, this is captured by the space $\ker(\extd,\cod)=:Z^k(M)$, more precisely it is there proven that there is
an exact sequence 
$$
0\longrightarrow H^k(M,\Sigma)\longrightarrow Z^k(M)\longrightarrow
{\mathcal H}^k(\Sigma)\longrightarrow H^{k+1}(M,\Sigma)\, ,
$$
where $k<\frac n2$ and $H^k(M,\Sigma)\cong \ker_{L^2}\FL$ is the relative cohomology of $M$.
\end{remark}

\appendix

\section{The ambient manifold}
\label{AMBIENT}

A detailed  analysis of tractor forms using an ambient manifold approach
has been given in~\cite{BrGodeRham}. We first sketch the ingredients needed
to provide expedient proofs of Lemma~\ref{DXid}, Proposition~\ref{Thomtang} and Theorem~\ref{dt}.

A conformal structure is equivalent to the ray bundle $\pi:\cG\to M$ of
conformally related metrics.  Let us use $\rho $ to denote the ${\Bbb
  R}_+$ action on $ \cG$ given by $\rho(t) (x,g_x)=(x,t^2g_x)$.  An
{\em ambient manifold\/}  is a smooth $(d+2)$-manifold $\aM$ endowed
with a free $\Bbb R_+$--action~$\rho$ and an $\Bbb R_+$--equivariant
embedding $i:\cG\to\aM$.  We write $\aX\in\Gamma(T \aM)$ for the
fundamental field generating the $\Bbb R_+$--action.  That is, for
$f\in C^\infty(\aM)$ and $ u\in \aM$, we have $\aX
f(u)=(d/ds)f(\rho(e^s)u)|_{s=0}$.  For an ambient manifold $\aM$, an
{\em ambient metric\/} is a pseudo--Riemannian metric $\h$ of
signature $(d+1,1)$ on $\aM$ satisfying the conditions: (i) ${\Cal
L}_{\sX}\h=2\h$, where $\Cal L_{\sX}$ denotes the Lie derivative by
$\aX$; (ii) for $u=(x,g_x)\in \cG$ and $\xi,\eta\in T_u\cG$, we have
$\h(i_*\xi,i_*\eta)=g_x(\pi_*\xi,\pi_*\eta)$. In~\cite{FGast,FGrnew} Fefferman and Graham considered formally the Gursat
problem of obtaining $\Ric(\h)=0$. They proved that, for the case of
$d=2$ and $d\geq 3$ odd, this may be achieved to all orders, while for
$d\geq 4$ even, the problem is obstructed at finite order by a natural
2-tensor conformal invariant (this is the Bach tensor if $d=4$, and is
called the Fefferman--Graham obstruction tensor in higher even
dimensions); for $d$ even one may obtain $\Ric(\h)=0$ up to the
addition of terms vanishing to order $d/2-1$. See~\cite{FGrnew} for
statements concerning uniqueness. For extracting results via
tractors we do not need this, as discussed in {\it e.g.}~\cite{CapGoamb,GoPetCMP}. (In fact, to obtain a correspondence with the normal tractor connection which is all that we require below, it suffices that the tangential components
of Ricci vanish along~$\cG$.)
We shall henceforth call any (approximately
or otherwise) Ricci-flat  metric on $\aM$  a {\em Fefferman--Graham
  metric}.   In the subsequent
discussion of ambient metrics all results can be assumed to hold
formally to all orders.

In the following discussion we  use bold symbols or tilded
symbols for the objects on~$\aM$. For example
$\boldsymbol{\nabla}$ denotes the Levi-Civita connection on
$\aM$.  
Familiarity with the 
treatment of the Fefferman--Graham metric, as in {\it e.g.}~\cite{CapGoamb,GoPetobstrn} 
or~\cite{BrGodeRham}, will be assumed.
In particular, we
shall use that suitably homogeneous tensor fields of
$\aM|_\cG$ correspond to tractor fields. This correspondence is
compatible with the Levi-Civita connection in that each weight zero
tractor field $F$ on $M$ is identified with (the restriction to $\cG$ of)
a homogenous tensor field $\aF$ on $\aM$
with the property that it is parallel in the vertical direction, that
is $\aNd_{\sX} \aF=:\aX^A \aNd_A \aF= 0$ along $\cG$.  The metric~$\h$
and its Levi-Civita connection $\aNd$ on~$\aM$ determine a metric and
connection on tractor bundles, and by~\cite[Theorem 2.5]{CapGoamb}
this agrees with the normal tractor metric and connection. 
We use abstract indices in an obvious way on
$\aM$ and these are lowered and raised using $\h_{AB}$ and its inverse
$\h^{AB}$.

We shall say $\aF$ is homogeneous of {\em weight} $w_0$
if $\aNd_{\sX} \aF=w_0 \aF$, and this corresponds to a tractor field
$F$ of weight $w_0$.  We shall always take such fields to be extended
off~$\cG$ smoothly and also such that $\aNd_{\sX} \aF= w_0 \aF$ on
$\aM$. The operator $\aNd_{\sX}$ gives an ambient realisation of the weight
operator, as applied to tensor fields of
well defined weight along~$\cG$. 

In this picture the operator $ \aD^A= \big(d+2\, w_0 -2\big) \aNd^A -
\aX^A \aNd\,^2 $ on tensors homogeneous of weight $w_0$ corresponds to
the tractor D-operator as applied to tractors of weight $w_0$.  Thus
we equivalently view this as a restriction of
$$
\aD^A= \aNd^A(d+2\aNd_{\sX}-2) +\aX^A \aNd\,^2.
$$
Here $\aNd\,^2$ is the ambient Bochner Laplacian. The above operator acts tangentially along the submanifold $\cG$ in $\aM$,
~\cite{BrGodeRham} and~\cite{GoPetCMP}.

This technology enables simple computations on $\aM$ of often complicated tractor quantities on~$M$.
\newcommand{\bh}{{h}}
For example the proof of Lemma~\ref{DXid} is:
%: $\bh:=d+2\aNd_\sX$.
\begin{eqnarray*}
&&\bh\, \aX^A\aD^B-(\bh-2)\, \aD^B \aX^A - 2\,  \aX^B\aD^A +\bh(\bh-2)\, \h^{AB}\\[2mm]
&=&\bh \, \aX^A \big(\bh\aNd^B-\aX^B \aNd\,^2\big) \ -\ (\bh-2)\, \big(\bh\aNd^B -\aX^B \aNd\,^2\big) \aX^A\\[1mm]
&-&2\, \aX^B\big(\bh\aNd^A-\aX^A \aNd\,^2\big) \ + \ \bh(\bh-2)\, \h^{AB}\\[2mm]
&=&\bh (\bh-2)\, \aX^A \aNd^B - \bh\,  \aX^A \aX^B \aNd\,^2 \\[1mm]
 &-& \bh (\bh-2)\, \big(\aX^A\aNd^B+\h^{AB}\big) +(\bh-2)\, \aX^B \big(\aX^A \aNd\,^2+2\, \aNd^A\big)\\[1mm]
&-&2(\bh-2)\, \aX^B \aNd^A+2\aX^A\aX^B \aNd\,^2 \, + \, \bh(\bh-2)\, \h^{AB}\ = \ 0\, .   \end{eqnarray*}
\vspace{-13mm}
\begin{flushright} $\square$\end{flushright}
\vspace{4mm}

We will employ the same notations as used on $M$ for the natural operators
on sections of $\Lambda^\bullet \aM$. Namely
$\extd$ for the ambient exterior derivative, $\cod$ the ambient codifferential and $\FL=\{\extd,\cod\}$ for the
ambient form Laplacian. No confusion should arise from this recycling of notation.
As shown in~\cite{BrGodeRham}, the ambient operator corresponding to the exterior Thomas D-operator is $\bm{\mathcal D}=(d+2w_0-2) \extd - \varepsilon(\aX) \FL$
acting on ambient differential forms homogeneous of weight~$w_0$. This, and its interior analog can both be equivalently viewed as restrictions 
of the operators
$$
\bm{\mathcal D}=h \extd  - \varepsilon(\aX) \FL\, ,\qquad \bm{\mathcal D}^* = h \cod - \iota(\aX) \FL\, ,
$$
where 
$$
h:=d+2\aNd_\sX\, .
$$

Given a defining scale $\si\in \Gamma\ce M[1]$ on $M$, we shall write
$\asi$ for the corresponding homogeneous weight~1 function on $\cG$ with
some homogeneous extension to $\aM$.  Then we introduce 
$$
\aI^A:=\frac{1}{d}\, \aD^A \asi
$$
and, again, we can define the interior and exterior operators
 $$
 \bm{\mathcal I}:=\varepsilon(\aI)~,\qquad\, \bm{\mathcal I}^*:=\iota(\aI)~,
 $$
 as well as
 $$
 \bm{\mathcal X}:=\varepsilon({\aX})~,\qquad \bm{\mathcal X}^*:=\iota(\aX)~.
 $$
 
 We now restrict our attention to the ambient analog of a Poincar\'e--Einstein structure with defining scale~$\sigma$. For simplicity, we begin by taking the dimension~$d$ odd and so may assume~\cite{Goal}
 that~$\aI=\aNd \, \tilde \sigma$ is parallel with respect to the ambient Levi--Civita connection and has unit length~$\aI^2=1$.
Then a restriction of the
differential operator
\begin{equation*}\label{IdotDamb}
-\bm{y}:=\{\iaI, \bm{\mathcal D}\}= \{\eaI, \bm{\mathcal D}^*\} = h \aNd_{\sI} - \tilde \sigma \al
\end{equation*}
(on $\aM$)  lifts the operator 
 $I_A \slashed D{}^A$ on $M$, enabling calculations on $\aM$. 
 Note now that the combination $\bm{\mathcal I}^*\bm{\mathcal I}$ projects onto boundary objects, thus we call 
 $$ \extd^T:=\bm{\mathcal I}^*\bm{\mathcal I}\, \extd\qquad \mbox{ and }\qquad \cod^T:=\cod \, \bm{\mathcal I}^*\bm{\mathcal I}\, .$$
 
 We are ready now to prove Theorem~\ref{dt}; along $\Sigma$ and acting on $\ker  \bm{\mathcal I}^*$ we find
\begin{eqnarray*}
 \bm{\mathcal D}^T&=& \bm{\mathcal D}+ \bm{\mathcal I} \bm{y} +\frac1{(h-1)(h-2)}\eaX \bm{y}^2 \\[1mm]
&=&h\, (\extd^T +\bm{\mathcal I}\, \bm{\mathcal I}^*\extd)- \eaX \big(\cod (\bm{\mathcal I}\, \bm{\mathcal I}^* + \bm{\mathcal I}^*\bm{\mathcal I})\extd+(\bm{\mathcal I}\, \bm{\mathcal I}^*+\bm{\mathcal I}^*\bm{\mathcal I})\, \extd\, \cod\  \bm{\mathcal I}^*\bm{\mathcal I}\big) \\
&-& \bm{\mathcal I}\, h\aNd_{\sI}+\frac1{(h-1)(h-2)} \eaX\  h \aNd_{\sI} \, \big(h\aNd_{\sI}-\tilde \sigma\FL\big)\\[1mm]
&=&h\,  \extd^T
-\frac{h}{h-1} \eaX
\big(
\FL^T 
+\cod\, \bm{\mathcal I}\,  \aNd_{\sI}
+\bm{\mathcal I} \, \aNd_{\sI}\, \cod - \aNd_{\sI}^{\: 2}\big)
\\[1mm]
&=&\frac{h}{h-1}\big((h-1)\,  \extd^T - \eaX \FL^T\big)=\frac{h}{h-1}  \bm{\mathcal D}_\Sigma\, .
\end{eqnarray*}
In the second line we inserted $1=\{\bm{\mathcal I},\bm{\mathcal I}^*\}$ in order to produce the operators  $\extd^T, \cod^T$ and $\FL^T:=\{\extd^T, \cod^T\}$ appearing in the third line. 
Also, to obtain the third line, we used, for example,   $\{\iaI,\extd\}=\aNd_\sI$. To obtain fourth line, we used $\{\cod,\eaI\}=\aNd_\sI$ as well as $[\aNd_I,\cod]=0$ to cancel all
terms with a $\aNd_\sI$ in the third line.
In the last step we evaluated our result along the boundary using that $\extd^T|_{\Sigma}=\extd_\Sigma$ and $\FL^T|_{\Sigma}=\FL_\Sigma$ there.

Note that the special case of Paneitz weight where $ \bm{\mathcal D}^T$ acts on ambient forms of homogeneity $w_0=2-\frac d2$ is not covered by the above 
computation. However, it is easy to see how that case is proven from the above display: Focussing on the last term on the third
line, we see that the potentially singular factor $h-2$ in the denominator cancels because $\eaX h = (h-2) \eaX$. The definition of $ \bm{\mathcal D}^T$
exactly removes the last singular term on that line.

Finally, the last special case of boundary Yamabe weight can be handled by multiplying both sides of the above computation from the left by $h-1$
and then specializing to homogeneity  $w_0=1-\frac n2$.
\vspace{-1mm}
\begin{flushright} $\square$\end{flushright}
\vspace{2mm}

The proof for the analogous formula for the Thomas D-operator acting on arbitrary tractors required for 
Proposition~\ref{Thomtang} is very similar. Using $\aI\cdot \aD=h \aNd_{\sI} - \tilde \sigma \aNd{}^{\, 2}$ and $$\aNd^{\, T}_A:=\aNd_A -\aI_A \aNd_\sI\, ,\qquad (\aNd^{\, T})^2= \aNd^{\, 2} - \aNd_\sI^{\, 2}\, ,$$ we have, calculating along~$\Sigma$
\begin{eqnarray*}
\aD_A^{\, T}\!&=& \aD_A - \aI_A \aI\cdot \aD + \frac1{(h-1)(h-2)} \aX_A \big(\aI\cdot \aD\big)^2\\[2mm]
&=&h(\aNd^{\, T}_A+\aI_A \aNd_\sI) - \aX_A \aNd^{\, 2} - \aI_A h \aNd_\sI
+\frac1{(h-1)(h-2)} \aX_A h \aNd_\sI \big(h \aNd_{\sI} - \tilde \sigma \aNd{}^{\, 2}\big)\\[2mm]
&=&h\aNd^{\, T}_A-\frac{h}{h-1}\aX_A(\aNd^{\, T})^2 \: = \: \frac{h}{h-1}\, \aD^{\Sigma}_A\, . 
\end{eqnarray*}
\vspace{-11mm}
\begin{flushright} $\square$\end{flushright}
\vspace{12mm}
For $d\geq 4$ even, the harmonic extension problem for $\tilde \sigma$ is potentially obstructed. 
However, we may still obtain that  $0=\al \tilde \sigma=\al^2 \tilde \sigma= \cdots = \al^{d/2} \tilde \sigma$ along~$\cG$
(see~\cite{GJMS,Goal}). This suffices for the above proofs of Theorem~\ref{dt} and Proposition~\ref{Thomtang}
to apply also in this dimension parity.

\newpage

\section{List of common symbols}
\vspace{3mm}
\begin{center}
\scalebox{.8}{
\begin{tabular}{lr}
$d:=n+1$ & $\dim M:=\dim\partial M+1$\\[1mm]
$\Lambda^\bullet M$ & Exterior bundle on $M$\\[1mm]
$\Omega^\bullet M$ & Sections of $\Lambda^\bullet M$ \\[1mm]
$\extd$ & Exterior derivative\\[1mm]
$\cod$ & Codifferential\\[1mm]
$*$ & Hodge star\\[1mm]
$\FL$ & Form Laplacian\\[1mm]
$\degree$ & Form degree\\[1mm] 
$\P$ & \hyperlink{Schouten}{Schouten endomorphism}\\[1mm]
$\varepsilon$ & Exterior multiplication\\[1mm]
$\iota$ & Interior product \\[1mm]
$\wepsilon$ & \hyperlink{wiota}{Holographic exterior normal}\\[1mm]
$\wiota$ & \hyperlink{wiota}{Holographic interior normal} \\[1mm]
$c$ & \hyperlink{conformal structure}{Conformal structure} \\[1mm]
$\bg$ & \hyperlink{conformal metric}{Conformal metric}\\[1mm]
$\ce M[\, . \, ]$ & \hyperlink{exterior density}{Conformal density bundle} on $M\!\!$\\[1mm]
$\ce^\bullet M[\, . \, ]\ \ $ & \hyperlink{exterior density}{Exterior density bundle} on $M$\\[1mm]
$\Gamma{\mathcal B}$ & Sections of bundle ${\mathcal B}$\\[1mm] 
$\ct M$ & \hyperlink{standard tractor bundle}{Standard tractor bundle} on $M$\\[1mm]
$\ct^\Phi M[\, . \, ]$ & \hyperlink{tractor bundle}{Tractor bundle}
\end{tabular}\hspace{.3cm}
\begin{tabular}{lr}
$h(.\, ,.),h_{AB}$ & \hyperlink{tractor metric}{Tractor metric}\\[1mm]
$\w$ & \hyperlink{weight operator}{Tractor weight operator}\\[1mm]
\hyperlink{h}{$h$} & $d+2\w$\\[1mm]
$\ct^\bullet M$ & \hyperlink{exterior tractor bundle}{Exterior tractor bundle} on $M$\\[1mm]
$D^A$ & \hyperlink{Thomas D}{Thomas D-operator}\\[1mm]
$\hash$ & \hyperlink{tensor endomorphism}{Tensorial endomorphism}\\[1mm]
$\D,\varepsilon(\slashed D)$ & \hyperlink{exterior D}{Exterior tractor D-operator}\\[1mm]
$\Ds,\iota(\slashed D)$ &  \hyperlink{exterior D}{Interior tractor D-operator}\\[1mm]
$\star$ & \hyperlink{tractor star}{Tractor Hodge star}\\[1mm]
$X^A$ & \hyperlink{canonical tractor}{Canonical tractor}\\[1mm]
$\X,\varepsilon(X)$ & \hyperlink{exterior canonical tractor}{Exterior canonical tractor}\\[1mm]
$\Xs,\iota(X)$ & \hyperlink{interior canonical tractor}{Interior canonical tractor}\\[1mm]
$\N$ & \hyperlink{tractor degree}{Tractor form degree}\\[1mm]
$q_N,q_E,q_S,q_W\ \ $ & \hyperlink{insertions}{Insertion operators}\\[1mm]
$\sigma,x$ & \hyperlink{scale}{Scale}\\[1mm]
$I^A$ & Scale tractor\\[1mm]
$\I,\varepsilon(I)$ & Exterior scale tractor\\[1mm]
$\Is,\iota(I)$ & Interior scale tractor\\[1mm]
$y$ & \hyperlink{extension operator}{Extension operator}
\end{tabular}}
\end{center}

\newpage

%number,names
\newcommand{\msn}[2]{\href{http://www.ams.org/mathscinet-getitem?mr=#1}{#2}}
%number
\newcommand{\hepth}[1]{\href{http://arxiv.org/abs/hep-th/#1}{arXiv:hep-th/#1}}
%number
\newcommand{\maths}[1]{\href{http://arxiv.org/abs/math/#1}{arXiv:math/#1}}
%number
\newcommand{\mathph}[1]{\href{http://lanl.arxiv.org/abs/math-ph/#1}{arXiv:math-ph/#1}}
\newcommand{\arxiv}[1]{\href{http://lanl.arxiv.org/abs/#1}{arXiv:#1}}

\end{document}